\documentclass[12pt]{amsart}
\usepackage[cmtip,all]{xy}
\usepackage{graphicx}
\usepackage{amssymb}
\usepackage{mathrsfs}
\usepackage{amsfonts}
\usepackage{amsmath}
\usepackage[pagebackref]{hyperref}%
\newtheorem{theorem}{Theorem}[section]
\newtheorem{lemma}[theorem]{Lemma}
\newtheorem{proposition}[theorem]{Proposition}
\newtheorem{corollary}[theorem]{Corollary}

\theoremstyle{definition}

 \theoremstyle{remark}
\newtheorem{remark}[theorem]{Remark}

 \numberwithin{equation}{section}

\begin{document}

\title[Fractional  GJMS
 operators and  sharp fractional  HSM inequality]{Explicit Formulas of Fractional  GJMS
 operators  on hyperbolic spaces and sharp fractional Poincar\'e-Sobolev and  Hardy-Sobolev-Maz'ya inequalities }

\author{Guozhen Lu}
\address{Department of Mathematics,   University of Connecticut, Storrs, CT 06290, USA}
\email{guozhen.lu@uconn.edu}

\author{Qiaohua  Yang}
\address{School of Mathematics and Statistics, Wuhan University, Wuhan, 430072, People's Republic of China}

\email{qhyang.math@whu.edu.cn}








\begin{abstract}
Using the  scattering theory on the hyperbolic space $\mathbb{H}^n$, we give the explicit formulas of  the fractional GJMS operators $P_{\gamma}$ for all $\gamma\in(0,\frac{n}{2})\setminus\mathbb{N}$ on   $\mathbb{H}^n$.
  These operators $P_{\gamma}$ for $\gamma\in(0,\frac{n}{2})\setminus\mathbb{N}$  are neither conformal to the fractional Laplacians on the upper half space $\mathbb{R}^n_{+}$ nor on the  unit ball $\mathbb{B}^n$ in $\mathbb{R}^{n}$ though  $P_{\gamma}$ are  conformal to  $(-\Delta)^{\gamma}$ via half space model and ball model of hyperbolic spaces when $\gamma\in\mathbb{N}$. To circumvent this, we introduce another family of fractional operators $\tilde{P}_{\gamma}$ on $\mathbb{H}^n$ which are conformal to the fractional Laplacians on $\mathbb{R}^n_{+}$ and the unit ball $\mathbb{B}^n$ via half space model and ball model of hyperbolic spaces. (Theorem \ref{th1.6}.) It is worthwhile to note that $\tilde{P}_{\gamma}\not =P_{\gamma}$ unless $\gamma$ is an integer. (see \eqref{relation} and Corollary \ref{co4.3}.)
 We establish the
fractional  Poincar\'e-Sobolev  inequalities associated with both $P_{\gamma}$ and $\tilde{P}_{\gamma}$ on   $\mathbb{H}^n$. (see Theorems \ref{th1.3} and \ref{th1.7}.) In particular, when $n\geq 3$ and $\frac{n-1}{2}\leq \gamma<\frac{n}{2}$, we
prove that the sharp constants in the
$\gamma$-th order  of Poincar\'e-Sobolev  inequalities on the hyperbolic space associated with  $P_{\gamma}$ and $\tilde{P}_{\gamma}$
 coincide with the best $\gamma$-th order Sobolev constant in the $n$-dimensional Euclidean space $\mathbb{R}^n$. (See Theorems \ref{th1.4} and \ref{th1.8}.)
We  also establish fractional Hardy-Sobolev-Maz'ya inequality on  half spaces $\mathbb{R}^{n}_+$ and  unit ball $\mathbb{B}^n$ and prove
that the sharp constants in the $\gamma$-th order Hardy-Sobolev-Maz'ya inequalities  on half space  $\mathbb{R}^{n}_+$ and  unit ball $\mathbb{B}^n$ are the same as  the best $\gamma$-th order Sobolev constants  in $\mathbb{R}^n$ when $n\geq 3$ and $\frac{n-1}{2}\leq \gamma<\frac{n}{2}$. (Theorems \ref{th1.7} and \ref{th1.8}.) A sharp Sobolev inequality with best constant  for the operator $\frac{|\Gamma(\nu+\gamma+i\sqrt{-\Delta_{\mathbb{H}}-\frac{(n-1)^{2}}{4}})|^{2}}
{|\Gamma(\nu+i\sqrt{-\Delta_{\mathbb{H}}-\frac{(n-1)^{2}}{4}})|^{2}}u$ is also proved which is of its independent interest. (Theorem \ref{th3.10}.)
Finally, in the borderline case, namely  $\gamma=\frac{n}{2}$, we establish the fractional Hardy-Adams inequalities associated with $P_{\frac{n}{2}}$, $\tilde{P}_{\frac{n}{2}}$ and $(-\Delta)^{\frac{n}{2}}$. (see Theorems \ref{1.5} and \ref{1.8}.)
 Our methods crucially rely on the Helgason-Fourier analysis on hyperbolic spaces and delicate analysis of special functions.

\end{abstract}

\thanks{The first author was partly supported Simons Collaboration grant and Simons Fellowship from the Simons Foundation. The second author  was  partly supported by   the National Natural
Science Foundation of China (No.12071353).}
\maketitle


\section{Introduction}
Let $P_{1}=-\Delta_{\mathbb{H}}-\frac{n(n-2)}{4}$ be the conformal Laplacian on hyperbolic space $\mathbb{H}^{n}$. The
Poincar\'e-Sobolev inequalities on hyperbolic space $\mathbb{H}^{n}$ reads (see \cite{an2,m1})
\begin{equation}\label{1.3}
\int_{\mathbb{H}^{n}}u P_{1}udV-\frac{1}{4}\int_{\mathbb{H}^{n}}u^{2}dV\geq C\left(\int_{\mathbb{B}^{n}}|u|^{p}dV\right)^{\frac{2}{p}},\;\;u\in C^{\infty}_{0}(\mathbb{H}^{n}),  \, n\geq3,
\end{equation}
where $\frac{1}{4}$ is the bottom of the spectrum of $P_1$.  By using the half space model $(\mathbb{H}^{n}, g_{\mathbb{H}})$ of hyperbolic space (see Section 3), one can see that (\ref{1.3}) is equivalent to the following  Hardy-Sobolev-Maz'ya inequality on $\mathbb{R}^{n}_{+}=\{(x_{1},x_{2},\cdots,x_{n})\in\mathbb{R}^{n}: x_{1}>0\}$ (see \cite{maz}, Section 2.1.6)
\begin{equation}\label{1.1}
\int_{\mathbb{R}^{n}_{+}}|\nabla u|^{2}dx-\frac{1}{4}\int_{\mathbb{R}^{n}_{+}}\frac{u^{2}}{x^{2}_{1}}dx\geq C\left(\int_{\mathbb{R}^{n}_{+}}
x^{\alpha}_{1}|u|^{p}dx\right)^{\frac{2}{p}},\;\;\; u\in C^{\infty}_{0}(\mathbb{R}^{n}_{+}),
\end{equation}
where $n\geq 3$,  $2<p\leq\frac{2n}{n-2}$ and  $\alpha=\frac{(n-2)p}{2}-n$.
Higher order Poincar\'e-Sobolev inequalities  have been established recently by  the authors and we state them as follows:
\begin{theorem}[\cite{LuYang3}]\label{th1.1}
Let $2\leq k<\frac{n}{2}$ and $2<p\leq\frac{2n}{n-2k}$. There exists a positive constant $C$ such that for each $u\in C^{\infty}_{0}(\mathbb{H}^{n})$,
\begin{equation}\label{1.4}
\int_{\mathbb{H}^{n}}uP_{k}udV- \left(\prod^{k}_{i=1}\frac{(2i-1)^{2}}{4}\right)\int_{\mathbb{H}^{n}}u^{2}dV\geq C\left(\int_{\mathbb{H}^{n}}|u|^{p}dV\right)^{\frac{2}{p}},
\end{equation}
where $P_{k}=P_{1}(P_{1}+2)\cdot\cdots\cdot(P_{1}+k(k-1))$, $ k\in\mathbb{N}$, is the GJMS  operator on $\mathbb{H}^{n}$.
\end{theorem}

Inequality \eqref{1.4} has been found useful in the study of the existence and symmetry of solutions to a class of higher order Brezis-Nirenberg problems  on hyperbolic spaces by J. Li and the authors \cite{LiLuYang-BN}.

By the conformal covariant property of  GJMS operators for the conformal change of metrics, we have
\begin{align}\label{1.5}
x_{1}^{k+\frac{n}{2}}\circ(-\Delta)^{k}\circ x_{1}^{k-\frac{n}{2}}=P_{k}\;\; \textrm{in}\;(\mathbb{H}^{n}, g_{\mathbb{H}}).
\end{align}
By using (\ref{1.5}), one sees  that (\ref{1.4}) is  equivalent to the following Hardy-Sobolev-Maz'ya inequalities  on $\mathbb{R}^{n}_{+}$:
\begin{equation}\label{1.6}
\int_{\mathbb{R}^{n}_{+}}u (-\Delta)^{k}udx- \left(\prod^{k}_{i=1}\frac{(2i-1)^{2}}{4}\right)\int_{\mathbb{R}^{n}_{+}}\frac{u^{2}}{x^{2k}_{1}}dx\geq C\left(\int_{\mathbb{R}^{n}_{+}}x^{\alpha}_{1}|u|^{p}dx\right)^{\frac{2}{p}},
\end{equation}
where  $2<p\leq\frac{2n}{n-2k}$ and $\alpha=\frac{(n-2k)p}{2}-n$.

\smallskip

We are interested in the sharp constants  $C$ of (\ref{1.6}).  We remark that
  Benguria,  Frank and   Loss (\cite{bfl}) proved that the sharp constant $C$ of the first order Hardy-Sobolev-Maz'ya inequality in (\ref{1.1}) for $n=3$ and $p=6$ coincides with
the  best Sobolev constant of Talent \cite{ta}. The same result  has been confirmed    for the $\frac{n-1}{2}$-th order Hardy-Sobolev-Maz'ya inequality of dimension $n$ when $n=5$ by Lu and Yang in \cite{LuYang3}, $n=7$ by Hong in \cite{Hong} and all $n\ge 7$ and odd by Lu and Yang in \cite{ly4}). We summarize the results in the following theorem:
\begin{theorem}[\cite{bfl,Hong, LuYang3,ly4}] \label{th1.2} Let $n\geq3$ be odd. There holds,  for each $u\in C^{\infty}_{0}(\mathbb{H}^{n})$,
\begin{equation}\label{1.7}
\int_{\mathbb{H}^{n}}uP_{\frac{n-1}{2}}udV- \left(\prod^{\frac{n-1}{2}}_{i=1}\frac{(2i-1)^{2}}{4}\right)\int_{\mathbb{H}^{n}}u^{2}dV\geq S_{n,(n-1)/2}\left(\int_{\mathbb{H}^{n}}|u|^{2n}dV\right)^{\frac{1}{n}},
\end{equation}
where
\begin{equation}\label{1.8}
S_{n,\gamma}=2^{2\gamma}\pi^{\gamma}\frac{\Gamma(\frac{n+2\gamma}{2})}{\Gamma(\frac{n-2\gamma}{2})}\left(\frac{\Gamma(\frac{n}{2})}{\Gamma(n)}\right)^{2\gamma/n}
\end{equation}
is the best  Sobolev
constant of order $\gamma$ (see \cite{lie,ct}). Furthermore, the inequality is  strict for nonzero $u$'s.

In terms of the half space model of hyperbolic space, (\ref{1.7})  is equivalent to the following higher order Hardy-Sobolev-Maz'ya inequality
\begin{equation}\label{1.9}
\int_{\mathbb{R}^{n}_{+}}u(-\Delta)^{\frac{n-1}{2}}udx- \left(\prod^{\frac{n-1}{2}}_{i=1}\frac{(2i-1)^{2}}{4}\right)\int_{\mathbb{R}^{n}_{+}}\frac{u^{2}}{x^{n-1}_{1}}dx\geq S_{n,(n-1)/2}\left(\int_{\mathbb{R}^{n}_{+}}|u|^{2n}dx\right)^{\frac{1}{n}}.
\end{equation}
\end{theorem}

We should mention that sharp constants for Poincar\'e-Sobolev type inequalities have played an important role in conformal geometry as the powerful applications of those works by
Talenti \cite{ta},  Aubin \cite{au} and Beckner \cite{Be0} have demonstrated.

We remark that Poincar\'e-Sobolev and Hardy-Sobolev-Maz'ya inequalities have been established on complex hyperbolic spaces by Lu and Yang \cite{LuYang-Complex} and on quaternionic and octonionic hyperbolic spaces by Flynn, Lu and Yang \cite{FlynnLuYang}. Therefore, these have been completely settled for all complete and noncompact symmetric spaces of rank 1.

As is well known, the GJMS operators $P_{k}$ on the hyperbolic space $\mathbb{H}^n$ are explicitly known as
$P_{k}=P_{1}(P_{1}+2)\cdot\cdots\cdot(P_{1}+k(k-1))$, $ k\in\mathbb{N}$,
where $P_{1}=-\Delta_{\mathbb{H}}-\frac{n(n-2)}{4}$ is the conformal Laplacian on $\mathbb{H}^{n}$. GJMS operators are conformally covariant operators introduced by Graham, Jenne, Mason and Sparling in \cite{GJMS} based on the construction of ambient metric by C. Fefferman and Graham \cite{FeffermanGr} and \cite{FeffermanGr2}. We refer the reader to   works  by C. Fefferman and Graham \cite{fe3}, Gover \cite{go} and  Juhl \cite{j} for more properties of  GJMS operators.

 The fractional GJMS operators $P_{\gamma}$ are introduced by Graham and Zworski \cite{grz} on asymptotically hyperbolic spaces through the scattering theory initially developed by Mazzeo and Melrose \cite{ma}.
It is constructed  on the conformal infinity  $M$
of   a conformally compact Einstein manifold $(X^{n+1}, g_{+})$ via Dirichlet-to-Neumann
operator for the eigenvalue problem
\begin{align*}
-\Delta_{g_{+}}u-s(n-s)u=0,\;\;\; s=\frac{n}{2}+\gamma.
\end{align*}
The operator $P_{\gamma}$  is a non-local pseudo-differential operator of order $2\gamma$. Furthermore, it is  conformally
invariant:
 for a conformal change
of metric $\widehat{g}=e^{2\tau}g$,
we have
\begin{equation}\label{a1.9}
\begin{split}
\widehat{P}_{\gamma}f=e^{-\frac{n+2\gamma}{2}\tau}P_{\gamma}\left(e^{\frac{n-2\gamma}{2}\tau}f\right),\;\;\forall f\in C^{\infty}(M).
\end{split}
\end{equation}
However, the explicit formulas for the GJMS operators $P_{\gamma}$ on hyperbolic spaces are only known when $\gamma$ is an integer.
Their explicit formulas  of $P_{\gamma}$  when $\gamma$ is not an integer remain a very interesting open problem.

\smallskip

One of the purposes of this paper is to identify the explicit formulas of GJMS operators $P_{\gamma}$ on $\mathbb{H}^n$ when $\gamma$ is not an integer.
Another main purpose is
to establish fractional Poincar\'e-Sobolev inequality on hyperbolic space and fractional Hardy-Sobolev-Maz'ya inequality  on half space
and look for their sharp constants.
To this end,
  we define for $\gamma>0$,
\begin{align}\label{1.10}
P_{\gamma}=2^{2\gamma}\frac{|\Gamma(\frac{3+2\gamma}{4}+\frac{i}{2}\sqrt{-\Delta_{\mathbb{H}}-\frac{(n-1)^{2}}{4}})|^{2}}
{|\Gamma(\frac{3-2\gamma}{4}+\frac{i}{2}\sqrt{-\Delta_{\mathbb{H}}-\frac{(n-1)^{2}}{4}})|^{2}}
\end{align}
in terms of Helgason-Fourier transform on hyperbolic space (see Section \ref{Section3} for more details), where $i=\sqrt{-1}$ and  $\Gamma(\cdot)$ is the Gamma function.
In Section \ref{Section5}, we shall show that $P_{\gamma}$, $\gamma\in (0,\frac{n}{2})\setminus\mathbb{N}$, is nothing but the fractional GJMS operator on hyperbolic space. Here we denote by  $\mathbb{N}=\{0,1,2,\cdots\}$.
One expects that (\ref{1.5})   still holds for fractional GJMS operator on hyperbolic space. However, this is not the case (see Theorem \ref{th1.6}).
The reason is that $P_{\gamma}$ is a non-local pseudo-differential operator on $\mathbb{H}^{n}$, while $(-\Delta)^{\gamma}$ is a non-local pseudo-differential operator on the whole space $\mathbb{R}^{n}$.

\smallskip

We firstly establish  the following Poincar\'e-Sobolev inequalities associated with the operators $P_{\gamma}$ on the hyperbolic space $\mathbb{H}^n$:
\begin{theorem}\label{th1.3}
Let $0<\gamma<\frac{n}{2}$ and $2<p\leq\frac{2n}{n-2\gamma}$. There exists a positive constant $C$ such that for each $u\in C^{\infty}_{0}(\mathbb{H}^{n})$,
\begin{equation}\label{1.11}
\int_{\mathbb{H}^{n}}uP_{\gamma}udV- 2^{2\gamma}\frac{\Gamma(\frac{3+2\gamma}{4})^{2}}
{\Gamma(\frac{3-2\gamma}{4})^{2}}\int_{\mathbb{H}^{n}}u^{2}dV\geq C\left(\int_{\mathbb{H}^{n}}|u|^{p}dV\right)^{\frac{2}{p}}.
\end{equation}
Furthermore, the constant $2^{2\gamma}\frac{\Gamma(\frac{3+2\gamma}{4})^{2}}
{\Gamma(\frac{3-2\gamma}{4})^{2}}$ in (\ref{1.11}) is sharp in the sense that it cannot be replaced by a larger constant.
\end{theorem}

In particular, when $n\geq 3$ and $\frac{n-1}{2}\leq \gamma<\frac{n}{2}$,  we have the following result with the best constant $C=S_{n, \gamma}$. Namely, the best constant for the Poincar\'e-Sobolev inequality on $\mathbb{H}^n$ is the same as the best constant for the fractional Sobolev inequality in $\mathbb{R}^n$.
\begin{theorem} \label{th1.4} Let  $n\geq 3$ and $\frac{n-1}{2}\leq \gamma<\frac{n}{2}$. It holds that
\begin{equation}\label{1.13}
\int_{\mathbb{H}^{n}}uP_{\gamma}udV- 2^{2\gamma}\frac{\Gamma(\frac{3+2\gamma}{4})^{2}}
{\Gamma(\frac{3-2\gamma}{4})^{2}}\int_{\mathbb{H}^{n}}u^{2}dV\geq S_{n,\gamma}\left(\int_{\mathbb{H}^{n}}|u|^{\frac{2n}{n-2\gamma}}dV\right)^{\frac{n-2\gamma}{n}},
\; \; u\in C^{\infty}_{0}(\mathbb{H}^{n}),
\end{equation}
where $S_{n,\gamma}$ is the best Sobolev constant of fractional order $\gamma$.
Furthermore, the inequality is  strict for nonzero $u$'s.
\end{theorem}
\begin{remark}
Even the best constant for the following fractional Sobolev inequality on $\mathbb{H}^n$  when $n\geq 3$ and $\frac{n-1}{2}\leq \gamma<\frac{n}{2}$ (i.e., without subtracting the second term) cannot be obtained from the best constant for the fractional Sobolev inequality in the Euclidean space $\mathbb{R}^n$ directly because the conformal relations  such as \eqref{1.13} and \eqref{1.14} for $\tilde{P}_{\gamma}$  do not hold for the operator $P_{\gamma}$ and thus the following fractional Sobolev inequality for the fractional GJMS operator $P_{\gamma}$ appears to be new:
\end{remark}
\begin{equation}\label{1.13}
\int_{\mathbb{H}^{n}}uP_{\gamma}udV\geq S_{n,\gamma}\left(\int_{\mathbb{H}^{n}}|u|^{\frac{2n}{n-2\gamma}}dV\right)^{\frac{n-2\gamma}{n}},
\; \; u\in C^{\infty}_{0}(\mathbb{H}^{n}).
\end{equation}

\smallskip
In the limiting case, namely $\gamma=\frac{n}{2}$, we have the following fractional Hardy-Adams inequality associated with the GJMS operators $P_{\frac{n}{2}}$ on $\mathbb{H}^n$ (see earlier works by Li, Lu and Yang \cite{LuYang2,LiLuYang1} when $n$ is an even integer):
\begin{theorem}\label{th1.5} Let $n\geq3$ be odd.
There exists a constant $C>0$ such that for all $u\in
C^{\infty}_{0}(\mathbb{H}^{n})$ with
\[
\int_{\mathbb{H}^{n}}uP_{\frac{n}{2}}udV-
2^{n}\frac{\Gamma(\frac{3+n}{4})^{2}}
{\Gamma(\frac{3-n}{4})^{2}}\int_{\mathbb{H}^{n}}u^{2}dV\leq1,
\]
it holds
\[
\int_{\mathbb{H}^{n}}(e^{\beta_{0}\left(n/2,n\right) u^{2}}-1-\beta_{0}\left(n/2,n\right)
u^{2})dV\leq C,
\]
where
\begin{align*}
\beta_{0}(n,m)=\frac{n}{|\mathbb{S}^{n-1}|}\left[\frac{\pi^{n/2}2^{m}\Gamma(m/2)}{\Gamma((n-m)/2)}\right]^{n/(n-m)}
\end{align*}
is the  Adams constant of fractional order (see \cite{laml}, Theorem 1.1)
and  $|\mathbb{S}^{n-1}|=\frac{2\pi^{\frac{n}{2}}}{\Gamma(\frac{n}{2})}$ is the surface measure of sphere $\mathbb{S}^{n-1}=\{x\in\mathbb{R}^{n}: |x|=1\}$.
\end{theorem}

We note that sharp fractional Adams and Hardy-Adams inequalities on hyperbolic spaces  have been established in \cite{LiLuYang2} under the constraint of
Sobolev norm of the fractional Laplacian of the functions.

Next, we consider the Hardy-Sobolev-Maz'ya inequality for  the fractional Laplacian $(-\Delta)^{\gamma}$ on half space $\mathbb{R}^n_+$. We
recall that the fractional Laplace operator on $\mathbb{R}^n$ is defined for $0<\gamma<1$ by

$$(-\Delta )^{\gamma}u(x)=C_{n, \gamma}P.V.\int_{\mathbb{R}^{n}}\frac{u(x)-u(y)}{|x-y|^{n+2\gamma}}dxdy.$$
This definition is equivalent to the one given by Caffarelli and Silvestre \cite{c} by using the extension method.
We mention that   Banica,  Gonz\'alez and   S\'aez \cite{ba}  have also constructed fractional Laplacians on noncompact and complete Riemannian manifolds satisfying certain conditions (including hyperbolic spaces)
through similar  extension techniques  introduced by Caffarelli-Silvestre.

On the other hand, on a subdomain
  $\Omega\subset \mathbb{R}^n$, the fractional Laplace operator can be defined for $0<\gamma<1$ by

$$(-\Delta )^{\gamma}u(x)=C_{n, \gamma}P.V.\int_{\Omega}\frac{u(x)-u(y)}{|x-y|^{n+2\gamma}}dxdy.$$
However, it is well known that
there are several  other ways to define the fractional Laplacian in a
domain $\Omega\subset\mathbb{R}^{n}$, which may  be quite different when $\Omega\neq \mathbb{R}^{n}$.
In this paper, we only consider the domain $\Omega=\mathbb{R}^n_+$ and functions $u\in C_{0}^{\infty}(\Omega)\subset C_{0}^{\infty}(\mathbb{R}^{n})$ while we deal with the Hardy-Sobolev-Maz'ya inequalities on half spaces. Thus,
we will define the fractional  operator $(-\Delta)^{\gamma}$ in terms of Fourier transform in a
domain $\Omega\subset\mathbb{R}^{n}$ for all $\gamma>0$ as follows:
\begin{align*}
\widehat{(-\Delta)^{\gamma}u}=|\xi|^{\gamma}\widehat{u}(\xi),\;\; \widehat{u}(\xi)=\int_{\mathbb{R}^{n}}u(x)e^{-2\pi ix\cdot\xi}dx,
\;\; u\in C_{0}^{\infty}(\Omega)\subset C_{0}^{\infty}(\mathbb{R}^{n}).
\end{align*}
We note that the relationship
  (\ref{1.5}) fails for $P_{\gamma}$ when $\gamma$ is not an integer because   $P_{\gamma}$ is a global operator in $\mathbb{H}^{n}$, while $(-\Delta)^{\gamma}$ is a global operator on the whole $\mathbb{R}^{n}$. Moreover, the Hardy-Sobolev-Maz'ya inequalities associated with the fractional GJMS operators $P_{\gamma}$ do not allow us to derive the fractional Hardy-Sobolev-Maz'ya inequalities associated with the fractional Laplacian $(-\Delta)^{\gamma}$ on half spaces $\mathbb{R}^{n}_{+}$ when $\gamma$ is not an integer. Therefore, we need to introduce another type of
  fractional operators $\widetilde{P}_{\gamma}$ on the hyperbolic space $\mathbb{H}^n$ that satisfy \eqref{1.13} and \eqref{1.14} below.
   In fact, we have the following theorem:
\begin{theorem} \label{th1.6} Let $\gamma>0$ and set
$$
\widetilde{P}_{\gamma}=\frac{|\Gamma(\gamma+\frac{1}{2}+i\sqrt{-\Delta_{\mathbb{H}}-\frac{(n-1)^{2}}{4}})|^{2}}
{|\Gamma(\frac{1}{2}+i\sqrt{-\Delta_{\mathbb{H}}-\frac{(n-1)^{2}}{4}})|^{2}}.$$
 Then the following holds:
\begin{align}\label{1.13}
x_{1}^{\gamma+\frac{n}{2}}(-\Delta)^{\gamma}( x_{1}^{\gamma-\frac{n}{2}}u)=&\widetilde{P}_{\gamma}u\;\;\;\; \textrm{in}\;(\mathbb{H}^{n}, g_{\mathbb{H}});\\
\label{1.14}
\left(\frac{1-|x|^{2}}{2}\right)^{\gamma+\frac{n}{2}}(-\Delta)^{\gamma}\left[ \left(\frac{1-|x|^{2}}{2}\right)^{\gamma-\frac{n}{2}}v\right]=&\widetilde{P}_{\gamma}v\;\;\;\; \textrm{in}\;(\mathbb{B}^{n}, g_{\mathbb{B}}),
\end{align}
where $(\mathbb{B}^{n}, g_{\mathbb{B}})$ is the ball model of hyperbolic space (see Section \ref{Section2}),  $u\in C_{0}^{\infty}(\mathbb{R}_{+}^{n})$ and $v\in C_{0}^{\infty}(\mathbb{B}^{n})$.
\end{theorem}

One sees that
\begin{align*}
\widetilde{P}_{\gamma}=P_{\gamma},\;\;\gamma\in \mathbb{N}\;\;\textrm{but} \;\; \widetilde{P}_{\gamma}\neq P_{\gamma},\;\;\gamma\in (0,+\infty)\backslash\mathbb{N}.
\end{align*}

 In fact, we have (see Theorem \ref{co4.3} below)
\begin{equation}\label{relation}
P_{\gamma}=\widetilde{P}_{\gamma}+\frac{\sin\gamma\pi}{\pi}\left|\Gamma\left(\gamma+\frac{1}{2}+i\sqrt{-\Delta_{\mathbb{H}}-\frac{(n-1)^{2}}{4}}\right)\right|^{2}.
\end{equation}

We then establish  the fractional Poincar\'e-Sobolev inequalities associated with the fractional operators $\widetilde{P}_{\gamma}$ on $\mathbb{H}^n$ and  fractional  Hardy-Sobolev-Maz'ya inequalities on $\mathbb{R}_{+}^{n}$ and $\mathbb{B}^{n}$ associated the fractional Laplacian $(-\Delta)^{\gamma}$. The main result is the following theorem:
\begin{theorem}\label{th1.7}
Let $n\geq2$, $0<\gamma<\frac{n}{2}$ and $2<p\leq\frac{2n}{n-2\gamma}$. There exists a positive constant $C$ such that for each $u\in C^{\infty}_{0}(\mathbb{H}^{n})$,
\begin{equation}\label{1.15}
\int_{\mathbb{H}^{n}}u\widetilde{P}_{\gamma}udV- \frac{\Gamma(\gamma+\frac{1}{2})^{2}}
{\Gamma(\frac{1}{2})^{2}}\int_{\mathbb{H}^{n}}u^{2}dV\geq C\left(\int_{\mathbb{H}^{n}}|u|^{p}dV\right)^{\frac{2}{p}}.
\end{equation}
Furthermore, the constant $ \frac{\Gamma(\gamma+\frac{1}{2})^{2}}
{\Gamma(\frac{1}{2})^{2}}$ in (\ref{1.15}) is sharp in the sense that it cannot be replaced by a larger number.

In terms of half space and ball models  of hyperbolic space,  (\ref{1.15})  is equivalent to the following fractional  Hardy-Sobolev-Maz'ya inequalities on $\mathbb{R}^{n}_{+}$
and $\mathbb{B}^{n}$:
\begin{align}\label{1.16}
\int_{\mathbb{R}^{n}_{+}}u(-\Delta)^{\gamma}udx- \frac{\Gamma(\gamma+\frac{1}{2})^{2}}
{\Gamma(\frac{1}{2})^{2}}\int_{\mathbb{R}^{n}_{+}}\frac{u^{2}}{x_{1}^{2\gamma}}dx\geq& C\left(\int_{\mathbb{R}_{+}^{n}}x^{\alpha}_{1}|u|^{p}dx\right)^{\frac{2}{p}};\\
\label{1.17}
\int_{\mathbb{B}^{n}}v(-\Delta)^{\gamma}vdx- \frac{\Gamma(\gamma+\frac{1}{2})^{2}}
{\Gamma(\frac{1}{2})^{2}}\int_{\mathbb{B}^{n}}v^{2}\frac{2^{2\gamma}}{(1-|x|^{2})^{2\gamma}}dx\geq& C\left(\int_{\mathbb{B}^{n}}(1-|x|^{2})^{\alpha}|v|^{p}dx\right)^{\frac{2}{p}},
\end{align}
where  $u\in C^{\infty}_{0}(\mathbb{R}_{+}^{n})$,  $v\in C^{\infty}_{0}(\mathbb{B}^{n})$ and $\alpha=\frac{(n-2\gamma)p}{2}-n$.
\end{theorem}

When $n\geq 3$ and $\frac{n-1}{2}\leq \gamma<\frac{n}{2}$,  we further have the following inequalities with the best constant $S_{n,\gamma}$.

\begin{theorem} \label{th1.8} Let $n\geq 3$ and $\frac{n-1}{2}\leq \gamma<\frac{n}{2}$. Then there holds
\begin{equation}\label{1.18}
\int_{\mathbb{H}^{n}}u\widetilde{P}_{\gamma}udV- \frac{\Gamma(\gamma+\frac{1}{2})^{2}}
{\Gamma(\frac{1}{2})^{2}}\int_{\mathbb{H}^{n}}u^{2}dV\geq S_{n,\gamma}\left(\int_{\mathbb{H}^{n}}|u|^{\frac{2n}{n-2\gamma}}dV\right)^{\frac{n-2\gamma}{n}},
\; \; u\in C^{\infty}_{0}(\mathbb{H}^{n}).
\end{equation}
Furthermore, the inequality is  strict for nonzero $u$'s.

In terms of half space and ball models  of hyperbolic space, we have
\begin{align*}
\int_{\mathbb{R}^{n}_{+}}u(-\Delta)^{\gamma}udx- \frac{\Gamma(\gamma+\frac{1}{2})^{2}}
{\Gamma(\frac{1}{2})^{2}}\int_{\mathbb{R}^{n}_{+}}\frac{u^{2}}{x_{1}^{2\gamma}}dx\geq&
S_{n,\gamma}\left(\int_{\mathbb{R}_{+}^{n}}|u|^{\frac{2n}{n-2\gamma}}dV\right)^{\frac{n-2\gamma}{n}};\\
\int_{\mathbb{B}^{n}}v(-\Delta)^{\gamma}vdx- \frac{\Gamma(\gamma+\frac{1}{2})^{2}}
{\Gamma(\frac{1}{2})^{2}}\int_{\mathbb{B}^{n}}\frac{2^{2\gamma}v^{2}}{(1-|x|^{2})^{2\gamma}}dx\geq&
S_{n,\gamma}\left(\int_{\mathbb{B}^{n}}|v|^{\frac{2n}{n-2\gamma}}dx\right)^{\frac{n-2\gamma}{n}},
\end{align*}
where  $u\in C^{\infty}_{0}(\mathbb{R}_{+}^{n})$ and  $v\in C^{\infty}_{0}(\mathbb{B}^{n})$.
\end{theorem}

In the borderline case $\gamma=\frac{n}{2}$, we  have the following sharp fractional Hardy-Adams inequalities associated with both $\widetilde{P}_{\frac{n}{2}}$ on $\mathbb{H}^n$ and $(-\Delta)^{\frac{n}{2}}$ on half space $\mathbb{R}_{+}^{n}$ and unit ball $\mathbb{B}^{n}$:
\begin{theorem}\label{th1.9} Let $n\geq3$ be odd.
There exists a constant $C>0$ such that for all $u\in
C^{\infty}_{0}(\mathbb{H}^{n})$ with
\[
\int_{\mathbb{H}^{n}}u\widetilde{P}_{\frac{n}{2}}udV-
\frac{\Gamma(\frac{n+1}{2})^{2}}
{\Gamma(\frac{1}{2})^{2}}\int_{\mathbb{H}^{n}}u^{2}dV\leq1,
\]
it holds
\begin{align}\label{1.19}
\int_{\mathbb{H}^{n}}(e^{\beta_{0}\left(n/2,n\right) u^{2}}-1-\beta_{0}\left(n/2,n\right)
u^{2})dV\leq C.
\end{align}

In terms of half space model and ball model  of hyperbolic space, we have
\begin{align*}
\int_{\mathbb{R}_{+}^{n}}\frac{e^{\beta_{0}\left(n/2,n\right) u^{2}}-1-\beta_{0}\left(n/2,n\right)
u^{2}}{x_{1}^{n}}dx\leq &C;\\
\int_{\mathbb{B}^{n}}\frac{e^{\beta_{0}\left(n/2,n\right) v^{2}}-1-\beta_{0}\left(n/2,n\right)
v^{2}}{(1-|x|^{2})^{n}}dx\leq& C,
\end{align*}
for each $u\in C_{0}^{\infty}(\mathbb{R}_{+}^{n})$ and $v\in C_{0}^{\infty}(\mathbb{B}^{n})$ with
\begin{align*}
\int_{\mathbb{R}_{+}^{n}}u(-\Delta)^{\frac{n}{2}}udx-
\frac{\Gamma(\frac{n+1}{2})^{2}}
{\Gamma(\frac{1}{2})^{2}}\int_{\mathbb{R}^{n}_{+}}\frac{u^{2}}{x_{1}^{n}}dx\leq&1;\\
\int_{\mathbb{B}^{n}}v(-\Delta)^{\frac{n}{2}}vdx-
\frac{\Gamma(\frac{n+1}{2})^{2}}
{\Gamma(\frac{1}{2})^{2}}\int_{\mathbb{B}^{n}}\frac{2^{n}v^{2}}{(1-|x|^{2})^{n}}dx\leq&1.
\end{align*}
\end{theorem}

The organization of this paper is as follows. In Section \ref{Section2}, we will review some preliminary facts  about special functions that will be needed in the subsequent sections. Section \ref{Section3} review
some Helgason-Fourier analysis on hyperbolic spaces. In Section \ref{Section4}, we will give the explicit formula of Green's function  of the fractional  operators $\frac{|\Gamma(\gamma+\nu+i\sqrt{-\Delta_{\mathbb{H}}-\frac{(n-1)^{2}}{4}})|^{2}}
{|\Gamma(\nu+i\sqrt{-\Delta_{\mathbb{H}}-\frac{(n-1)^{2}}{4}})|^{2}}$ for each $(\nu\geq 0)$  via Helgason-Fourier analysis. This Green's function estimate  plays an important role in the proofs of Theorem \ref{th1.4} and Theorem \ref{1.8}. A sharp Sobolev inequality with the best constant for the operator   $\frac{|\Gamma(\nu+\gamma+i\sqrt{-\Delta_{\mathbb{H}}-\frac{(n-1)^{2}}{4}})|^{2}}
{|\Gamma(\nu+i\sqrt{-\Delta_{\mathbb{H}}-\frac{(n-1)^{2}}{4}})|^{2}}u$ is also proved in this section. (Theorem \ref{th3.10}.)
In Section \ref{Section5},  we compute the  explicit formula of fractional  GJMS
 operators $P_{\gamma}$ on $\mathbb{H}^{n}$ and give the proofs of   Theorems \ref{th1.3} and  \ref{th1.5}.
The
proofs of Theorems \ref{th1.6}, \ref{th1.7} and \ref{th1.9} are given in Section \ref{Section6}.
In Section \ref{Section7}, we show that the sharp constants of the fractional Poincar\'e-Sobolev inequalities on hyperbolic space and fractional Hardy-Sobolev-Maz'ya inequalities.

\section{Notations and preliminaries}\label{Section2}
We begin by quoting some preliminary facts which will be needed in
the sequel and  refer to \cite{ah,ge,he,he2,hu,liup} for more information about this subject.
Throughout this paper, the symbol $A=O(B)$, or $A\lesssim B$ (resp. $A\gtrsim B$),
between two positive expressions means that there is a constant $C > 0$ such that $A \leq CB$ (resp. $A\geq CB$). The symbol $A\thicksim B$ means
that $A\lesssim B$ and $B\lesssim A$.

 It is well known that hyperbolic space  is a
  noncompact Riemannian symmetric space of rank one that has a constant negative curvature $-1$.
It has several  models, for example,
 the Poincar\'e  half space model $(\mathbb{H}^{n}, g_{\mathbb{H}})$  and the Poincar\'e ball model $(\mathbb{B}^{n}, g_{\mathbb{B}})$.

\subsection{The Poincar\'e half space model $(\mathbb{H}^{n}, g_{\mathbb{H}})$}

It is given by $\mathbb{R}_{+}\times\mathbb{R}^{n-1}=\{(x_{1},\cdots,x_{n}): x_{1}>0\}$ equipped with the Riemannian metric
$$g_{\mathbb{H}}=\frac{dx_{1}^{2}+\cdots+dx_{n}^{2}}{x^{2}_{1}}.$$
The induced Riemannian measure can be written as $dV=\frac{dx}{x^{n}_{1}}$, where $dx$ is the Lebesgue measure on
$\mathbb{R}^{n}$.
The
Laplace-Beltrami operator on $\mathbb{H}^{n}$ is given by
\begin{align*}
\Delta_{\mathbb{H}}=x^{2}_{1}\sum\limits^{n}_{i=1}\frac{\partial^{2}}{\partial
x^{2}_{i}}-(n-2)x_{1}\frac{\partial}{\partial x_{1}}.
\end{align*}
For simplicity, we denote by
\begin{align*}
\|u\|_{p}=\left(\int_{\mathbb{H}^{n}}|u|^{p}dV\right)^{\frac{1}{p}},\; p>0.
\end{align*}

Let $x, y\in\mathbb{H}^{n}$.
The geodesic distance $\rho(x,y)$  between $x$ and $y$  satisfies
$
\sinh\frac{\rho(x,y)}{2}=\frac{|x-y|}{2\sqrt{x_{1}y_{1}}}.
$
For simplicity, we set
\begin{align*}
\sinh\frac{\rho(x)}{2}=\sinh\frac{\rho(x,(1,0,\cdots,0))}{2}=\frac{\sqrt{|x|^{2}-2x_{1}+1}}{2\sqrt{x_{1}}}.
\end{align*}
Then
\begin{align}\label{2.1}
\cosh\frac{\rho(x)}{2}=\frac{\sqrt{|x|^{2}+2x_{1}+1}}{2\sqrt{x_{1}}}.
\end{align}
The polar coordinate formula on $\mathbb{H}^{n}$ reads
\begin{align}\label{2.2}
\int_{\mathbb{H}^{n}}f(x)dV=\int_{0}^{\infty}\int_{\mathbb{S}^{n-1}}f(\rho,\theta)(\sinh\rho)^{n-1}d\rho d\sigma,\;\; f\in L^{1}(\mathbb{H}^{n}).
\end{align}
We remark that if $f$ is radial, then
\begin{align}\label{2.3}
 \Delta_{\mathbb{H}}f=\frac{\partial^{2}}{\partial \rho^{2}}f+(n-1)\coth\rho\frac{\partial}{\partial \rho}f.
\end{align}

\subsection{The Poincar\'e ball model $(\mathbb{B}^{n}, g_{\mathbb{B}})$ }
It is given by the unit ball
\[\mathbb{B}^{n}=\{x=(x_{1},\cdots,x_{n})\in \mathbb{R}^{n}| |x|<1\}\]
equipped  with the Riemannian metric
\[
g_{\mathbb{B}^{n}}=\frac{4(dx^{2}_{1}+\cdots+dx^{2}_{n})}{(1-|x|^{2})^{2}}.
\]
The induced Riemannian measure is  $dV=\left(\frac{2}{1-|x|^{2}}\right)^{n}dx$.
The
Laplace-Beltrami operator is given by
\[
\Delta_{\mathbb{H}}=\frac{1-|x|^{2}}{4}\left\{(1-|x|^{2})\sum^{n}_{i=1}\frac{\partial^{2}}{\partial
x^{2}_{i}}+2(n-2)\sum^{n}_{i=1}x_{i}\frac{\partial}{\partial
x_{i}}\right\}.
\]
The geodesic
distance $\rho(x,y)$  from $x$ to $y$ in $(\mathbb{B}^{n}, g_{\mathbb{B}})$ is
\begin{equation*}
\rho(x,y)=\log\frac{1+|T_{y}(x)|}{1-|T_{y}(x)|},
\end{equation*}
where $T_{y}(x)$ is the M\"obius
transformation defined by (see e.g. \cite{ah,hu})
\[
T_{y}(x)=\frac{|x-y|^{2}y-(1-|y|^{2})(x-y)}{1-2x\cdot
y+|x|^{2}|y|^{2}}.
\]
Here $x\cdot y=x_{1}y_{1}+x_{2}y_{2}+\cdots +x_{n}y_{n}$ denotes
the  scalar product in $\mathbb{R}^{n}$.
We remark  that the hyperbolic measure on $\mathbb{B}^{n}$ is invariant with respect to the M\"obius
transformations. Furthermore,
we need the following facts of $T_{a}$:
\begin{equation}\label{2.4}
\begin{split}
\cosh\frac{\rho(T_{a}(x))}{2}=&\frac{1}{\sqrt{1-|T_{a}(x)|^{2}}}=\frac{\sqrt{1-2x\cdot a+|x|^{2}|a|^{2}}}{\sqrt{(1-|a|^{2})(1-|x|^{2})}}.
\end{split}
\end{equation}

Using the M\"obius transformations,  we can define the convolution of
measurable functions $f$ and $g$ on $\mathbb{B}^{n}$ by (see e.g. \cite{liup})
\begin{equation*}
(f\ast g)(x)=\int_{\mathbb{B}^{n}}f(y)g(T_{x}(y))dV(y)
\end{equation*}
provided this integral exists. If $g$ is radial,  then (see e.g. \cite{liup}, Proposition 3.15)
\begin{equation}\label{2.5}
  (f\ast g)\ast h= f\ast (g\ast h)
\end{equation}
provided $f,g,h\in L^{1}(\mathbb{B}^{n})$

\subsection{Hypergeometric functions}
We use the notation $F(a,b;c;z)$ to denote
  \begin{equation*}
  \begin{split}
F(a,b;c;z)=\sum^{\infty}_{k=0}\frac{(a)_{k}(b)_{k}}{(c)_{k}}\frac{z^{k}}{k!},
\end{split}
\end{equation*}
where $c\neq0,-1,\cdots,-n,\cdots$ and $(a)_{k}$ is the rising Pochhammer symbol defined by
$$
(a)_{0}=1,\;(a)_{k}=a(a+1)\cdots(a+k-1)=\frac{\Gamma(a+k)}{\Gamma(a)}, \;k\geq1.
$$
If either $a$ or $b$ is a nonpositive integer, then the series terminates and  the function reduces to a polynomial.

 Here we
only list some of properties of hypergeometric function which will be used in the rest of our paper. For
more information about  these functions, we refer to \cite{gr}, section 9.1 and \cite{er}, Chapter II.
\begin{itemize}
  \item Integral representation:
  \begin{equation}\label{2.6}
  \begin{split}
F(a,b;c;z)=\frac{\Gamma(c)}{\Gamma(c-b)\Gamma(b)}\int_{0}^{1}t^{b-1}(1-t)^{c-b-1}(1-tz)^{-a}dt,\;  c>b>0.
\end{split}
\end{equation}

  \item If $\textrm{Re} (c-a-b)>0$, then $F(a,b;c;1)$ exists and
      \begin{equation}\label{2.7}
  \begin{split}
F(a,b;c;1)=\frac{\Gamma(c)\Gamma(c-a-b)}{\Gamma(c-a)\Gamma(c-b)}.
\end{split}
\end{equation}

  \item Transformation formulas:
    \begin{equation}\label{2.8}
  \begin{split}
F(a,b;c;z)=(1-z)^{c-a-b} F(c-a,c-b;c;z).
\end{split}
\end{equation}

 \item Differentiation formula:
    \begin{equation}\label{2.9}
  \begin{split}
\frac{d^{k}}{dz^{k}}F(a,b;c;z)=\frac{(a)_{k}(b)_{k}}{(c)_{k}}F(a+k,b+k;c+k;z),\;\;k\geq1.
\end{split}
\end{equation}
\end{itemize}

\subsection{Gamma function} The gamma function $\Gamma(z)$  is defined as the analytic continuation of the integration
 \begin{align*}
\int_{0}^{\infty}t^{z-1}e^{-t}dt,\;\;\; \textrm{Re}(z)>0,
 \end{align*}
  to a meromorphic function that is holomorphic in the whole complex plane $\mathbb{C}$ except zero and the negative integers.
  We shall use  the following facts of the gamma function:
 \begin{align}\label{2.10}
\overline{  \Gamma(z)}=&\Gamma(\overline{z}); \\
\label{a2.10}
 \Gamma(2z)=&2^{2z-1}\frac{\Gamma(z)\Gamma(z+1/2)}{\sqrt{\pi}}\\
\label{2.11}
 \Gamma(z)\Gamma(1-z)=&\frac{\pi}{\sin\pi z},\;z\neq0,\pm1,\pm2,\cdots;\\
 \label{2.12}
|\Gamma(i\lambda)|^{2}=&\frac{\pi}{\lambda\sinh\pi \lambda},\;\;\;\; \lambda\in\mathbb{R},\;\; \lambda\neq 0;\\
\label{2.13}
|\Gamma(a+i\lambda)|^{2}=&|\Gamma(a)|^{2}\prod_{k=0}^{\infty}\frac{1}{1+\frac{\lambda^{2}}{(a+k)^{2}}},\;\; a\in\mathbb{R},\;\; a\neq 0,-1,-2,\cdots;\\
\label{2.14}
|\Gamma(n+1+i\lambda)|^{2}=&\frac{\pi \lambda}{\sinh\pi \lambda}\prod_{k=1}^{n}(k^{2}+\lambda^{2}),\;\; n\in\mathbb{N};\\
\label{2.15}
|\Gamma(1/2+i\lambda)|^{2}=&\frac{\pi }{\cosh\pi \lambda};\\
\label{2.16}
|\Gamma(n+1/2+i\lambda)|^{2}=&\frac{\pi }{\cosh\pi \lambda}\prod_{k=1}^{n}((k-1/2)^{2}+\lambda^{2}),\;\; n\in\mathbb{N};\\
\label{2.17}
\lim_{|\lambda|\rightarrow\infty}|\Gamma(a+i\lambda)|e^{\frac{\pi}{2}\lambda}|\lambda|^{\frac{1}{2}-a}=&\sqrt{2\pi};\\
\label{2.18}
\frac{\Gamma'(z)}{\Gamma(z)}=&-\Gamma'(1)-\frac{1}{z}+\sum_{k=1}^{\infty}\left(\frac{1}{k}-\frac{1}{z+k}\right).
  \end{align}
  By (\ref{2.13}), we obtain
  \begin{align}\label{2.19}
   |\Gamma(a+i\lambda)|\leq |\Gamma(a)|
  \end{align}
  and
  \begin{align}\label{2.20}
 \frac{|\Gamma(a+i\lambda)|}{|\Gamma(b+i\lambda)|}\geq \frac{|\Gamma(a)|}{|\Gamma(b)|},\;\; \textrm{provided }\;\; (a+k)^{2}\geq (b+k)^{2},\;\forall k\in\mathbb{N}.
  \end{align}
We also need the following inequality  of the gamma function (see \cite{pa}, Page 34, (2.1.16))
\begin{align}\label{2.21}
|\Gamma(z+b)|\geq |z|^{b-a}|\Gamma(z+a)|, \; a\geq0,\; b-a\geq 1,\; \textrm{Re}(z)>0.
\end{align}

\subsection{Associated Legendre functions of the first
kind} We denote by $P_{\nu}^{\mu}(z)$ the associated Legendre functions of the first
kind. It is known that $P_{\nu}^{\mu}(z)$ can be defined via hypergeometric function:
\begin{align*}
P_{\nu}^{\mu}(z)=\frac{1}{\Gamma(1-\mu)}\left(\frac{z+1}{z-1}\right)^{\frac{\mu}{2}}F\big(-\nu,\nu+1;1-\mu,\frac{1-z}{2}\big).
\end{align*}
An integral formula of $P_{\nu}^{\mu}(z)$ is  (see \cite{pr}, Page 197, 2.17.3(7)):
\begin{align*}
\int_{a}^{\infty}(x+a)^{\alpha-1}(x-a)^{-\frac{\mu}{2}}&P^{\mu}_{\nu}\left(\frac{x}{a}\right)dx=(2a)^{\alpha-\frac{\mu}{2}}
\frac{\Gamma(\frac{\mu}{2}-\nu-\alpha)\Gamma(1+\frac{\mu}{2}+\nu-\alpha)}{\Gamma(1-\frac{\mu}{2}-\alpha)\Gamma(1+\frac{\mu}{2}-\alpha)},\\
&a>0,\; \textrm{Re}(\mu)<1,\; \textrm{Re}(\alpha)< 1+\textrm{Re}(\nu+\frac{\mu}{2}),\;\textrm{Re}(\alpha)< \textrm{Re}(\frac{\mu}{2}-\nu).
\end{align*}
Letting $a=1$ and substituting $\alpha=1-\frac{\mu}{2}-\gamma$ in the above equality, we obtain
\begin{align*}
\int_{1}^{\infty}(x+1)^{-\gamma}(x^{2}-1)^{-\frac{\mu}{2}}&P^{\mu}_{\nu}(x)dx=2^{1-\mu -\gamma}
\frac{\Gamma(\gamma+\mu-\nu-1)\Gamma(\gamma+\mu+\nu)}{\Gamma(\gamma)\Gamma(\gamma+\mu)},\\
& \textrm{Re}(\mu)<1,\; \textrm{Re}(\gamma)>-\textrm{Re}(\nu+\frac{\mu}{2}),\;\textrm{Re}(\gamma)> -\textrm{Re}(\mu-\nu)-1.
\end{align*}
In particular, we have, for  $\mu=1-\frac{n}{2}$ and $\nu=i\lambda-\frac{1}{2}$ $(\lambda\in\mathbb{R})$,
\begin{align}\label{2.22}
&\int_{1}^{\infty}(x+1)^{-\gamma}(x^{2}-1)^{\frac{n-2}{4}}P^{\frac{2-n}{2}}_{i\lambda-\frac{1}{2}}(x)dx=2^{\frac{n}{2} -\gamma}
\frac{|\Gamma(\gamma-\frac{n-1}{2}+i\lambda)|^{2}}{\Gamma(\gamma)\Gamma(\gamma+1-\frac{n}{2})},
\; \gamma>\frac{n-1}{2}.
\end{align}
Substituting  $x=\cosh\rho$  into (\ref{2.22}), we get
\begin{align*}
&\int_{0}^{\infty}\big(\cosh\frac{\rho}{2}\big)^{-2\gamma}(\sinh\rho)^{\frac{n}{2}}P^{\frac{2-n}{2}}_{i\lambda-\frac{1}{2}}(\cosh\rho)d\rho=2^{\frac{n}{2}}
\frac{|\Gamma(\gamma-\frac{n-1}{2}+i\lambda)|^{2}}{\Gamma(\gamma)\Gamma(\gamma+1-\frac{n}{2})},
\; \gamma>\frac{n-1}{2},
\end{align*}
i.e.,
\begin{align}\label{2.23}
&\int_{0}^{\infty}\big(\cosh\frac{\rho}{2}\big)^{-\gamma}(\sinh\rho)^{\frac{n}{2}}P^{\frac{2-n}{2}}_{i\lambda-\frac{1}{2}}(\cosh\rho)d\rho=2^{\frac{n}{2}}
\frac{|\Gamma(\frac{\gamma+1-n}{2}+i\lambda)|^{2}}{\Gamma(\frac{\gamma}{2})\Gamma(\frac{\gamma+2-n}{2})},
\; \gamma>n-1.
\end{align}

\section{ Helgason-Fourier transform on hyperbolic spaces  }\label{Section3}

We first  review some basic facts about Helgason-Fourier transform on hyperbolic spaces in terms of ball model $(\mathbb{B}^{n}, g_{\mathbb{B}})$.
For more information about this subject, we refer  to  \cite{he,he2,ge}.

Set
\[
e_{\lambda,\zeta}(x)=\left(\frac{\sqrt{1-|x|^{2}}}{|x-\zeta|}\right)^{n-1+i\lambda}, \;\; x\in \mathbb{B}^{n},\;\;\lambda\in\mathbb{R},\;\;\zeta\in\mathbb{S}^{n-1}.
\]
The Helgason-Fourier  transform of a function  $f$  on $\mathbb{B}^{n}$ is defined as
\[
\widehat{f}(\lambda,\zeta)=\int_{\mathbb{B}^{n}} f(x)e_{-\lambda,\zeta}(x)dV
\]
provided this integral exists.
The Helgason-Fourier transform  is an isometry from
$L^{2}(\mathbb{B}^{n}, dV)$ onto $L^{2}(\mathbb{R}\times \mathbb{S}^{n-1}, |\mathfrak{c}(\lambda)|^{-2}d\lambda d\sigma(\zeta))$, where
$\mathfrak{c}(\lambda)$ is the  Harish-Chandra $\mathfrak{c}$-function given by
\begin{align}\label{3.1}
\mathfrak{c}(\lambda)=\sqrt{2}(2\pi)^{n/2}\frac{\Gamma(i\lambda)}{\Gamma(i\lambda+\frac{n-1}{2})}.
\end{align}
In fact, it holds that, for  complex-valued functions $f, g\in L^{2}(\mathbb{B}^{n})$,
\begin{align}\label{3.2}
\int_{\mathbb{B}^{n}}f(x)\overline{g(x)}dV=
\int^{+\infty}_{-\infty}\int_{\mathbb{S}^{n-1}} \widehat{f}(\lambda,\zeta)
 \overline{\widehat{g}(\lambda,\zeta)}|\mathfrak{c}(\lambda)|^{-2}d\lambda d\sigma(\zeta).
\end{align}
In particular, we have  the
 Plancherel formula:
\begin{equation*}
\int_{\mathbb{B}^{n}}|f(x)|^{2}dV=\int^{+\infty}_{-\infty}\int_{\mathbb{S}^{n-1}}|\widehat{f}(\lambda,\zeta)|^{2}|\mathfrak{c}(\lambda)|^{-2}d\lambda d\sigma(\zeta).
\end{equation*}
The inversion formula  for $f\in C^{\infty}_{0}(\mathbb{B}^{n})$ reads (see e.g. \cite{ge}, Theorem 8.4.1):
\[
f(x)=\int^{+\infty}_{-\infty}\int_{\mathbb{S}^{n-1}} \widehat{f}(\lambda,\zeta)e_{\lambda,\zeta}(x)|\mathfrak{c}(\lambda)|^{-2}d\lambda d\sigma(\zeta).
\]
We remark that if $f,g\in C^{\infty}_{0}(\mathbb{B}^{n})$ and $g$ is radial, then
\begin{align*}
\widehat{(f\ast g)}=\widehat{f}\cdot\widehat{g}.
\end{align*}

Since $e_{\lambda,\zeta}(x)$ is an eigenfunction of $\Delta_{\mathbb{H}}$ with eigenvalue $-\frac{(n-1)^{2}}{4}+\lambda^{2}$ (see \cite{ge}, Page 136, (8.2.16)), we have, for
$f\in C^{\infty}_{0}(\mathbb{B}^{n})$,
\[
\widehat{\Delta_{\mathbb{H}}f}(\lambda,\zeta)=-\left(\frac{(n-1)^{2}}{4}+\lambda^{2}\right)\widehat{f}(\lambda,\zeta).
\]
Therefore, in analogy with the Euclidean setting, we can  define the fractional
Laplacian on hyperbolic space as follows:
\begin{equation}\label{3.3}
\widehat{(-\Delta_{\mathbb{H}})^{\gamma}f}(\lambda,\zeta)=\left(\frac{(n-1)^{2}}{4}+\lambda^{2}\right)^{\gamma}\widehat{f}(\lambda,\zeta),\;\;\gamma\in \mathbb{R}.
\end{equation}

If $f$ is radial, then $\widehat{f}(\lambda,\zeta)=\widehat{f}(\lambda)$
is
independent of $\zeta$. Furthermore, we have (see e.g. \cite{ge}, Page 137, (8.3.2))
\begin{align}\label{3.4}
\widehat{f}(\lambda)=|\mathbb{S}^{n-1}|\int_{0}^{\infty}f(\cosh \rho)\varphi_{-\lambda}(x)(\sinh \rho)^{n-1}dr,
\end{align}
where
\begin{align*}
\varphi_{\lambda}(x)=\frac{1}{|\mathbb{S}^{n-1}|}\int_{\mathbb{S}^{n-1}}e_{\lambda,\zeta}(x)d\omega
\end{align*}
is the
spherical function on hyperbolic space. Since $\varphi_{\lambda}$ is a spherical function, we have
\begin{align}\label{3.5}
\varphi_{-\lambda}=\varphi_{\lambda}.
\end{align}
Furthermore,  $\varphi_{\lambda}(x)$ is  radial and we have (see \cite{ge}, Page 138, (8.3.9))
\begin{align}\label{3.6}
\varphi_{\lambda}(x)=2^{\frac{n-2}{2}}\Gamma(n/2)(\sinh \rho)^{\frac{2-n}{2}}P^{\frac{2-n}{2}}_{i\lambda-\frac{1}{2}}(\cosh \rho).
\end{align}
Therefore, if $f$ is radial, then by using (\ref{3.4})-(\ref{3.6}) and $|\mathbb{S}^{n-1}|=\frac{2\pi^{n/2}}{\Gamma(\frac{n}{2})}$, we obtain
\begin{align}\nonumber
\widehat{f}(\lambda)=&|\mathbb{S}^{n-1}|\int_{0}^{\infty}f(\cosh \rho)2^{\frac{n-2}{2}}\Gamma(n/2)(\sinh \rho)^{\frac{2-n}{2}}P^{\frac{2-n}{2}}_{i\lambda-\frac{1}{2}}(\cosh \rho)(\sinh \rho)^{n-1}d\rho\\
\label{3.7}
=&(2\pi)^{n/2}\int_{0}^{\infty}f(\cosh \rho)(\sinh \rho)^{\frac{n}{2}}P^{\frac{2-n}{2}}_{i\lambda-\frac{1}{2}}(\cosh \rho)d\rho.
\end{align}

\section{Green's  function of the operator $\frac{|\Gamma(\gamma+\nu+i\sqrt{-\Delta_{\mathbb{H}}-\frac{(n-1)^{2}}{4}})|^{2}}
{|\Gamma(\nu+i\sqrt{-\Delta_{\mathbb{H}}-\frac{(n-1)^{2}}{4}})|^{2}}$ and its Helgason-Fourier transform }\label{Section4}

We shall use (\ref{3.7}) to compute the Helgason-Fourier transform of the   function
\begin{align}\label{3.8}
K_{\nu,\gamma}(\cosh\rho)=C_{\nu,\gamma} &\left(\cosh\frac{\rho}{2}\right)^{1-n-2\nu} F\big(\nu+\frac{n-1}{2},\nu+\frac{1}{2};2\nu+\gamma;
(\cosh\frac{\rho}{2})^{-2}\big),\\
&\;\; \nu\geq0, \; \; \gamma>0,\; \; C_{\nu,\gamma}=\frac{\Gamma(\frac{n-1}{2}+\nu)\Gamma(\nu+\frac{1}{2})}{2^{n}\pi^{\frac{n}{2}}
\Gamma(\gamma)\Gamma(2\nu+\gamma)}. \nonumber
\end{align}

Before we show that the function $K_{\nu,\gamma}(\cosh\rho)$ is exactly the Green function of the operator $\frac{|\Gamma(\gamma+\nu+i\sqrt{-\Delta_{\mathbb{H}}-\frac{(n-1)^{2}}{4}})|^{2}}
{|\Gamma(\nu+i\sqrt{-\Delta_{\mathbb{H}}-\frac{(n-1)^{2}}{4}})|^{2}}$,
we will first give the asymptotic estimates of $K_{\nu,\gamma}$ in the following lemma:
\begin{lemma}\label{lm3.1}  Given $\nu\geq0$ and $\gamma>0$. Then the following holds:
\begin{align}\label{3.9}
K_{\nu,\gamma}(\cosh\rho)\thicksim&\; e^{\frac{1-n-2\nu}{2}\rho},\;\;\;\;\;\;\;\;\;\;\;\;\;\;\;\;\;\;\;\;\;\;\;\;\rho\rightarrow\infty;\\
\label{3.10}
K_{\nu,\gamma}(\cosh\rho)\thicksim&\left\{
                                             \begin{array}{ll}
                                               1, & \hbox{$\gamma>\frac{n}{2}$;} \\
                                               -\ln\rho, & \hbox{$\gamma=\frac{n}{2}$;} \\
                                               \rho^{2\gamma-n}, & \hbox{$0<\gamma<\frac{n}{2}$,}
                                             \end{array}
                                           \right.
\;\; \rho\rightarrow 0^{+}.
\end{align}
\end{lemma}
\begin{proof}
Using  $F\big(\nu+\frac{n-1}{2},\nu+\frac{1}{2};2\nu+\gamma;
0\big)=1$, we obtain
\begin{align*}
K_{\nu,\gamma}(\cosh\rho)=&C_{\nu,\gamma} \left(\cosh\frac{\rho}{2}\right)^{1-n-2\nu} F\big(\nu+\frac{n-1}{2},\nu+\frac{1}{2};2\nu+\gamma;
(\cosh\frac{\rho}{2})^{-2}\big)\\
\thicksim &\left(\cosh\frac{\rho}{2}\right)^{1-n-2\nu} \\
\thicksim & e^{\frac{1-n-2\nu}{2}\rho},\;\;\;\rho\rightarrow\infty.
\end{align*}
This proves (\ref{3.9}).

Next, we consider the asymptotic estimates of $K_{\nu,\gamma}$  as  $\rho\rightarrow 0^{+}.$ Obviously,
\begin{align}\label{3.11}
K_{\nu,\gamma}(\cosh\rho)\thicksim F\big(\nu+\frac{n-1}{2},\nu+\frac{1}{2};2\nu+\gamma;(\cosh\frac{\rho}{2})^{-2}\big),\;\; \rho\rightarrow 0^{+}.
\end{align}

If $\gamma>\frac{n}{2}$, then by (\ref{2.7}),
\begin{align}\label{3.12}
 F\big(\nu+\frac{n-1}{2},\nu+\frac{1}{2};2\nu+\gamma;1\big)=\frac{\Gamma(2\nu+\gamma)
\Gamma(\gamma-\frac{n}{2})}{\Gamma(\nu+\gamma-\frac{n-1}{2})\Gamma(\nu+\gamma-\frac{1}{2})},\;\;\rho\rightarrow 0^{+}.
\end{align}
Combining (\ref{3.11}) and (\ref{3.12}) yields
\begin{align*}
K_{\nu,\gamma}(\cosh\rho)\thicksim 1,\;\;\rho\rightarrow 0^{+}.
\end{align*}

If  $\gamma=\frac{n}{2}$, then by (\ref{2.6}),
\begin{align*}
&F\big(\nu+\frac{n-1}{2},\nu+\frac{1}{2};2\nu+\gamma;
(\cosh\frac{\rho}{2})^{-2}\big)\\
=&\frac{\Gamma(2\nu+\frac{n}{2})}{\Gamma(\nu+\frac{n-1}{2})\Gamma(\nu+\frac{1}{2})}
\int_{0}^{1}t^{\nu-\frac{1}{2}}(1-t)^{\nu+\frac{n-3}{2}}(1-(\cosh\frac{\rho}{2})^{-2}t)^{-\nu-\frac{n-1}{2}}dt\\
\rightarrow&\frac{\Gamma(2\nu+\frac{n}{2})}{\Gamma(\nu+\frac{n-1}{2})\Gamma(\nu+\frac{1}{2})}
\int_{0}^{1}t^{\nu-\frac{1}{2}}(1-t)^{-1}dt
=\infty,\;\;\;\;\rho\rightarrow 0^{+}.
\end{align*}
Therefore,  by (\ref{2.9}) and L'Hopital's rule, we have
\begin{align*}
&\lim_{\rho\rightarrow 0^{+}}\frac{F\big(\nu+\frac{n-1}{2},\nu+\frac{1}{2};2\nu+\frac{n}{2};(\cosh\frac{\rho}{2})^{-2}\big)}{-\ln\rho}\\
=&2\frac{(\nu+\frac{n-1}{2})(\nu+\frac{1}{2})}{4\nu+n}\lim_{\rho\rightarrow 0^{+}}\frac{\rho\sinh\frac{\rho}{2}}{\cosh^{3}\frac{\rho}{2}}F\big(\nu+\frac{n+1}{2},\nu+\frac{3}{2};2\nu+\frac{n+2}{2};(\cosh\frac{\rho}{2})^{-2}\big)\\
=&2\frac{(\nu+\frac{n-1}{2})(\nu+\frac{1}{2})}{4\nu+n}\lim_{\rho\rightarrow 0^{+}}\frac{\rho}{\cosh\frac{\rho}{2}\sinh\frac{\rho}{2}}F\big(\nu+\frac{1}{2},\nu+\frac{n-1}{2};2\nu+\frac{n+2}{2};(\cosh\frac{\rho}{2})^{-2}\big).
\end{align*}
To get the second equality above, we use (\ref{2.8}).
Therefore, by using (\ref{2.7}), we get
\begin{equation}\label{3.13}
\begin{split}
&\lim_{\rho\rightarrow 0^{+}}\frac{F\big(\nu+\frac{n-1}{2},\nu+\frac{1}{2};2\nu+\frac{n}{2};(\cosh\frac{\rho}{2})^{-2}\big)}{-\ln\rho}\\
 =&4\frac{(\nu+\frac{n-1}{2})(\nu+\frac{1}{2})}{4\nu+n}F\big(\nu+\frac{1}{2},\nu+\frac{n-1}{2};2\nu+\frac{n+2}{2};1\big)\\
 =&4\frac{(\nu+\frac{n-1}{2})(\nu+\frac{1}{2})}{4\nu+n}\frac{\Gamma(2\nu+\frac{n+2}{2})\Gamma(1)}{\Gamma(\nu+\frac{1}{2})\Gamma(\nu+\frac{n-1}{2})}.
\end{split}
\end{equation}
Combining (\ref{3.11}) and (\ref{3.13}) yields
\begin{align*}
K_{\nu, \frac{n}{2}}(\cosh\rho)\thicksim -\ln\rho,\;\;\rho\rightarrow 0^{+}.
\end{align*}

If $0<\gamma<\frac{n}{2}$, then by (\ref{2.8}) and (\ref{2.7}), we have
\begin{align*}
K_{\nu,\gamma}(\cosh\rho)\thicksim&F\big(\nu+\frac{n-1}{2},\nu+\frac{1}{2};2\nu+\gamma;(\cosh\frac{\rho}{2})^{-2}\big)\\
=&\left(\tanh\frac{\rho}{2}\right)^{2\gamma-n}F\big(\nu+\gamma-\frac{n-1}{2},\nu+\gamma-\frac{1}{2};2\nu+\gamma;(\cosh\frac{\rho}{2})^{-2}\big)\\
\thicksim&\rho^{2\gamma-n}F\big(\nu+\gamma-\frac{n-1}{2},\nu+\gamma-\frac{1}{2};2\nu+\gamma;1\big)\\
=&\rho^{2\gamma-n}\frac{\Gamma(2\nu+\gamma)\Gamma(\frac{n}{2}-\gamma)}{\Gamma(\nu+\frac{n-1}{2})\Gamma(\nu+\frac{1}{2})},\;\;\rho\rightarrow 0^{+}.
\end{align*}
The proof of Lemma \ref{lm3.1} is thereby completed.

\end{proof}

Now  we compute  the Helgason-Fourier transform of $K_{\nu,\gamma}(\cosh\rho)$.
\begin{lemma}\label{lm3.2}   Given $\nu>0$ and $\gamma>0$. It holds that
\begin{align*}
\widehat{K_{\nu,\gamma}}(\lambda)=\frac{|\Gamma(\nu+i\lambda)|^{2}}{|\Gamma(\nu+\gamma+i\lambda)|^{2}}.
\end{align*}
\end{lemma}
\begin{proof}
We compute
\begin{align*}\label{}
K_{\nu,\gamma}(\cosh\rho)=&C_{\nu,\gamma}  \left(\cosh\frac{\rho}{2}\right)^{1-n-2\nu}F\big(\nu+\frac{n-1}{2},\nu+\frac{1}{2};2\nu+\gamma;
(\cosh\frac{\rho}{2})^{-2}\big)\\
=&C_{\nu,\gamma}\sum_{k=0}^{\infty}\frac{\Gamma(\nu+\frac{n-1}{2}+k)\Gamma(\nu+\frac{1}{2}+k)\Gamma(2\nu+\gamma)}
{\Gamma(\nu+\frac{n-1}{2})\Gamma(\nu+\frac{1}{2})\Gamma(2\nu+\gamma+k)k!}\left(\cosh\frac{\rho}{2}\right)^{1-n-2\nu-2k}\\
=&\frac{1}{2^{n}\pi^{\frac{n}{2}}\Gamma(\gamma)}\sum_{k=0}^{\infty}\frac{\Gamma(\nu+\frac{n-1}{2}+k)\Gamma(\nu+\frac{1}{2}+k)}
{\Gamma(2\nu+\gamma+k)k!}\left(\cosh\frac{\rho}{2}\right)^{1-n-2\nu-2k}.
\end{align*}
By using  (\ref{3.7}) and (\ref{2.23}), we have, for $\nu>0$,
\begin{align}\nonumber
\widehat{K_{\nu,\gamma}}(\lambda)=&\frac{1}{\Gamma(\gamma)}\sum_{k=0}^{\infty}\frac{\Gamma(\nu+\frac{n-1}{2}+k)\Gamma(\nu+\frac{1}{2}+k)}
{\Gamma(2\nu+\gamma+k)k!} \cdot
\frac{ |\Gamma(\nu+k+i\lambda)|^{2}}{\Gamma(\frac{n-1}{2}+\nu+k)\Gamma(\frac{1}{2}+\nu+k)}\\
\label{a3.18}
=&\frac{1}{\Gamma(\gamma)}
\sum_{k=0}^{\infty}\frac{|\Gamma(\nu+k+i\lambda)|^{2}}
{\Gamma(2\nu+\gamma+k)}  \frac{1}{k!}\\
=&\frac{1}{\Gamma(\gamma)}
\sum_{k=0}^{\infty}\frac{\Gamma(\nu+k+i\lambda)\Gamma(\nu+k-i\lambda)}
{\Gamma(2\nu+\gamma+k)}  \frac{1}{k!} \nonumber\\
=&\frac{1}{\Gamma(\gamma)}\frac{|\Gamma(\nu+i\lambda)|^{2}}
{\Gamma(2\nu+\gamma)} F(\nu+i\lambda,\nu-i\lambda;2\nu+\gamma;1). \nonumber
\end{align}
By  using (\ref{2.7}), we get
\begin{align*}
\widehat{K_{\nu,\gamma}}(\lambda)=&\frac{1}{\Gamma(\gamma)} \frac{|\Gamma(\nu+i\lambda)|^{2}}
{\Gamma(2\nu+\gamma)} \frac{\Gamma(2\nu+\gamma)\Gamma(\gamma)}{|\Gamma(\nu+\gamma+i\lambda)|^{2}}= \frac{|\Gamma(\nu+i\lambda)|^{2}}{|\Gamma(\nu+\gamma+i\lambda)|^{2}}.
\end{align*}
This completes the proof of Lemma \ref{lm3.2}.
\end{proof}

In the limiting case, namely $\nu=0$, we have the following lemma:
\begin{lemma}\label{lm3.3}
Let   $f, g\in C^{\infty}_{0}(\mathbb{B}^{n})$ be complex-valued functions and $\gamma>0$.

(1) The Helgason-Fourier transform of the function $K_{0,\gamma}(\cosh\rho)$ is the function
$\frac{|\Gamma(i\lambda)|^{2}}{|\Gamma(\gamma+i\lambda)|^{2}}$, in the sense that
\begin{align}\label{3.14}
\int_{\mathbb{B}^{n}}K_{0,\gamma}(\cosh\rho)\overline{f(x)}dV=
\int^{+\infty}_{-\infty}\int_{\mathbb{S}^{n-1}} \frac{|\Gamma(i\lambda)|^{2}}{|\Gamma(\gamma+i\lambda)|^{2}} \overline{\widehat{f}(\lambda,\zeta)}|\mathfrak{c}(\lambda)|^{-2}d\lambda d\sigma(\zeta).
\end{align}

(2) The identity
$$(K_{0,\gamma}\ast f)^{\wedge}= \frac{|\Gamma(i\lambda)|^{2}}{|\Gamma(\gamma+i\lambda)|^{2}} \widehat{f}(\lambda,\zeta)$$
holds in the sense that
\begin{align}\label{3.15}
\int_{\mathbb{B}^{n}}(K_{0,\gamma}\ast f)(x)\overline{g(x)}dV=
\int^{+\infty}_{-\infty}\int_{\mathbb{S}^{n-1}} \widehat{f}(\lambda,\zeta)\frac{|\Gamma(i\lambda)|^{2}}{|\Gamma(\gamma+i\lambda)|^{2}} \overline{\widehat{g}(\lambda,\zeta)}|\mathfrak{c}(\lambda)|^{-2}d\lambda d\sigma(\zeta).
\end{align}
\end{lemma}
\begin{proof} Proof of (1): By Lemma \ref{lm3.2} and (\ref{3.2}), we have, for $\nu>0$,
\begin{align}\label{3.16}
\int_{\mathbb{B}^{n}}K_{\nu,\gamma}(\cosh\rho)\overline{f(x)}dV=
\int^{+\infty}_{-\infty}\int_{\mathbb{S}^{n-1}} \frac{|\Gamma(\nu+i\lambda)|^{2}}{|\Gamma(\nu+\gamma+i\lambda)|^{2}} \overline{\widehat{f}(\lambda,\zeta)}|\mathfrak{c}(\lambda)|^{-2}d\lambda d\sigma(\zeta).
\end{align}
Substituting (\ref{3.1}) into (\ref{3.16}), we obtain, for $\nu>0$,
\begin{align*}
\int_{\mathbb{B}^{n}}K_{\nu,\gamma}(\cosh\rho)\overline{f(x)}dV=2(2\pi)^{n}
\int^{+\infty}_{-\infty}\int_{\mathbb{S}^{n-1}} \frac{|\Gamma(\nu+i\lambda)|^{2}}{|\Gamma(\nu+\gamma+i\lambda)|^{2}}
 \frac{|\Gamma(i\lambda+\frac{n-1}{2})|^{2}}{|\Gamma(i\lambda)|^{2}}\overline{\widehat{f}(\lambda,\zeta)}d\lambda d\sigma(\zeta).
\end{align*}
Passing to the limit as $\nu\rightarrow 0^{+}$ yields
\begin{align*}
\int_{\mathbb{B}^{n}}K_{0,\gamma}(\cosh\rho)\overline{f(x)}dV=&2(2\pi)^{n}
\int^{+\infty}_{-\infty}\int_{\mathbb{S}^{n-1}} \frac{|\Gamma(i\lambda+\frac{n-1}{2})|^{2}}{|\Gamma(\gamma+i\lambda)|^{2}}
 \overline{\widehat{f}(\lambda,\zeta)}d\lambda d\sigma(\zeta)\\
=&\int^{+\infty}_{-\infty}\int_{\mathbb{S}^{n-1}} \frac{|\Gamma(i\lambda)|^{2}}{|\Gamma(\gamma+i\lambda)|^{2}} \overline{\widehat{f}(\lambda,\zeta)}|\mathfrak{c}(\lambda)|^{-2}d\lambda d\sigma(\zeta).
\end{align*}
This proves (\ref{3.14}).

\medskip
Proof of
(2): The proof of (\ref{3.15}) is similar. By Lemma \ref{lm3.2} and (\ref{3.2}), we have, for $\nu>0$,
\begin{align*}
&\int_{\mathbb{B}^{n}}(K_{\nu,\gamma}\ast f)(x)\overline{g(x)}dV\\
=&
\int^{+\infty}_{-\infty}\int_{\mathbb{S}^{n-1}}\widehat{f}(\lambda,\zeta)\frac{|\Gamma(\nu+i\lambda)|^{2}}{|\Gamma(\gamma+\nu+i\lambda)|^{2}}
\overline{\widehat{g}(\lambda,\zeta)}|\mathfrak{c}(\lambda)|^{-2}d\lambda d\sigma(\zeta).
\end{align*}
Passing to the limit as $\nu\rightarrow 0^{+}$ yields (\ref{3.15}).
\end{proof}

As an application of  Lemmas \ref{lm3.2} and  \ref{lm3.3}, we have the following result (see the work by the authors \cite{ly4} for $\gamma\in \mathbb{N}\setminus\{0\}$):
\begin{theorem}\label{co3.4}
Let $\gamma>0$ and $\nu\geq 0$.  The Green function of $\frac{|\Gamma(\gamma+\nu+i\sqrt{-\Delta_{\mathbb{H}}-\frac{(n-1)^{2}}{4}})|^{2}}
{|\Gamma(\nu+i\sqrt{-\Delta_{\mathbb{H}}-\frac{(n-1)^{2}}{4}})|^{2}}$ is $K_{\nu, \gamma}(\cosh\rho)$.
\end{theorem}

Before stating and  proving the next lemma, we need the following elementary inequality:
\begin{lemma}\label{lm3.4}
Let $0\leq a\leq b\leq c\leq d$. If $a+d=b+c$, then $ad\leq bc$.
\end{lemma}
\begin{proof}
Set $s=a+d=b+c$. We have
\begin{align*}
bc-ad=b(s-b)-a(s-a)=(b-a)(s-a-b)=(b-a)(d-b)\geq0.
\end{align*}
The desired result follows.
\end{proof}

Now  we  give the Helgason-Fourier transform of
\begin{align}\label{a3.18}
H_{\nu,\gamma}(\cosh\rho)=\left(\cosh\frac{\rho}{2}\right)^{1-2\gamma-2\nu} \left(\sinh\frac{\rho}{2}\right)^{2\gamma-n}.
\end{align}

\begin{lemma}\label{lm3.5}
Let $\nu>0$ and  $\frac{n-1}{2}\leq\gamma<\frac{n}{2}$.
The Helgason-Fourier transform of $H_{\nu,\gamma}(\cosh\rho)$ is
\begin{align*}
\widehat{H_{\nu,\gamma}}(\lambda)=\sum_{k=0}^{\infty}
\frac{|\Gamma(\nu+k+i\lambda)|^{2}(\frac{n}{2}-\gamma)_{k} }{\Gamma(\frac{n-1}{2}+\nu+k)\Gamma(\frac{1}{2}+\nu+k)k!}.
\end{align*}
Furthermore, the following holds:
\begin{align*}
0\leq \widehat{H_{\nu,\gamma}}(\lambda)\leq \frac{\Gamma(\gamma+2\nu)\Gamma(\gamma)}{\Gamma(\nu+\frac{n-1}{2})\Gamma(\nu+\frac{1}{2})}\widehat{K_{\nu,\gamma}}(\lambda).
\end{align*}
\end{lemma}
\begin{proof}
We have
\begin{align*}
H_{\nu,\gamma}(\cosh\rho)=& \left(\cosh\frac{\rho}{2}\right)^{1-n-2\nu}\left(1-\frac{1}{(\cosh\frac{\rho}{2})^{2}}\right)^{\gamma-\frac{n}{2}}\\
=& \left(\cosh\frac{\rho}{2}\right)^{1-n-2\nu}\sum_{k=0}^{\infty}\frac{(\frac{n}{2}-\gamma)_{k}}{k!}
\left(\cosh\frac{\rho}{2}\right)^{-2k}\\
=& \sum_{k=0}^{\infty}\frac{(\frac{n}{2}-\gamma)_{k}}{k!}
\left(\cosh\frac{\rho}{2}\right)^{1-n-2\nu-2k}.
\end{align*}
By using  (\ref{3.7}) and (\ref{2.23}), we have
\begin{align}\label{3.18}
\widehat{H_{\nu,\gamma}}(\lambda)= \sum_{k=0}^{\infty}\frac{(\frac{n}{2}-\gamma)_{k}}{k!}
\frac{ |\Gamma(\nu+k+i\lambda)|^{2}}{\Gamma(\frac{n-1}{2}+\nu+k)\Gamma(\frac{1}{2}+\nu+k)}.
\end{align}
We claim
\begin{align}\label{3.19}
 (n/2-\gamma)_{k}\leq \frac{(\nu+\frac{n-1}{2})_{k}(\nu+\frac{1}{2})_{k}}{(\gamma+2\nu)_{k}},\;\; k=0,1,2,\cdots.
\end{align}
In fact, by Lemma \ref{lm3.4}, we have
\begin{align*}
(n/2-\gamma)_{k}(\gamma+2\nu)_{k}=&\prod_{m=0}^{k-1}(n/2-\gamma+m)(\gamma+2\nu+m)\\
\leq& \prod_{m=0}^{k-1}\left(\nu+\frac{1}{2}+m\right)\left(\nu+\frac{n-1}{2}+m\right)\\
=&\left(\nu+\frac{1}{2}\right)_{k}\left(\nu+\frac{n-1}{2}\right)_{k}.
\end{align*}
This proves the claim.

Substituting (\ref{3.19}) into (\ref{3.18}) and using (\ref{a3.18}), we obtain
\begin{align*}
0\leq \widehat{H_{\nu,\gamma}}(\lambda)\leq& \sum_{k=0}^{\infty}\frac{(\nu+\frac{1}{2})_{k}(\nu+\frac{n-1}{2})_{k}}{(\gamma+2\nu)_{k}k!}
\frac{ |\Gamma(\nu+k+i\lambda)|^{2}}{\Gamma(\frac{n-1}{2}+\nu+k)\Gamma(\frac{1}{2}+\nu+k)}\\
=&\frac{\Gamma(\gamma+2\nu)}{\Gamma(\nu+\frac{n-1}{2})\Gamma(\nu+\frac{1}{2})}\sum_{k=0}^{\infty}\frac{|\Gamma(\nu+k+i\lambda)|^{2}}
{\Gamma(2\nu+\gamma+k)}  \frac{1}{k!}\\
=&\frac{\Gamma(\gamma+2\nu)\Gamma(\gamma)}{\Gamma(\nu+\frac{n-1}{2})\Gamma(\nu+\frac{1}{2})}\widehat{K_{\nu,\gamma}}(\lambda).
\end{align*}
The proof of the lemma is thereby completed.
\end{proof}

With the same arguments as in Lemma \ref{lm3.3}, we also have the following lemma. Since the proof is very similar to that of Lemma \ref{lm3.5}, we omit its details.
\begin{lemma}\label{lm3.6}
Let $\frac{n-1}{2}\leq \gamma<\frac{n}{2}$.
 The Helgason-Fourier transform of  $H_{0,\gamma}(\cosh\rho)$ is
 \begin{align*}
\widehat{H_{0,\gamma}}(\lambda)=\sum_{k=0}^{\infty}
\frac{|\Gamma(k+i\lambda)|^{2}(\frac{n}{2}-\gamma)_{k} }{\Gamma(\frac{n-1}{2}+k)\Gamma(\frac{1}{2}+k)k!},
 \end{align*}
 in the sense that
\begin{align*}
\int_{\mathbb{B}^{n}}H_{0,\gamma}(\cosh\rho)\overline{f(x)}dV=
\int^{+\infty}_{-\infty}\int_{\mathbb{S}^{n-1}} \sum_{k=0}^{\infty}
\frac{|\Gamma(k+i\lambda)|^{2}(\frac{n}{2}-\gamma)_{k} }{\Gamma(\frac{n-1}{2}+k)\Gamma(\frac{1}{2}+k)k!}
\overline{\widehat{f}(\lambda,\zeta)}|\mathfrak{c}(\lambda)|^{-2}d\lambda d\sigma(\zeta),
\end{align*}
whenever $f\in C^{\infty}_{0}(\mathbb{B}^{n})$.  The identity
$$(H_{0,\gamma}\ast f)^{\wedge}= \sum_{k=0}^{\infty}
\frac{|\Gamma(k+i\lambda)|^{2}(\frac{n}{2}-\gamma)_{k} }{\Gamma(\frac{n-1}{2}+k)\Gamma(\frac{1}{2}+k)k!} \widehat{f}(\lambda,\zeta)$$
holds in the sense that
\begin{align*}
\int_{\mathbb{B}^{n}}(H_{0,\gamma}\ast f)(x)\overline{g(x)}dV=
\int^{+\infty}_{-\infty}\int_{\mathbb{S}^{n-1}} \widehat{f}(\lambda,\zeta)
\sum_{k=0}^{\infty}
\frac{|\Gamma(k+i\lambda)|^{2}(\frac{n}{2}-\gamma)_{k} }{\Gamma(\frac{n-1}{2}+k)\Gamma(\frac{1}{2}+k)k!}
 \overline{\widehat{g}(\lambda,\zeta)}|\mathfrak{c}(\lambda)|^{-2}d\lambda d\sigma(\zeta),
\end{align*}
whenever  $f, g\in C^{\infty}_{0}(\mathbb{B}^{n})$.
\end{lemma}

\begin{lemma}\label{lm3.8a}
Let $f$ be measurable.  If $|f(x)|\lesssim (\cosh\frac{\rho(x)}{2})^{-\gamma}$ with  $\gamma>n-1$, then the  Helgason-Fourier transform
of $f$ exists. Moreover, we have, for $\nu>0$,
\begin{align}\label{a3.21}
\widehat{(K_{\nu,\gamma}\ast f)}(\lambda,\zeta)=\widehat{K_{\nu,\gamma}}(\lambda)\widehat{f}(\lambda,\zeta).
\end{align}
\end{lemma}
\begin{proof}
Without loss of generality, we assume $\gamma<n-1+\nu$. Since $|f(x)|\lesssim (\cosh\frac{\rho(x)}{2})^{-\gamma}$,
we have
\begin{align*}
\int_{\mathbb{B}^{n}} |f(x)e_{-\lambda,\zeta}(x)|dV\lesssim&\int_{\mathbb{B}^{n}}\big(\cosh\frac{\rho(x)}{2}\big)^{-\gamma} \left(\frac{\sqrt{1-|x|^{2}}}{|x-\zeta|}\right)^{n-1}dV\\
=&[\big(\cosh\frac{\rho}{2}\big)^{-\gamma} ]^{\wedge}(0),
\end{align*}
 where
$[\big(\cosh\frac{\rho}{2}\big)^{-\gamma} ]^{\wedge}(\lambda)$ is the Helgason-Fourier transform
of $\big(\cosh\frac{\rho}{2}\big)^{-\gamma} $. Therefore, by using (\ref{2.23}) and (\ref{3.7}), we obtain
\begin{align*}
\int_{\mathbb{B}^{n}} |f(x)e_{-\lambda,\zeta}(x)|dV<\infty,
\end{align*}
which implies the existence of  $\widehat{f}(\lambda,\zeta)$.

Before the proof of  (\ref{a3.21}),  we first give  the asymptotic estimates of $K_{\nu,\gamma}\ast (\cosh\frac{\rho}{2})^{-\gamma}$.
We write
\begin{align}\nonumber
&\left(K_{\nu,\gamma}\ast \big(\cosh\frac{\rho}{2}\big)^{-\gamma}\right)(x)\\
\label{a3.23}
=&\left(\int_{\{y||y|<\frac{1}{2}\}}+\int_{\{y|\frac{1}{2}\leq |y|<1\}}\right)K_{\nu,\gamma}(\cosh\rho(y))\big(\cosh\frac{\rho(x,y)}{2}\big)^{-\gamma}dV_{y}.
\end{align}
By Lemma \ref{lm3.1} and (\ref{2.4}), we have
\begin{align}\label{a3.24}
\int_{\{y||y|<\frac{1}{2}\}}K_{\nu,\gamma}(\cosh\rho(y))\big(\cosh\frac{\rho(x,y)}{2}\big)^{-\gamma}dV_{y}\lesssim
\int_{\{y||y|<\frac{1}{2}\}}K_{\nu,\gamma}(\cosh\rho(y))dV_{y}\lesssim1.
\end{align}
and
\begin{align}\nonumber
&\int_{\{y|\frac{1}{2}\leq |y|<1\}}K_{\nu,\gamma}(\cosh\rho(y))\big(\cosh\frac{\rho(x,y)}{2}\big)^{-\gamma}dV_{y}\\
\nonumber
\lesssim&              \int_{\mathbb{B}^{n}}\big(\cosh\frac{\rho(y)}{2}\big)^{1-n-\nu}  \big(\cosh\frac{\rho(x,y)}{2}\big)^{-\gamma}dV_{y}
\\
\label{a3.25}
=&   2^{n}(1-|x|^{2})^{\frac{\gamma}{2}}\int_{\mathbb{B}^{n}}\frac{(1-|y|^{2})^{\frac{\gamma+\nu-n-1}{2}}}{(1-2x\cdot y+|x|^{2}|y|^{2})^{\frac{\gamma}{2}}}dy.
\end{align}
Noticing  that
\begin{align*}
\int_{\mathbb{B}^{n}}\frac{(1-|y|^{2})^{\frac{\gamma+\nu-n-1}{2}}}{(1-2x\cdot y+|x|^{2}|y|^{2})^{\frac{\gamma}{2}}}dy
\end{align*}
is bounded in $\mathbb{B}^{n}$ when $n-1<\gamma<n-1+\nu$ (see \cite{lius}, Proposition 2.2), we obtain, by using (\ref{a3.25}),
\begin{align}\label{a3.26}
\int_{\{y|\frac{1}{2}\leq |y|<1\}}K_{\nu,\gamma}(\cosh\rho(y))\big(\cosh\frac{\rho(x,y)}{2}\big)^{-\gamma}dV_{y}\lesssim
(1-|x|^{2})^{\frac{\gamma}{2}}=\big(\cosh\frac{\rho(x)}{2}\big)^{-\gamma}.
\end{align}
Substituting (\ref{a3.24}) and (\ref{a3.26}) into (\ref{a3.23}), we get
\begin{align*}
\left(K_{\nu,\gamma}\ast \big(\cosh\frac{\rho}{2}\big)^{-\gamma}\right)(x)\lesssim \big(\cosh\frac{\rho(x)}{2}\big)^{-\gamma}.
\end{align*}
Therefore,
\begin{align*}
|(K_{\nu,\gamma}\ast f)(x)|\lesssim \left(K_{\nu,\gamma}\ast \big(\cosh\frac{\rho}{2}\big)^{-\gamma}\right)(x)\lesssim \big(\cosh\frac{\rho(x)}{2}\big)^{-\gamma},
\end{align*}
which implies the existence of Helgason-Fourier transform of $K_{\nu,\gamma}\ast f$. The rest of the proof of
(\ref{a3.21}) is  similar to that given in \cite{liup}, Proposition 6.3 and we omit it.

\end{proof}

Before proving the next lemma, we need the following  lemma given   by Beckner (see \cite{be}) and the
Hardy-Littlewood-Sobolev inequality on $\mathbb{H}^{n}$ (see the first proof given by Beckner  on half space model in  [\cite{be1} Theorem 11] or by the authors \cite{LuYang3}, Theorem 4.1 for ball model of its equivalent form).
\begin{lemma}[\cite{be}]\label{lm3.7}
Let $K$ and $\Lambda$ be densely defined, positive-definite, self-adjoint operators acting on functions
defined on a $\sigma$-finite measure space $N$ and satisfying the relation
$K\Lambda=\Lambda K=1$. Then the following two inequalities are equivalent:
\begin{align*}
&\|K f\|_{L^{p'}(M)}\leq C_{p}\|f\|_{L^{p}(M)},\\
&\|g\|_{L^{p'}(M)}\leq \sqrt{C_{p}}\|\Lambda^{1/2}g\|_{L^{2}(M)}.
\end{align*}
Here $1 < p < 2$ and $1/p+1/p' = 1$. Extremal functions for one inequality will determine extremal functions
for the other inequality if the operator forms are well-defined.
\end{lemma}

\begin{theorem}[\cite{be,LuYang3}]\label{th3.8}
Let $0<\lambda<n$ and $p=\frac{2n}{2n-\lambda}$. Then for $f,g\in  C_{0}^{\infty}(\mathbb{B}^{n})$,
\begin{equation}\label{3.21}
\left|\int_{\mathbb{B}^{n}}\int_{\mathbb{B}^{n}}\frac{f(x)g(y)}{\left(2\sinh\frac{\rho(T_{y}(x))}{2}\right)^{\lambda}}dV_{x}dV_{y}\right|\leq C_{n,\lambda}\|f\|_{p}\|g\|_{p},
\end{equation}
or equivalently,
\begin{align*}
\left\|\left(2\sinh\frac{\rho}{2}\right)^{-\lambda}\ast f\right\|_{p'}\leq C_{n,\lambda}\|f\|_{p},
\end{align*}
 where
 \begin{equation}\label{3.22}
C_{n,\lambda}=\pi^{\lambda/2}\frac{\Gamma(n/2-\lambda/2)}{\Gamma(n-\lambda/2)}\left(\frac{\Gamma(n/2)}{\Gamma(n)}\right)^{-1+\lambda/n}
\end{equation}
 is the best constant for the classical Hardy-Littlewood-Sobolev constant on $\mathbb{R}^{n}$.
Furthermore, the constant $C_{n,\lambda}$
 is sharp for the inequality (\ref{3.21}) and there is no nonzero extremal function for the inequality (\ref{3.21}).
\end{theorem}

Now we can give the following  sharp Sobolev inequality on hyperbolic space with best constant $S_{n,\gamma}$:
\begin{theorem} \label{th3.10} Let $n\geq2$ and $0<\frac{n-1}{2}-\nu\leq \gamma<\frac{n}{2}$. Then for $u\in C_{0}^{\infty}(\mathbb{B}^{n})$, we have
\begin{align*}
\int_{\mathbb{B}^{n}}u\frac{|\Gamma(\nu+\gamma+i\sqrt{-\Delta_{\mathbb{H}}-\frac{(n-1)^{2}}{4}})|^{2}}
{|\Gamma(\nu+i\sqrt{-\Delta_{\mathbb{H}}-\frac{(n-1)^{2}}{4}})|^{2}}u dV\geq S_{n,\gamma}\left(\int_{\mathbb{B}^{n}}|u|^{\frac{2n}{n-2\gamma}}dV\right)^{\frac{n-2\gamma}{n}}.
\end{align*}
Furthermore, the inequality is  strict for nonzero $u$'s.
\end{theorem}
\begin{proof}
Since,  for $0\leq t\leq 1$,
\begin{align*}
&\frac{d}{dt}F\big(\gamma+\nu-\frac{n-1}{2},\gamma+\nu-\frac{1}{2};2\nu+\gamma;
t\big)\\
=&\frac{(\gamma+\nu-\frac{n-1}{2})(\gamma+\nu-\frac{1}{2})}{\gamma+2\nu}F\big(\gamma+\nu-\frac{n-3}{2},\gamma+\nu+\frac{1}{2};2\nu+\gamma+1;
t\big)\\
\geq&0,
\end{align*}
the function
 $F\big(\gamma+\nu-\frac{n-1}{2},\gamma+\nu-\frac{1}{2};2\nu+\gamma;
t\big)$ is increasing for $t\in[0,1]$.
Therefore, by using  (\ref{2.8}), we get
\begin{align*}
K_{\nu,\gamma}(\cosh\rho)= &\frac{\Gamma(\frac{n-1}{2}+\nu)\Gamma(\nu+\frac{1}{2})}{2^{n}\pi^{\frac{n}{2}}
\Gamma(\gamma)\Gamma(2\nu+\gamma)}\left(\cosh\frac{\rho}{2}\right)^{1-n-2\nu}\times\\
& F\big(\nu+\frac{n-1}{2},\nu+\frac{1}{2};2\nu+\gamma;
(\cosh\frac{\rho}{2})^{-2}\big)\\
=&\frac{\Gamma(\frac{n-1}{2}+\nu)\Gamma(\nu+\frac{1}{2})}{2^{n}\pi^{\frac{n}{2}}
\Gamma(\gamma)\Gamma(2\nu+\gamma)}\left(\cosh\frac{\rho}{2}\right)^{1-n-2\nu}\times\\
&\left(\tanh\frac{\rho}{2}\right)^{2\gamma-n}F\big(\gamma+\nu-\frac{n-1}{2},\gamma+\nu-\frac{1}{2};2\nu+\gamma;
(\cosh\frac{\rho}{2})^{-2}\big)\\
\leq&\frac{\Gamma(\frac{n-1}{2}+\nu)\Gamma(\nu+\frac{1}{2})}{2^{n}\pi^{\frac{n}{2}}
\Gamma(\gamma)\Gamma(2\nu+\gamma)}\left(\cosh\frac{\rho}{2}\right)^{1-2\gamma-2\nu}\left(\sinh\frac{\rho}{2}\right)^{2\gamma-n}\times\\
&F\big(\gamma+\nu-\frac{n-1}{2},\gamma+\nu-\frac{1}{2};2\nu+\gamma;
1\big)\\
=&\frac{\Gamma(\frac{n}{2}-\gamma)}{2^{n}\pi^{\frac{n}{2}}
\Gamma(\gamma)}\left(\cosh\frac{\rho}{2}\right)^{1-2\gamma-2\nu}\left(\sinh\frac{\rho}{2}\right)^{2\gamma-n}.
\end{align*}
To get the last equality, we use (\ref{2.7}).
By Theorem \ref{th3.8}, we have
\begin{align}\label{3.17}
 \|K_{\nu,\gamma}\ast u\|_{L^{\frac{2n}{n-2\gamma}}(\mathbb{B}^{n})}\leq \frac{1}{S_{n,\gamma}}\|u\|_{L^{\frac{2n}{n+2\gamma}}(\mathbb{B}^{n})}
\end{align}
and the inequality is  strict for nonzero $u$'s.
The desired result follows by combining (\ref{3.17}) and
Lemma \ref{lm3.7}. This completes the proof of Theorem  \ref{th3.10}.
\end{proof}

\section{Fractional  GJMS
 operators  on hyperbolic space and the proofs of Theorems \ref{th1.3} and \ref{th1.5}}\label{Section5}

Firstly, we briefly review  the definition of the fractional GJMS operator via scattering theory.  The following material is based on \cite{FeffermanGr,grz}.

Let $(X^{n+1}, g_{+})$ be a conformally compact Einstein manifold  of dimension $n + 1$ with boundary $M$.
A  function $r\in C^{\infty}(X^{n+1})$  is called a \emph{defining function} if $r^{-1}(\{0\})=M^{n}$,
$dr\neq0$ along $M$, and the metric $g:=r^{2}g_{+}$ extends to a smooth metric on
$\overline{X}^{n+1}$. Given  a representative $[h]$
on the conformal boundary $M$, there is a unique defining function $r$ such that
$g_{+}=r^{-2}(dr^{2}+h_{r})$ on $M\times(0,\delta)$, where $h_{r}$
is a one-parameter
family of metrics on $M$ satisfying $h_{0}=h$.
In particular, if $(M,h)$ is Einstein with $R_{ij}=2\lambda (n-1)h_{ij}$, then (see \cite{FeffermanGr}, Page 74, (7.13))
\begin{align}\label{4.1}
g_{+}=\frac{dr^{2}+(1-\frac{1}{2}\lambda r^{2})^{2}h}{r^{2}}.
\end{align}

Given $f\in C^{\infty}(M)$.
It has been shown (see Mazzeo-Melrose \cite{ma} and Graham-Zworski  \cite{grz}) that the Poisson equation
\begin{equation}\label{4.2}
\begin{split}
-\Delta_{g_{+}}u-s(n-s)u=0
\end{split}
\end{equation}
has a unique solution of the form
\begin{equation}\label{4.3}
\begin{split}
u=F r^{n-s} +H r^{s},\;\; F,H\in C^{\infty}(X),\; F|_{ r=0}=f,
\end{split}
\end{equation}
where $s\in\mathbb{C}$ and $s(n-s)$ does not belong to the pure point spectrum of $-\Delta_{g_{+}}$.
The scattering
operator on $M$ is defined as $S(s)f=H|_{M}$. If $\textrm{Re}(s)>\frac{n}{2}$, then the scattering
operator is a meromorphic family of pseudo-differential
operators. Graham and Zworski \cite{grz} defined the fractional GJMS operator $P_{\gamma}$, $\gamma=s-\frac{n}{2}\in(0,\frac{n}{2})\setminus\mathbb{N}$,
 as follows
\begin{equation}\label{4.4}
\begin{split}
P_{\gamma}f:=d_{\gamma}S\left(\frac{n}{2}+\gamma\right)f,\;\;d_{\gamma}=2^{2\gamma}\frac{\Gamma(\gamma)}{\Gamma(-\gamma)}.
\end{split}
\end{equation}
It has been  shown by Graham and Zworski \cite{grz}  that the
principal symbol of $P_{\gamma}$ is  exactly the principal
symbol of the fractional Laplacian $(-\Delta)^{\gamma}$ and
satisfy an important conformal covariance property (\ref{a1.9}).

We remark that if $\gamma\in\mathbb{N}\setminus\{0\}$, then $P_{\gamma}$ is nothing but the GJMS operator on $M$ (see \cite{gr}).
In the particular case that $(X^{n+1},g_{+})=(\mathbb{H}^{n+1}, g_{\mathbb{H}})$,  $P_{\gamma}$
is nothing but the fractional Laplacian $(-\Delta)^{\gamma}$ on $\mathbb{R}^{n}$ (see e.g.  Chang-Gonz\'alez \cite{chg}). For the case $(M,h)=(\mathbb{B}^{n},g_{\mathbb{B}})$,
we refer to Ao et al.  \cite{ao} for the construction of
conformal
fractional Laplacian $P_{\gamma}$ on $\mathbb{S}^{n-k-1}\times \mathbb{H}^{k+1}$ $(1\leq k<n-1)$, which are  conformal to the fractional Laplacians
$(-\Delta)^{\gamma}$ on $\mathbb{R}^{n}\setminus \mathbb{R}^{k}$ when $k<n-1$.

The main result is the following lemma:
\begin{lemma}\label{lm4.1}
Let $0<\gamma<\frac{n}{2}$. The Helgason-Fourier transform of fractional GJMS operator $P_{\gamma}$ is
\begin{align*}
\widehat{P_{\gamma}f}(\lambda,\zeta)=2^{2\gamma}\frac{|\Gamma(\frac{3+2\gamma}{4}+\frac{i}{2}\lambda|^{2}}
{|\Gamma(\frac{3-2\gamma}{4}+\frac{i}{2}\lambda)|^{2}}\widehat{f}(\lambda,\zeta),\;\;f\in C_{0}^{\infty}(\mathbb{B}^{n}).
\end{align*}
\end{lemma}
\begin{proof}
Since $(\mathbb{B}^{n},g_{\mathbb{B}})$  has a constant negative curvature $-1$, we have, by (\ref{4.1}),
\begin{align*}
 g_{+}=\frac{dr^{2}+(1+\frac{1}{4} r^{2})^{2}g_{\mathbb{B}}}{r^{2}},\;\; r\in(0,\infty).
\end{align*}
Under the change of variable $r=2e^{-t}$, the metric $g_{+}$ becomes
\begin{align*}
 g_{+}=dt^{2}+(\cosh t)^{2} g_{\mathbb{B}},\;\; t\in (-\infty,\infty).
\end{align*}
The corresponding  Poisson equation (\ref{4.2}) is
\begin{align*}
\frac{\partial^{2}}{\partial t^{2}}u+n\tanh t \frac{\partial }{\partial t}u+\frac{1}{\cosh^{2}t}\Delta_{\mathbb{H}}u+\big(\frac{n^{2}}{4}-\gamma^{2}\big)u=0.
\end{align*}
Taking the Helgason-Fourier transform in the above equality, we get
\begin{align*}
\frac{\partial^{2}}{\partial t^{2}}\widehat{u}+n\tanh t \frac{\partial }{\partial t}\widehat{u}-\frac{1}{\cosh^{2}t}\left[
\frac{(n-1)^{2}}{4}+\lambda^{2}\right]\widehat{u}+\big(\frac{n^{2}}{4}-\gamma^{2}\big)\widehat{u}=0.
\end{align*}
Using the change of variable $\tau=\tanh t$, we get
\begin{align*}
(1-\tau^{2})\frac{\partial^{2}}{\partial \tau^{2}}\widehat{u}+(n-2)\tau  \frac{\partial }{\partial \tau }\widehat{u}+
\left[\left(\frac{n^{2}}{4}-\gamma^{2}\right)\frac{1}{1-\tau^{2}}-\frac{(n-1)^{2}}{4}-\lambda^{2}\right]\widehat{u}=0,\; -1<\tau<1.
\end{align*}
Therefore,  we have $\widehat{u}(r,\lambda,\zeta)=\varphi(r)\widehat{f}(\lambda,\zeta)$, where $\varphi(r)$ satisfies
\begin{align}\label{4.5}
(1-\tau^{2})\frac{d^{2}}{d \tau^{2}}\varphi+(n-2)\tau  \frac{d }{d \tau }\varphi+
\left[\left(\frac{n^{2}}{4}-\gamma^{2}\right)\frac{1}{1-\tau^{2}}-\frac{(n-1)^{2}}{4}-\lambda^{2}\right]\varphi=0.
\end{align}
Noticing that the ordinary differential equation above is invariant  with respect to the transformation $\tau\rightarrow-\tau$, the solution of (\ref{4.5}) is (see \cite{ao}, (3.30))
\begin{align*}
\varphi(r)=A(1-\tau^{2})^{\frac{n-2\gamma}{4}}\tau F(a,b;c;\tau^{2})+B(1-\tau^{2})^{-\frac{n-2\gamma}{4}}|\tau|^{-\frac{n-1}{2}} F(a',b';c';\tau^{2}),\;A,B\in\mathbb{R},
\end{align*}
where
\begin{align*}
 a=&\frac{3}{4}-\frac{\gamma}{2}+\frac{i}{2}\lambda,\; b=\frac{3}{4}-\frac{\gamma}{2}-\frac{i}{2}\lambda,\; c=\frac{3}{2},\; a'=\frac{1}{4}-\frac{\gamma}{2}+\frac{i}{2}\lambda,\;
b'=\frac{1}{4}-\frac{\gamma}{2}-\frac{i}{2}\lambda,\; c'=\frac{1}{2}.
\end{align*}
The regularity of $\varphi(r)$  at $r=2$ (i.e. $\tau=0$)  implies $B=0$ .
Therefore, by (\ref{2.12}) and (\ref{2.10}), we get
\begin{align*}
\varphi(r)=&A\frac{\Gamma(c)\Gamma(c-a-b)}{\Gamma(c-a)\Gamma(c-b)}(1-\tau^{2})^{\frac{n-2\gamma}{4}}\tau  F(a,b;a+b-c+1;1-\tau^{2})+\\
&A \frac{\Gamma(c)\Gamma(a+b-c)}{\Gamma(a)\Gamma(b)}(1-\tau^{2})^{\frac{n-2\gamma}{4}+c-a-b}\tau F(c-a,c-b;c-a-b+1;1-\tau^{2})\\
=&A\frac{\Gamma(\frac{3}{2})\Gamma(\gamma)}{|\Gamma(\frac{3+2\gamma}{4}+\frac{i}{2}\lambda)|^{2}}(1-\tau^{2})^{\frac{n-2\gamma}{4}}\tau  F(a,b;a+b-c+1;1-\tau^{2})+\\
&A\frac{\Gamma(\frac{3}{2})\Gamma(-\gamma)}{|\Gamma(\frac{3-2\gamma}{4}+\frac{i}{2}\lambda)|^{2}}(1-\tau^{2})^{\frac{n+2\gamma}{4}}\tau
F(c-a,c-b;c-a-b+1;1-\tau^{2}).
\end{align*}
By using the asymptotic estimation $\tau=\tanh t=\frac{4-r^{2}}{4+r^{2}}=1-\frac{1}{2}r^{2}+O(r^{4}),\; r\searrow0$, we have
\begin{align*}
\varphi(r)=A\frac{\Gamma(\frac{3}{2})\Gamma(\gamma)}{|\Gamma(\frac{3+2\gamma}{4}+\frac{i}{2}\lambda)|^{2}}
\left[r^{\frac{n}{2}-\gamma}\left(1+O(r)\right)+d_{\gamma}2^{2\gamma}\frac{|\Gamma(\frac{3+2\gamma}{4}+\frac{i}{2}\lambda|^{2}}
{|\Gamma(\frac{3-2\gamma}{4}+\frac{i}{2}\lambda)|^{2}}r^{\frac{n}{2}+\gamma}(1+O(r))\right],\;r\searrow0.
\end{align*}
On the other hand, using (\ref{4.3}) and $\widehat{u}(r,\lambda,\zeta)=\varphi(r)\widehat{f}(\lambda,\zeta)$, one has
\begin{align*}
A\frac{\Gamma(\frac{3}{2})\Gamma(\gamma)}{|\Gamma(\frac{3+2\gamma}{4}+\frac{i}{2}\lambda)|^{2}}=1.
\end{align*}
Therefore,
\begin{align*}
 \widehat{u}(r,\lambda,\zeta)=\left[r^{\frac{n}{2}-\gamma}\left(1+O(r)\right)+d_{\gamma}2^{2\gamma}\frac{|\Gamma(\frac{3+2\gamma}{4}+\frac{i}{2}\lambda|^{2}}
{|\Gamma(\frac{3-2\gamma}{4}+\frac{i}{2}\lambda)|^{2}}r^{\frac{n}{2}+\gamma}(1+O(r))\right]
\widehat{f}(\lambda,\zeta),\;r\searrow0.
\end{align*}
Then the desired result follows. This completes the proof of Lemma \ref{lm4.1}.
\end{proof}

\begin{lemma}\label{lm4.2}
It holds that
\begin{align*}
2^{2\gamma}\frac{|\Gamma(\frac{3+2\gamma}{4}+\frac{i}{2}\lambda)|^{2}}
{|\Gamma(\frac{3-2\gamma}{4}+\frac{i}{2}\lambda)|^{2}}=\frac{|\Gamma(\gamma+\frac{1}{2}+i\lambda)|^{2}}
{|\Gamma(\frac{1}{2}+i\lambda)|^{2}}+\frac{\sin\gamma\pi}{\pi}|\Gamma(\gamma+1/2+i\lambda)|^{2}.
\end{align*}
In particular, for $\lambda=0$, we have
\begin{align*}
2^{2\gamma}\frac{|\Gamma(\frac{3+2\gamma}{4})|^{2}}
{|\Gamma(\frac{3-2\gamma}{4})|^{2}}=\frac{|\Gamma(\gamma+\frac{1}{2})|^{2}}
{|\Gamma(\frac{1}{2})|^{2}}+\frac{\sin\gamma\pi}{\pi}|\Gamma(\gamma+1/2)|^{2}.
\end{align*}
\end{lemma}
\begin{proof}
We have, by using  (\ref{2.10}),
\begin{align}\nonumber
2^{2\gamma}\frac{|\Gamma(\frac{3+2\gamma}{4}+\frac{i}{2}\lambda)|^{2}}
{|\Gamma(\frac{3-2\gamma}{4}+\frac{i}{2}\lambda)|^{2}}=&2^{2\gamma}\frac{|\Gamma(\frac{3+2\gamma}{4}+\frac{i}{2}\lambda)
\Gamma\left(\frac{2\gamma+1}{4}-\frac{i}{2}\lambda\right)|^{2}}
{|\Gamma(\frac{3-2\gamma}{4}+\frac{i}{2}\lambda)\left(\frac{2\gamma+1}{4}-\frac{i}{2}\lambda\right)|^{2}}\\
\label{4.6}
=&2^{2\gamma}\frac{|\Gamma(\frac{3+2\gamma}{4}+\frac{i}{2}\lambda)\Gamma\left(\frac{2\gamma+1}{4}+\frac{i}{2}\lambda\right)|^{2}}
{|\Gamma(\frac{3-2\gamma}{4}+\frac{i}{2}\lambda)\Gamma\left(\frac{2\gamma+1}{4}-\frac{i}{2}\lambda\right)|^{2}}.
\end{align}
By using  (\ref{2.11}), we get
\begin{align*}
\left|\Gamma\left(\frac{3-2\gamma}{4}+\frac{i}{2}\lambda\right)\Gamma\left(\frac{2\gamma+1}{4}-\frac{i}{2}\lambda\right)\right|^{-2}=
\frac{|\sin(\frac{3-2\gamma}{4}+\frac{i}{2}\lambda)\pi|^{2}}{\pi^{2}}.
\end{align*}
We compute
\begin{align*}
\left|\sin\big(\frac{3-2\gamma}{4}+\frac{i}{2}\lambda\big)\pi\right|^{2}=&\left|\frac{e^{(\frac{3-2\gamma}{4}i-\frac{1}{2}\lambda)\pi}
-e^{(-\frac{3-2\gamma}{4}i+\frac{1}{2}\lambda)\pi}}{2i}\right|^{2}\\
=&\frac{1}{2}\left(\cosh\pi\lambda-\cos(\gamma-3/2)\pi\right)\\
=&\frac{1}{2}\left(\cosh\pi\lambda+\sin\gamma\pi\right).
\end{align*}
Therefore,
\begin{align}\label{4.7}
\left|\Gamma\left(\frac{3-2\gamma}{4}+\frac{i}{2}\lambda\right)\Gamma\left(\frac{2\gamma+1}{4}-\frac{i}{2}\lambda\right)\right|^{-2}=
\frac{1}{2\pi^{2}}\left(\cosh\pi\lambda+\sin\gamma\pi\right).
\end{align}
On the other hand, by using (\ref{a2.10}), we have
\begin{align}\label{4.8}
|\Gamma(\frac{3+2\gamma}{4}+\frac{i}{2}\lambda)\Gamma(\frac{2\gamma+1}{4}+\frac{i}{2}\lambda)|^{2}
=&2^{1-2\gamma}\pi |\Gamma(\gamma+1/2+i\lambda)|^{2}.
\end{align}
Substituting (\ref{4.7}) and (\ref{4.8}) into (\ref{4.6}), we get
\begin{align*}
2^{2\gamma}\frac{|\Gamma(\frac{3+2\gamma}{4}+\frac{i}{2}\lambda)|^{2}}
{|\Gamma(\frac{3-2\gamma}{4}+\frac{i}{2}\lambda)|^{2}}=&\frac{\cosh\gamma\pi}{\pi}|\Gamma(\gamma+1/2+i\lambda)|^{2}+\frac{\sin\gamma\pi}{\pi}|\Gamma(\gamma+1/2+i\lambda)|^{2}\\
=&\frac{|\Gamma(\gamma+\frac{1}{2}+i\lambda)|^{2}}
{|\Gamma(\frac{1}{2}+i\lambda)|^{2}}+\frac{\sin\gamma\pi}{\pi}|\Gamma(\gamma+1/2+i\lambda)|^{2}.
\end{align*}
This completes the proof of Lemma \ref{lm4.2}.
\end{proof}

Combining Lemmas \ref{lm4.1} and \ref{lm4.2} yields the following corollary:
\begin{corollary}\label{co4.3}
It holds that
\begin{align*}
P_{\gamma}=\widetilde{P}_{\gamma}+\frac{\sin\gamma\pi}{\pi}\left|\Gamma\left(\gamma+\frac{1}{2}+i\sqrt{-\Delta_{\mathbb{H}}-\frac{(n-1)^{2}}{4}}\right)\right|^{2}.
\end{align*}
\end{corollary}

Next we prove the fractional Poincar\'e inequality on $\mathbb{H}^{n}$.
\begin{lemma}\label{lm4.4}
It holds that
\begin{align*}
\int_{\mathbb{H}^{n}}uP_{\gamma}udV\geq  2^{2\gamma}\frac{\Gamma(\frac{3+2\gamma}{4})^{2}}
{\Gamma(\frac{3-2\gamma}{4})^{2}}\int_{\mathbb{H}^{n}}u^{2}dV,\;\; \gamma>0,\; \; u\in C^{\infty}_{0}(\mathbb{H}^{n}).
\end{align*}
Furthermore, the constant $2^{2\gamma}\frac{\Gamma(\frac{3+2\gamma}{4})^{2}}
{\Gamma(\frac{3-2\gamma}{4})^{2}}$ is sharp.
\end{lemma}
\begin{proof}
By Plancherel formula and Lemma \ref{lm4.1}, we have
\begin{align*}
\int_{\mathbb{H}^{n}}uP_{\gamma}udV=&\int^{+\infty}_{-\infty}\int_{\mathbb{S}^{n-1}}
2^{2\gamma}\frac{|\Gamma(\frac{3+2\gamma}{4}+\frac{i}{2}\lambda|^{2}}
{|\Gamma(\frac{3-2\gamma}{4}+\frac{i}{2}\lambda)|^{2}}|\widehat{u}(\lambda,\zeta)|^{2}|\mathfrak{c}(\lambda)|^{-2}d\lambda d\sigma(\zeta);\\
\int_{\mathbb{H}^{n}}u^{2}dV=&\int^{+\infty}_{-\infty}\int_{\mathbb{S}^{n-1}}
|\widehat{u}(\lambda,\zeta)|^{2}|\mathfrak{c}(\lambda)|^{-2}d\lambda d\sigma(\zeta).
\end{align*}
Therefore, by using (\ref{2.19}), we get
\begin{align*}
 \inf_{u\in C_{0}^{\infty}(\mathbb{H}^{n})}\frac{\int_{\mathbb{H}^{n}}uP_{\gamma}udV}{\int_{\mathbb{H}^{n}}u^{2}dV}
 =&\inf_{\lambda\in\mathbb{R}}2^{2\gamma}\frac{|\Gamma(\frac{3+2\gamma}{4}+\frac{i}{2}\lambda|^{2}}
{|\Gamma(\frac{3-2\gamma}{4}+\frac{i}{2}\lambda)|^{2}}
=2^{2\gamma}\frac{|\Gamma(\frac{3+2\gamma}{4})|^{2}}
{|\Gamma(\frac{3-2\gamma}{4})|^{2}}.
\end{align*}
Then the desired result follows.
\end{proof}

Before the proof of  Theorem \ref{th1.3}, we need the following Sobolev inequalities on hyperbolic space  (see \cite{y} for $0\leq s<\frac{n}{2}-\alpha$):
\begin{lemma}\label{lm4.5}
Let $n\geq2$,    $\zeta>0$ and $u\in C^{\infty}_{0}(\mathbb{H}^{n})$.

(1) if $n\geq3$, $0<\alpha<3/2$ and   $0< s+\alpha<\frac{n}{2}$,
  then for $2<p\leq\frac{2n}{n-2(\alpha+s)}$ we have
 \begin{equation}\label{a4.9}
\int_{\mathbb{H}^{n}}(\zeta^{2}-(n-1)^{2}/4-\Delta_{\mathbb{H}})^{s}(-(n-1)^{2}/4-\Delta_{\mathbb{H}})^{\alpha}u\cdot udV
\geq C\|u\|^{2}_{L^{p}(\mathbb{H}^{n})};
\end{equation}

(2) if   $n=2$  and  $-1< s<0$,   then for $2<p\leq-\frac{2}{s}$ we have
 \begin{equation}\label{a4.10}
\int_{\mathbb{H}^{2}}(\zeta^{2}-1/4-\Delta_{\mathbb{H}})^{s}(-1/4-\Delta_{\mathbb{H}})u\cdot udV
\geq C\|u\|^{2}_{L^{p}(\mathbb{H}^{2})}.
\end{equation}

Here  $C$ in (\ref{a4.9}) and (\ref{a4.10}) is a positive constant and  independent of $u$.
\end{lemma}
\begin{proof}
The proof depends on the following $L^{p'}\rightarrow L^{2}$ estimates for functions of the Laplace-Beltrami operator on $\mathbb{H}^{n}$ (see Anker \cite{an2}, Corollary 4.2 and Lohou\'e \cite{loh}):
\begin{itemize}
  \item $n\geq3$
\begin{align}\label{a4.11}
\|(-(n-1)^{2}/4-\Delta_{\mathbb{H}})^{-\alpha/2}   u\|^{2}_{L^{2}(\mathbb{H}^{n})}\lesssim
\|u\|_{L^{p'}(\mathbb{H}^{n})},\; 2<p=\frac{p'}{p'-1}\leq \frac{2n}{n-2\alpha};
\end{align}

  \item $n=2$
\begin{align}\label{}
\|(-1/4-\Delta_{\mathbb{H}})^{-\alpha/2}   u\|^{2}_{L^{2}(\mathbb{H}^{2})}\lesssim
\|u\|_{L^{p'}(\mathbb{H}^{2})},\;\;\left\{
                                    \begin{array}{ll}
                                      2< p=p'/(p'-1)<2/(1-\alpha), & \hbox{$0<\alpha<1$;} \\
                                      p=p'/(p'-1)>2, & \hbox{$\alpha=1$,}
                                    \end{array}
                                  \right.
\end{align}
\end{itemize}

(1) Since (\ref{a4.9}) is valid for $0\leq s<\frac{n}{2}-\alpha$ (see \cite{y}, Theorem 1.3), we need only to show the case $-\alpha<s<0$.

By using (\ref{a4.11}), we obtain
\begin{align}\label{a4.12}
&\|[(-(n-1)^{2}/4-\Delta_{\mathbb{H}})^{-(s+\alpha)/2}+(-(n-1)^{2}/4-\Delta_{\mathbb{H}})^{-\alpha/2}]u\|_{L^{2}(\mathbb{H}^{n})}\\
\leq& \|(-(n-1)^{2}/4-\Delta_{\mathbb{H}})^{-(s+\alpha)/2}u\|_{L^{2}(\mathbb{H}^{n})}+
\|(-(n-1)^{2}/4-\Delta_{\mathbb{H}})^{-\alpha/2}u\|_{L^{2}(\mathbb{H}^{n})}\nonumber\\
\lesssim&\|u\|_{L^{p'}(\mathbb{H}^{n})},\;\;2< p=\frac{p'}{p'-1}\leq \frac{2n}{n-2(s+\alpha)}.\nonumber
\end{align}
On the other hand, by the Plancherel formula, we have
\begin{align}\nonumber
&\|[(-(n-1)^{2}/4-\Delta_{\mathbb{H}})^{-(s+\alpha)/2}+(-(n-1)^{2}/4-\Delta_{\mathbb{H}})^{-\alpha/2}]u\|^{2}_{L^{2}(\mathbb{H}^{n})}\\
\nonumber
=&\int^{+\infty}_{-\infty}\int_{\mathbb{S}^{n-1}}(|\lambda|^{-(s+\alpha)}+|\lambda|^{-\alpha})^{2}
\widehat{u}(\lambda,\zeta)|^{2}|\mathfrak{c}(\lambda)|^{-2}d\lambda d\sigma(\zeta)\\
\nonumber
\thicksim&\int^{+\infty}_{-\infty}\int_{\mathbb{S}^{n-1}}(|\lambda|^{2}+\zeta^{2})^{-s/2}|\lambda|^{-2\alpha}
\widehat{u}(\lambda,\zeta)|^{2}|\mathfrak{c}(\lambda)|^{-2}d\lambda d\sigma(\zeta)\\
\label{a4.13}
=&\|(\zeta^{2}-(n-1)^{2}/4-\Delta_{\mathbb{H}})^{-s/2}(-(n-1)^{2}/4-\Delta_{\mathbb{H}})^{-\alpha/2}   u\|^{2}_{L^{2}(\mathbb{H}^{n})}.
\end{align}
Substituting (\ref{a4.13}) into (\ref{a4.12}), we obtain
\begin{align*}
\|(\zeta^{2}-(n-1)^{2}/4-\Delta_{\mathbb{H}})^{-s/2}(-(n-1)^{2}/4-\Delta_{\mathbb{H}})^{-\alpha/2}   u\|_{L^{2}(\mathbb{H}^{n})}\lesssim
\|u\|_{L^{p'}(\mathbb{H}^{n})},
\end{align*}
where $\frac{2n}{n+2(s+\alpha)}\leq p'<2$.
Therefore, by duality, we get
\begin{align*}
\|(\zeta^{2}-(n-1)^{2}/4-\Delta_{\mathbb{H}})^{-s/2}(-(n-1)^{2}/4-\Delta_{\mathbb{H}})^{-\alpha/2} v\|_{L^{p}(\mathbb{H}^{n})}
\lesssim \|v\|_{L^{2}(\mathbb{H}^{n})},\; v\in C^{\infty}_{0}(\mathbb{H}^{n}),
\end{align*}
which implies (\ref{a4.9}).

(2) The proof of (\ref{a4.10}) is similar to that given in (1) and we omit it.

\end{proof}

Now we can give the proof of Theorem \ref{th1.3}.

\medskip

\textbf{Proof of Theorem \ref{th1.3}}.
By Plancherel formula and Lemma \ref{lm4.1},   we have
\begin{align}\nonumber
&\int_{\mathbb{H}^{n}}uP_{\gamma}udV- 2^{2\gamma}\frac{\Gamma(\frac{3+2\gamma}{4})^{2}}
{\Gamma(\frac{3-2\gamma}{4})^{2}}\int_{\mathbb{H}^{n}}u^{2}dV\\
\label{4.11}
=&\int^{+\infty}_{-\infty}\int_{\mathbb{S}^{n-1}}
\left(2^{2\gamma}\frac{|\Gamma(\frac{3+2\gamma}{4}+\frac{i}{2}\lambda|^{2}}
{|\Gamma(\frac{3-2\gamma}{4}+\frac{i}{2}\lambda)|^{2}}|-2^{2\gamma}\frac{|\Gamma(\frac{3+2\gamma}{4})|^{2}}
{|\Gamma(\frac{3-2\gamma}{4})|^{2}}\right)\widehat{u}(\lambda,\zeta)|^{2}|\mathfrak{c}(\lambda)|^{-2}d\lambda d\sigma(\zeta).
\end{align}
On the other hand, by using  (\ref{2.13}) and (\ref{2.17}), we get
\begin{align*}
2^{2\gamma}\frac{|\Gamma(\frac{3+2\gamma}{4}+\frac{i}{2}\lambda|^{2}}
{|\Gamma(\frac{3-2\gamma}{4}+\frac{i}{2}\lambda)|^{2}}|-2^{2\gamma}\frac{|\Gamma(\frac{3+2\gamma}{4})|^{2}}
{|\Gamma(\frac{3-2\gamma}{4})|^{2}}\thicksim\left\{
                                              \begin{array}{ll}
                                                \lambda^{2}, & \hbox{$\lambda\rightarrow0$;} \\
                                                |\lambda|^{2\gamma}, & \hbox{$\lambda\rightarrow\infty$.}
                                              \end{array}
                                            \right.
\end{align*}
Therefore,
\begin{align}\label{4.12}
2^{2\gamma}\frac{|\Gamma(\frac{3+2\gamma}{4}+\frac{i}{2}\lambda|^{2}}
{|\Gamma(\frac{3-2\gamma}{4}+\frac{i}{2}\lambda)|^{2}}|-2^{2\gamma}\frac{|\Gamma(\frac{3+2\gamma}{4})|^{2}}
{|\Gamma(\frac{3-2\gamma}{4})|^{2}}\thicksim&\lambda^{2}(\lambda^{2}+1)^{\gamma-1},\;\;\lambda\in\mathbb{R}.
\end{align}
Substituting (\ref{4.12}) into (\ref{4.11}) and using Lemma \ref{lm4.5}, we obtain
\begin{align*}
&\int_{\mathbb{H}^{n}}uP_{\gamma}udV- 2^{2\gamma}\frac{\Gamma(\frac{3+2\gamma}{4})^{2}}
{\Gamma(\frac{3-2\gamma}{4})^{2}}\int_{\mathbb{H}^{n}}u^{2}dV\\
\thicksim&\int^{+\infty}_{-\infty}\int_{\mathbb{S}^{n-1}}\lambda^{2}(\lambda^{2}+1)^{\gamma-1}
\widehat{u}(\lambda,\zeta)|^{2}|\mathfrak{c}(\lambda)|^{-2}d\lambda d\sigma(\zeta)\\
=&\int_{\mathbb{H}^{n}}u(-\Delta_{\mathbb{H}}-\frac{(n-1)^{2}}{4})(-\Delta_{\mathbb{H}}-\frac{(n-1)^{2}}{4}+1)^{\gamma-1}udV\\
\geq&C\|u\|^{2}_{L^{p}(\mathbb{H}^{n})}.
\end{align*}
This completes the proof of Theorem \ref{th1.3}.

\medskip

Before the proof of  Theorem \ref{th1.5}, we need the following lemma:
\begin{lemma}\label{lm4.6}
Let $n\geq3$ be odd. There exists $\zeta>0$ such that
\begin{align*}
 2^{n}\frac{|\Gamma(\frac{3+n}{4}+\frac{i}{2}\lambda)|^{2}}
{|\Gamma(\frac{3-n}{4}+\frac{i}{2}\lambda)|^{2}}-2^{n}\frac{|\Gamma(\frac{3+n}{4})|^{2}}
{|\Gamma(\frac{3-n}{4})|^{2}}\geq&\lambda^{2}(\lambda^{2}+\zeta)^{\frac{n}{2}-1},\;\lambda\in\mathbb{R}.
\end{align*}
\end{lemma}
\begin{proof}
Without loss of generality, we  assume $\lambda>0$.
We have, by using (\ref{2.20}),
\begin{align*}
&2^{n}\frac{|\Gamma(\frac{3+n}{4}+\frac{i}{2}\lambda)|^{2}}
{|\Gamma(\frac{3-n}{4}+\frac{i}{2}\lambda)|^{2}}-2^{n}\frac{|\Gamma(\frac{3+n}{4})|^{2}}
{|\Gamma(\frac{3-n}{4})|^{2}}\\
=&2^{n-2}(\lambda^{2}+\frac{(n-1)^{2}}{4})\frac{|\Gamma(\frac{n-1}{4}+\frac{i}{2}\lambda)|^{2}}
{|\Gamma(\frac{3-n}{4}+\frac{i}{2}\lambda)|^{2}}-2^{n-4}(n-1)^{2}\frac{|\Gamma(\frac{n-1}{4})|^{2}}
{|\Gamma(\frac{3-n}{4})|^{2}}\\
=&2^{n-2}\lambda^{2}\frac{|\Gamma(\frac{n-1}{4}+\frac{i}{2}\lambda)|^{2}}
{|\Gamma(\frac{3-n}{4}+\frac{i}{2}\lambda)|^{2}}+2^{n-4}(n-1)^{2}\left(
\frac{|\Gamma(\frac{n-1}{4}+\frac{i}{2}\lambda)|^{2}}
{|\Gamma(\frac{3-n}{4}+\frac{i}{2}\lambda)|^{2}}-\frac{|\Gamma(\frac{n-1}{4})|^{2}}
{|\Gamma(\frac{3-n}{4})|^{2}}\right)\\
\geq&2^{n-2}\lambda^{2}\frac{|\Gamma(\frac{n-1}{4}+\frac{i}{2}\lambda)|^{2}}
{|\Gamma(\frac{3-n}{4}+\frac{i}{2}\lambda)|^{2}}.
\end{align*}
Therefore, to finish the proof,  it is enough to show
\begin{align}\label{4.13}
2^{n-2}\frac{|\Gamma(\frac{n-1}{4}+\frac{i}{2}\lambda)|^{2}}
{|\Gamma(\frac{3-n}{4}+\frac{i}{2}\lambda)|^{2}}\geq (\lambda^{2}+\zeta)^{\frac{n}{2}-1},\;\lambda>0.
\end{align}

The proof of (\ref{4.13}) is divided into two parts.

\medskip

Case 1: $n=4k+1$.  By using (\ref{2.14}) and (\ref{2.15}), we have
\begin{align}\nonumber
2^{n-2}\frac{|\Gamma(\frac{n-1}{4}+\frac{i}{2}\lambda)|^{2}}
{|\Gamma(\frac{3-n}{4}+\frac{i}{2}\lambda)|^{2}}
=&2^{4k-1}\frac{|\Gamma(k+\frac{i}{2}\lambda)|^{2}}
{|\Gamma(\frac{1}{2}-k+\frac{i}{2}\lambda)|^{2}}\\
\label{4.14}
=&\prod_{l=1}^{k-1}(\lambda^{2}+4l^{2})\times \prod_{m=0}^{k-1}(\lambda^{2}+(2m-1)^{2})\times\frac{\lambda\cosh\frac{\pi}{2}\lambda}{\sinh\frac{\pi}{2}\lambda}
\end{align}
We claim that there exists  $\zeta>0$ such that
\begin{align}\label{4.15}
\sqrt{\lambda^{2}+1}\frac{\lambda\cosh\frac{\pi}{2}\lambda}{\sinh\frac{\pi}{2}\lambda}\geq \lambda^{2}+\zeta,\;\;\lambda>0.
\end{align}
In fact, if we set $f(\lambda)=\lambda\sqrt{\lambda^{2}+1}\frac{\cosh\frac{\pi}{2}\lambda}{\sinh\frac{\pi}{2}\lambda}- \lambda^{2}$, then for $\lambda>0$,
\begin{align*}
f(\lambda)=&\lambda\sqrt{\lambda^{2}+1}\frac{\cosh\frac{\pi}{2}\lambda}{\sinh\frac{\pi}{2}\lambda}- \lambda^{2}\\
\geq& \lambda\sqrt{\lambda^{2}+1}- \lambda^{2}\\
>&0.
\end{align*}
On the other hand,  a simple calculation gives
\begin{align*}
 \lim_{\lambda\rightarrow0^{+}} f(\lambda)=&\frac{2}{\pi};\\
\lim_{\lambda\rightarrow+\infty} f(\lambda)=&\lim_{\lambda\rightarrow+\infty} (\sqrt{\lambda^{2}+1}-\lambda)
\frac{\lambda\cosh\frac{\pi}{2}\lambda}{\sinh\frac{\pi}{2}\lambda}+\lim_{\lambda\rightarrow+\infty}\lambda^{2}\frac{\cosh\frac{\pi}{2}\lambda}{\sinh\frac{\pi}{2}\lambda}\\
=&\lim_{\lambda\rightarrow+\infty}
\frac{\lambda\cosh\frac{\pi}{2}\lambda}{(\sqrt{\lambda^{2}+1}+\lambda)\sinh\frac{\pi}{2}\lambda}+
\lim_{\lambda\rightarrow+\infty}\lambda^{2}\frac{\cosh\frac{\pi}{2}\lambda}{\sinh\frac{\pi}{2}\lambda}\\
=&1.
\end{align*}
Therefore, we have
$
\inf\limits_{\lambda>0}f(\lambda)>0.
$
This proves the claim.

Substituting (\ref{4.15}) into (\ref{4.14}), we obtain
\begin{align*}
2^{n-2}\frac{|\Gamma(\frac{n-1}{4}+\frac{i}{2}\lambda)|^{2}}
{|\Gamma(\frac{3-n}{4}+\frac{i}{2}\lambda)|^{2}}
\geq&\prod_{l=1}^{k-1}(\lambda^{2}+4l^{2})\times \prod_{m=0}^{k-1}(\lambda^{2}+(2m-1)^{2})\times\frac{\lambda^{2}+\zeta}{\sqrt{\lambda^{2}+1}}\\
\geq&(\lambda^{2}+\zeta)^{2k-\frac{1}{2}}\\
=&(\lambda^{2}+\zeta)^{\frac{n}{2}-1},\; \lambda>0.
\end{align*}

Case 2: $n=4k+3$. By using (\ref{2.12}) and (\ref{2.16}), we have,
\begin{align}\nonumber
 2^{n}\frac{|\Gamma(\frac{3+n}{4}+\frac{i}{2}\lambda)|^{2}}
{|\Gamma(\frac{3-n}{4}+\frac{i}{2}\lambda)|^{2}}-2^{n}\frac{|\Gamma(\frac{3+n}{4})|^{2}}
{|\Gamma(\frac{3-n}{4})|^{2}}=& 2^{4k+3}\frac{|\Gamma(k+\frac{3}{2}+\frac{i}{2}\lambda)|^{2}}
{|\Gamma(-k+\frac{i}{2}\lambda)|^{2}}-2^{n}\frac{|\Gamma(k+1+\frac{1}{2})|^{2}}
{|\Gamma(-k)|^{2}}\nonumber
\\
\label{4.16}
=&\prod_{l=0}^{k}(\lambda^{2}+(2l-1)^{2})\times \prod_{m=1}^{k}(\lambda^{2}+4m^{2})\times\frac{\lambda\sinh\frac{\pi}{2}\lambda}{\cosh\frac{\pi}{2}\lambda}.
\end{align}
Notice that
\begin{align}\label{a4.17}
 (\lambda^{2}+1)(\sinh\beta\lambda)^{2}-\lambda^{2}(\cosh\beta\lambda)^{2}=(\sinh\beta\lambda)^{2}-\lambda^{2}>0,\;\;\forall \beta\geq1,
\end{align}
we have
\begin{align}\label{4.17}
\frac{\sinh\frac{\pi}{2}\lambda}{\cosh\frac{\pi}{2}\lambda}\geq\frac{\lambda}{\sqrt{\lambda^{2}+1}}.
\end{align}
Substituting (\ref{4.17}) into (\ref{4.16}), we obtain
\begin{align*}
 2^{n}\frac{|\Gamma(\frac{3+n}{4}+\frac{i}{2}\lambda)|^{2}}
{|\Gamma(\frac{3-n}{4}+\frac{i}{2}\lambda)|^{2}}-2^{n}\frac{|\Gamma(\frac{3+n}{4})|^{2}}
{|\Gamma(\frac{3-n}{4})|^{2}}
\geq&\prod_{l=0}^{k}(\lambda^{2}+(2l-1)^{2})\times \prod_{m=1}^{k}(\lambda^{2}+4m^{2})\times
\frac{\lambda^{2}}{\sqrt{\lambda^{2}+1}}\\
\geq&\lambda^{2}(\lambda^{2}+1)^{2k+\frac{1}{2}}\\
=&\lambda^{2}(\lambda^{2}+1)^{\frac{n}{2}-1}.
\end{align*}
The proof of Lemma \ref{lm4.6} is then completed.
\end{proof}

 Before proving Theorem \ref{th1.5}, we need the following Adams inequalities on hyperbolic space (see \cite{y}, Corollary 1.4):
\begin{theorem} \label{th4.7}
Let $n\geq3$,  $\zeta>0$ and  $0<s<3/2$.
There exists $C=C(\zeta,n)$ such that
\[
\int_{\mathbb{B}^{n}}[\exp(\beta_{0}(n,n/2)u^{2})-1-\beta_{0}(n,n/2)u^{2}]dV\leq C
\]
for  any $u\in C^{\infty}_{0}(\mathbb{B}^{n})$ with
\[\int_{\mathbb{B}^{n}}|(-\Delta_{\mathbb{H}}-(n-1)^{2}/4)^{s/2}(-\Delta_{\mathbb{H}}-(n-1)^{2}/4
+\zeta^{2})^{\frac{n-2s}{4}} u|^{2}dV\leq1.\]
Furthermore, this inequality is sharp  in the sense that if $\beta_{0}(n,n/2)$ is replaced by
any $\beta>\beta_{0}(n,n/2)$, then the above inequality can no longer hold with any $C$ independent of $u$.
\end{theorem}

\medskip

\textbf{Proof of Theorem \ref{th1.5}}.
By Plancherel formula and Lemma \ref{lm4.6}, we have, for some $\zeta>0$,
\begin{align*}
&\int_{\mathbb{H}^{n}}uP_{\frac{n}{2}}udV-
2^{n}\frac{\Gamma(\frac{3+n}{4})^{2}}
{\Gamma(\frac{3-n}{4})^{2}}\int_{\mathbb{H}^{n}}u^{2}dV\\
=&\int^{+\infty}_{-\infty}\int_{\mathbb{S}^{n-1}}
\left(2^{n}\frac{|\Gamma(\frac{3+n}{4}+\frac{i}{2}\lambda)|^{2}}
{|\Gamma(\frac{3-n}{4}+\frac{i}{2}\lambda)|^{2}}-2^{n}\frac{|\Gamma(\frac{3+n}{4})|^{2}}
{|\Gamma(\frac{3-n}{4})|^{2}}\right)\widehat{u}(\lambda,\zeta)|^{2}|\mathfrak{c}(\lambda)|^{-2}d\lambda d\sigma(\zeta)\\
\geq&\int^{+\infty}_{-\infty}\int_{\mathbb{S}^{n-1}}
\lambda^{2}(\lambda^{2}+\zeta)^{\frac{n}{2}-1}\widehat{u}(\lambda,\zeta)|^{2}|\mathfrak{c}(\lambda)|^{-2}d\lambda d\sigma(\zeta)\\
=&\int_{\mathbb{H}^{n}}|(-\Delta_{\mathbb{H}}-(n-1)^{2}/4)^{1/2}(-\Delta_{\mathbb{H}}-(n-1)^{2}/4
+\zeta)^{\frac{n-2}{4}} u|^{2}dV.
\end{align*}
The desired result follows by using Theorem \ref{th4.7}.

\section{ proofs of Theorems \ref{th1.6}, \ref{th1.7} and \ref{th1.9}}\label{Section6}
We first recall  the Green's  function of $(-\Delta)^{\gamma}$ on $\mathbb{B}^{n}$ (see \cite{blu,k,buc} for $0<\gamma<1$ and \cite{dip,ab} for $\gamma\geq1$)
\begin{align*}
G_{\mathbb{B}^{n},\gamma}(x,y)=&\frac{\Gamma(\frac{n}{2})}{\pi^{n/2}4^{\gamma}\Gamma(\gamma)^{2}}|x-y|^{2\gamma-n}\int_{0}^{\frac{(1-|x|^{2})(1-|y|^{2})}{|x-y|^{2}}}
\frac{t^{\gamma-1}}{(t+1)^{n/2}}dt.
\end{align*}
In terms of ball model of hyperbolic space, we have the following lemma:
\begin{lemma}\label{lm5.1} Let $\gamma>0$.
It holds that
\begin{align*}
G_{\mathbb{B}^{n},\gamma}(x,y)=2^{n-2\gamma}(1-|x|^{2})^{\gamma-\frac{n}{2}}(1-|y|^{2})^{\gamma-\frac{n}{2}}K_{1/2,\gamma}(\cosh\rho(x,y)),
\end{align*}
where $K_{\nu,\gamma}(\cosh\rho)$ is given by (\ref{3.8}).
\end{lemma}
\begin{proof}
By (\ref{2.4}), we have
\begin{align*}
&2^{2\gamma-n}(1-|x|^{2})^{\frac{n}{2}-\gamma}(1-|y|^{2})^{\frac{n}{2}-\gamma}G_{\mathbb{B}^{n},\gamma}(x,y)\\
=&
\frac{\Gamma(\frac{n}{2})}{\pi^{n/2}2^{n}\Gamma(\gamma)^{2}}\left(\sinh\frac{\rho(x,y)}{2}\right)^{2\gamma-n}
\int_{0}^{\left(\sinh\frac{\rho(x,y)}{2}\right)^{-2}}
\frac{t^{\gamma-1}}{(t+1)^{n/2}}dt.
\end{align*}
Substituting $s=1-t\left(\sinh\frac{\rho(x,y)}{2}\right)^{2}$ and using (\ref{2.6}), we obtain
\begin{align*}
&2^{2\gamma-n}(1-|x|^{2})^{\frac{n}{2}-\gamma}(1-|y|^{2})^{\frac{n}{2}-\gamma}G_{\mathbb{B}^{n},\gamma}(x,y)\\
=&
\frac{\Gamma(\frac{n}{2})}{\pi^{n/2}2^{n}\Gamma(\gamma)^{2}}\int_{0}^{1}
(1-s)^{\gamma-1}\left((\cosh\frac{\rho(x,y)}{2})^{2}-1\right)^{-\frac{n}{2}}ds\\
=&\frac{\Gamma(\frac{n}{2})}{2^{n}\pi^{\frac{n}{2}}
\Gamma(\gamma)\Gamma(\gamma+1)}(\cosh\frac{\rho(x,y)}{2})^{-n}F(\frac{n}{2}, 1,;1+\gamma;(\cosh\frac{\rho(x,y)}{2})^{-2})\\
=&K_{1/2,\gamma}(\cosh\rho(x,y)).
\end{align*}
The desired result then follows.
\end{proof}

Similarly, the Green's function of $(-\Delta)^{\gamma}$ on $\mathbb{R}_{+}^{n}$ is given by (see \cite{ch} for $0<\gamma<1$)
\begin{align}\label{5.1}
G_{\mathbb{R}_{+}^{n},\gamma}(x,y)=&\frac{\Gamma(\frac{n}{2})}{\pi^{n/2}4^{\gamma}\Gamma(\gamma)^{2}}|x-y|^{2\gamma-n}
\int_{0}^{\frac{4x_{1}y_{1}}{|x-y|^{2}}}
\frac{t^{\gamma-1}}{(t+1)^{n/2}}dt.
\end{align}
We remark that (\ref{5.1}) is also valid for $\gamma\geq 1$. The proof is similar to that given by \cite{ch,dip,ab} and we omit it. With the same
arguments as in the proof of Lemma \ref{lm5.1}, we have, in terms of half space model of hyperbolic space,
\begin{align}\label{a5.5}
G_{\mathbb{R}_{+}^{n},\gamma}(x,y)=(x_{1}y_{1})^{\gamma-\frac{n}{2}}K_{1/2,\gamma}(\cosh\rho(x,y)).
\end{align}
\medskip

Before proving  Theorems \ref{th1.6} and \ref{th1.7}, we need the following lemma:
\begin{lemma}\label{lm5.2a}
Let $\mathcal{S}(\mathbb{R}^{n})$ be the Schwartz space on $\mathbb{R}^{n}$.
It holds that
\begin{align}\label{a5.2}
|(-\Delta)^{\gamma}u(x)|\lesssim \frac{1}{|x|^{n+\gamma}},\;\; |x|\rightarrow\infty,\;\;  u\in \mathcal{S}(\mathbb{R}^{n}),\;\; \gamma>0.
\end{align}
\end{lemma}
\begin{proof}
We note that if $\gamma\in \mathbb{N}\setminus\{0\}$ , then   $(-\Delta)^{\gamma}u \in\mathcal{S}(\mathbb{R}^{n})$  and  (\ref{a5.2}) follows.

Now we assume $\gamma\in (0,\infty)\setminus\mathbb{N}$.
We shall prove (\ref{a5.2}) by induction.

For $0<\gamma<1$, (\ref{a5.2}) has been proved by C. Bucur (see \cite{buc}, (1.9)).
Assume  (\ref{a5.2}) is valid for $k<\gamma<k+1$. We compute, by the inverse Fourier transform,
\begin{align}\nonumber
(-\Delta)^{\gamma+1}u(x)=&\int_{\mathbb{R}^{n}}(2\pi|\xi|)^{2\gamma+2}\widehat{u}(\xi)e^{2\pi ix\cdot\xi}d\xi\\
\nonumber
=&\frac{1}{2\pi i|x|^{2}}\int_{\mathbb{R}^{n}}(2\pi|\xi|)^{2\gamma+2}\widehat{u}(\xi)\sum_{j=1}^{n}x_{j}\frac{\partial}{\partial \xi_{j}}e^{2\pi ix\cdot\xi}d\xi\\
\nonumber
=&-2(1+\gamma)\sum_{j=1}^{n}\frac{x_{j}}{i|x|^{2}}
\int_{\mathbb{R}^{n}}(2\pi|\xi|)^{2\gamma}2\pi\xi_{j}\widehat{u}(\xi)e^{2\pi ix\cdot\xi}d\xi-\\
\nonumber
&\sum_{j=1}^{n}\frac{x_{j}}{2\pi i|x|^{2}}\int_{\mathbb{R}^{n}}(2\pi|\xi|)^{2\gamma+2}\frac{\partial \widehat{u}(\xi)}{\partial \xi_{j}}e^{2\pi ix\cdot\xi}d\xi\\
\label{a5.3}
=&2(1+\gamma)\sum_{j=1}^{n}\frac{x_{j}}{|x|^{2}}(-\Delta)^{\gamma}\partial_{x_{j}}u(x)+
\sum_{j=1}^{n}\frac{x_{j}}{|x|^{2}}(-\Delta)^{\gamma+1}(x_{j}u(-x)).
\end{align}
Since $\partial_{x_{j}}u(x), \Delta(x_{j}u(-x))\in \mathcal{S}(\mathbb{R}^{n})$, we have, by the induction hypothesis,
\begin{equation}\label{a5.4}
  \begin{split}
 |(-\Delta)^{\gamma}\partial_{x_{j}}u(x)|\lesssim & \frac{1}{|x|^{n+\gamma}},\;\;|x|\rightarrow\infty;\\
 |(-\Delta)^{\gamma+1}(x_{j}u(-x))|=
 |(-\Delta)^{\gamma}\Delta(x_{j}u(-x))|\lesssim &\frac{1}{|x|^{n+\gamma}},\;\;|x|\rightarrow\infty.
   \end{split}
\end{equation}
Substituting (\ref{a5.4}) into (\ref{a5.3}), we obtain
\begin{align*}
|(-\Delta)^{\gamma+1}u(x)|\lesssim &\frac{1}{|x|^{n+1+\gamma}},\;\;|x|\rightarrow\infty.
\end{align*}
This completes the proof of Lemma \ref{lm5.2a}.
\end{proof}

Now we can give the Proofs of Theorems \ref{th1.6} and \ref{th1.7}.

\medskip

\textbf{Proof of Theorem \ref{th1.6}}.
We first   prove (1.13).
Set $g(x)=x_{1}^{\gamma+\frac{n}{2}}(-\Delta)^{\gamma}(x_{1}^{\gamma-\frac{n}{2}}u),\; x_{1}>0$.
It is enough to show
\begin{align*}
\widehat{g}(\lambda,\zeta)=\frac{|\Gamma(\frac{1}{2}+\gamma+i\lambda)|^{2}}{|\Gamma(\frac{1}{2}+i\lambda)|^{2}}\widehat{u}(\lambda,\zeta).
\end{align*}

By (\ref{a5.5}), we have
\begin{align}\label{a5.6}
  u(x)=&x_{1}^{\frac{n}{2}-\gamma}\int_{\mathbb{R}_{+}^{n}}G_{\mathbb{R}_{+}^{n},\gamma}(x,y)y_{1}^{-\frac{n}{2}-\gamma}g(y)dy
  =\int_{\mathbb{H}^{n}}K_{1/2,\gamma}(\cosh\rho(x,y))g(y)dV,\;\; x\in \mathbb{R}_{+}^{n}.
\end{align}
Moreover, since $u\in C_{0}^{\infty}(\mathbb{R}_{+}^{n})$, we have
\begin{align}\label{a5.7}
|g(x)|=|x_{1}^{\gamma+\frac{n}{2}}(-\Delta)^{\gamma}(x_{1}^{\gamma-\frac{n}{2}}u)|\lesssim x_{1}^{\gamma+\frac{n}{2}}
\end{align}
and by Lemma \ref{lm5.2a},
\begin{align}\label{a5.8}
|g(x)|\lesssim \frac{x_{1}^{\gamma+\frac{n}{2}}}{|x|^{n+\gamma}},\;\; |x|\rightarrow\infty.
\end{align}
Combining (\ref{a5.7}),  (\ref{a5.8}) and (\ref{2.1}) yields
\begin{align}\label{a5.9}
|g(x)|\lesssim\frac{x_{1}^{\gamma+\frac{n}{2}}}{1+|x|^{n+\gamma}}\leq \frac{x_{1}^{\frac{n}{2}}}{1+|x|^{n}}\thicksim\big(\cosh\frac{\rho(x)}{2}\big)^{-\frac{n}{2}}.
\end{align}
Therefore,  taking the Helgason-Fourier transform in both sides of  (\ref{a5.6}) and using Lemma \ref{lm3.8a}, we get
\begin{align*}
 \widehat{u}(\lambda,\zeta)=&\frac{|\Gamma(\frac{1}{2}+i\lambda)|^{2}}{|\Gamma(\frac{1}{2}+\gamma+i\lambda)|^{2}}\widehat{g}(\lambda,\zeta),
\end{align*}
i.e.
\begin{align*}
\widehat{g}(\lambda,\zeta)=\frac{|\Gamma(\frac{1}{2}+\gamma+i\lambda)|^{2}}{|\Gamma(\frac{1}{2}+i\lambda)|^{2}}\widehat{u}(\lambda,\zeta).
\end{align*}
This prove (\ref{1.13}). The  proof (\ref{1.14}) is similar and we omit it. The proof of Theorem \ref{th1.6} is thereby completed.

\medskip

\textbf{Proof of Theorem \ref{th1.7}}.
With the same arguments as the proof of Lemma \ref{4.4}, we know that the constant $ \frac{\Gamma(\gamma+\frac{1}{2})^{2}}
{\Gamma(\frac{1}{2})^{2}}$ is sharp.

By Plancherel formula and Lemma \ref{lm4.4},  we have
\begin{align*}
&\int_{\mathbb{H}^{n}}u\widetilde{P}_{\gamma}udV- \frac{\Gamma(\gamma+\frac{1}{2})^{2}}
{\Gamma(\frac{1}{2})^{2}}\int_{\mathbb{H}^{n}}u^{2}dV\\
=&\int^{+\infty}_{-\infty}\int_{\mathbb{S}^{n-1}}
\left(\frac{|\Gamma(\gamma+\frac{1}{2}+i\lambda)|^{2}}{|\Gamma(\frac{1}{2}+i\lambda)|^{2}}-
\frac{|\Gamma(\gamma+\frac{1}{2})|^{2}}{|\Gamma(\frac{1}{2})|^{2}}\right)\widehat{u}(\lambda,\zeta)|^{2}|\mathfrak{c}(\lambda)|^{-2}d\lambda d\sigma(\zeta).
\end{align*}
On the other hand, by using  (\ref{2.13}) and (\ref{2.17}), we have
\begin{align*}
\frac{|\Gamma(\gamma+\frac{1}{2}+i\lambda)|^{2}}{|\Gamma(\frac{1}{2}+i\lambda)|^{2}}-
\frac{|\Gamma(\gamma+\frac{1}{2})|^{2}}{|\Gamma(\frac{1}{2})|^{2}}\thicksim\left\{
                                              \begin{array}{ll}
                                                \lambda^{2}, & \hbox{$\lambda\rightarrow0$;} \\
                                                |\lambda|^{2\gamma}, & \hbox{$\lambda\rightarrow\infty$.}
                                              \end{array}
                                            \right.
\end{align*}
The rest of the proof  of (\ref{1.15})  is exactly similar to that given in the proof of Theorem \ref{th1.3} and we omit it.
Inequalities (\ref{1.16}) and (\ref{1.17}) follow by combining Theorem \ref{th1.6} and (\ref{1.15}). The proof of Theorem \ref{th1.7} is thereby completed.

\medskip

Before proving Theorem \ref{th1.9}, we need the following lemma:
\begin{lemma}\label{lm5.2}
Let $n\geq3$ be odd. There exists $\zeta>0$ such that
\begin{align*}
 \frac{|\Gamma(\frac{n+1}{2}+i\lambda)|^{2}}{|\Gamma(\frac{1}{2}+i\lambda)|^{2}}-
\frac{|\Gamma(\frac{n+1}{2})|^{2}}{|\Gamma(\frac{1}{2})|^{2}}\geq&\lambda^{2}(\lambda^{2}+\zeta)^{\frac{n}{2}-1},\;\lambda\in\mathbb{R}.
\end{align*}
\end{lemma}
\begin{proof}
Without loss of generality, we  assume $\lambda>0$.
We have
\begin{align}\nonumber
 &\frac{|\Gamma(\frac{n+1}{2}+i\lambda)|^{2}}{|\Gamma(\frac{1}{2}+i\lambda)|^{2}}-
\frac{|\Gamma(\frac{n+1}{2})|^{2}}{|\Gamma(\frac{1}{2})|^{2}}\\
\nonumber
=&(\lambda^{2}+1/4)\frac{|\Gamma(\frac{n+1}{2}+i\lambda)|^{2}}{|\Gamma(\frac{3}{2}+i\lambda)|^{2}}-\frac{1}{4}
\frac{|\Gamma(\frac{n+1}{2})|^{2}}{|\Gamma(\frac{3}{2})|^{2}}\\
\nonumber
=&\lambda^{2}\frac{|\Gamma(\frac{n+1}{2}+i\lambda)|^{2}}{|\Gamma(\frac{3}{2}+i\lambda)|^{2}}+\frac{1}{4}\left(
\frac{|\Gamma(\frac{n+1}{2}+i\lambda)|^{2}}{|\Gamma(\frac{3}{2}+i\lambda)|^{2}}-\frac{|\Gamma(\frac{n+1}{2})|^{2}}{|\Gamma(\frac{3}{2})|^{2}}\right)\\
\label{5.2}
\geq&\lambda^{2}\frac{|\Gamma(\frac{n+1}{2}+i\lambda)|^{2}}{|\Gamma(\frac{3}{2}+i\lambda)|^{2}}.
\end{align}
To get the last inequality above, we use (\ref{2.20}).

Set $n=2k+1$.  By using  (\ref{2.14}) and (\ref{2.15}), we have,
\begin{align}\label{5.3}
\lambda^{2}\frac{|\Gamma(\frac{n+1}{2}+i\lambda)|^{2}}{|\Gamma(\frac{3}{2}+i\lambda)|^{2}}=&\lambda^{2}\frac{|\Gamma(k+1+i\lambda)|^{2}}{|\Gamma(\frac{3}{2}+i\lambda)|^{2}}
=\lambda^{2}\prod_{m=1}^{k}(\lambda^{2}+m^{2})\times\frac{\lambda\cosh\pi\lambda}{(\lambda^{2}+\frac{1}{4})\sinh\pi\lambda}.
\end{align}
We claim
\begin{align}\label{5.4}
\lambda\sqrt{\lambda^{2}+1}\cosh\pi\lambda\geq (\lambda^{2}+1/4)\sinh\pi\lambda.
\end{align}
In fact,
\begin{align}\label{5.5}
\left(\lambda\sqrt{\lambda^{2}+1}\cosh\pi\lambda\right)^{2}- \left[(\lambda^{2}+1/4)\sinh\pi\lambda\right]^{2}
=&\frac{1}{2}(\lambda^{2}-1/8)(\sinh\pi\lambda)^{2}+\lambda^{2}+\lambda^{4}.
\end{align}
Using the  series expansion of $\cosh2\pi\lambda$, we get
\begin{align}\label{5.6}
 (\sinh\pi\lambda)^{2}=\frac{1}{2}(\cosh2\pi\lambda-1)=\frac{1}{2}\sum_{k=1}^{\infty}\frac{(2\pi)^{2k}}{(2k)!}\lambda^{2k}.
\end{align}
Substituting (\ref{5.6}) into (\ref{5.5}), we obtain
\begin{align*}
 &\left(\lambda\sqrt{\lambda^{2}+1}\cosh\pi\lambda\right)^{2}- \left[(\lambda^{2}+1/4)\sinh\pi\lambda\right]^{2}\\
 =&\frac{1}{4}(\lambda^{2}-1/8)\sum_{k=1}^{\infty}\frac{(2\pi)^{2k}}{(2k)!}\lambda^{2k}+\lambda^{2}+\lambda^{4}\\
 =&\frac{1}{4}\sum_{k=2}^{\infty}\frac{(2\pi)^{2k-2}}{(2k)!}\left[2k(2k-1)-\frac{\pi^{2}}{2}\right]\lambda^{2k}+\left(1-\frac{\pi^{2}}{16}\right)\lambda^{2}
 +\lambda^{4}\\
 \geq&0.
\end{align*}
This proves the claim.

Substituting (\ref{5.4}) into (\ref{5.3}), we obtain
\begin{align}\label{5.7}
\lambda^{2}\frac{|\Gamma(\frac{n-1}{2}+i\lambda)|^{2}}{|\Gamma(\frac{1}{2}+i\lambda)|^{2}}\geq \lambda^{2}
\prod_{m=1}^{k}(\lambda^{2}+m^{2})\times\frac{1}{\sqrt{\lambda^{2}+1}}\geq \lambda^{2}(\lambda^{2}+1)^{\frac{n}{2}-1}.
\end{align}
The desired result then follows by combining (\ref{5.2}) and (\ref{5.7}). This completes the proof of Lemma \ref{lm5.2}.
\end{proof}

\textbf{Proof of Theorem \ref{th1.9}}.
By Plancherel formula and Lemma \ref{lm5.2}, we have, for some $\zeta>0$
\begin{align*}
&\int_{\mathbb{H}^{n}}u\widetilde{P}_{\frac{n}{2}}udV-
\frac{\Gamma(\frac{n+1}{2})^{2}}
{\Gamma(\frac{1}{2})^{2}}\int_{\mathbb{H}^{n}}u^{2}dV\\
=&\int^{+\infty}_{-\infty}\int_{\mathbb{S}^{n-1}}
\left(\frac{|\Gamma(\frac{n+1}{2}+i\lambda)|^{2}}{|\Gamma(\frac{1}{2}+i\lambda)|^{2}}-
\frac{|\Gamma(\frac{n+1}{2})|^{2}}{|\Gamma(\frac{1}{2})|^{2}}\right)\widehat{u}(\lambda,\zeta)|^{2}|\mathfrak{c}(\lambda)|^{-2}d\lambda d\sigma(\zeta)\\
\geq&\int^{+\infty}_{-\infty}\int_{\mathbb{S}^{n-1}}
\lambda^{2}(\lambda^{2}+\zeta)^{\frac{n}{2}-1}\widehat{u}(\lambda,\zeta)|^{2}|\mathfrak{c}(\lambda)|^{-2}d\lambda d\sigma(\zeta)\\
=&\int_{\mathbb{H}^{n}}|(-\Delta_{\mathbb{H}}-(n-1)^{2}/4)^{1/2}(-\Delta_{\mathbb{H}}-(n-1)^{2}/4
+\zeta)^{\frac{n-2}{4}} u|^{2}dV.
\end{align*}
Therefore, by Theorem \ref{th4.7}, we obtain (\ref{1.19}). This completes the proof of Theorem \ref{th1.9}.

\section{proofs of Theorems \ref{th1.4} and \ref{th1.8} }\label{Section7}
We first prove Theorem \ref{th1.8}. The proof depends on the following two lemmas.
\begin{lemma}\label{lm6.1}
Let $H_{0,\gamma}(\cosh\rho)$ be given by (\ref{a3.18}). It holds that
\begin{align*}
\left(-\Delta_{\mathbb{H}}-\frac{(n-1)^{2}}{4}\right)H_{0,\gamma}
=&\frac{(n-2\gamma)(2\gamma-2)}{4}
\left(\sinh\frac{\rho}{2}\right)^{2\gamma-2-n}\left(\cosh\frac{\rho}{2}\right)^{1-2\gamma}+\\
&\frac{(2\gamma-1)(2\gamma+1-n)}{4(\cosh\frac{\rho}{2})^{2}}H_{0,\gamma}.
\end{align*}
In particular, if $\frac{n-1}{2}\leq \gamma<\frac{n}{2}$, then
\begin{align*}
\left(-\Delta_{\mathbb{H}}-\frac{(n-1)^{2}}{4}\right)H_{0,\gamma}
\geq&\frac{(n-2\gamma)(2\gamma-2)}{4}
\left(\sinh\frac{\rho}{2}\right)^{2\gamma-2-n}\left(\cosh\frac{\rho}{2}\right)^{1-2\gamma}.
\end{align*}
\end{lemma}
\begin{proof} Recall that $H_{0,\gamma}(\cosh\rho)=\left(\sinh\frac{\rho}{2}\right)^{2\gamma-n}\left(\cosh\frac{\rho}{2}\right)^{1-2\gamma}$.
We compute
\begin{align*}
\partial_{\rho}H_{0,\gamma}=&\left[\frac{2\gamma-n}{2} \coth\frac{\rho}{2}+\frac{1-2\gamma}{2}\tanh\frac{\rho}{2}\right]H_{0,\gamma};\\
\partial_{\rho\rho}H_{0,\gamma}=&\left[- \frac{2\gamma-n}{4(\sinh\frac{\rho}{2})^{2}}+\frac{1-2\gamma}{4(\cosh\frac{\rho}{2})^{2}}\right]H_{0,\gamma}+
  \left[\frac{2\gamma-n}{2} \coth\frac{\rho}{2}+\frac{1-2\gamma}{2}\tanh\frac{\rho}{2}\right]\partial_{\rho}H_{0,\gamma}\\
  =&\left[- \frac{2\gamma-n}{4(\sinh\frac{\rho}{2})^{2}}+\frac{1-2\gamma}{4(\cosh\frac{\rho}{2})^{2}}\right]H_{0,\gamma}+
  \left[\frac{2\gamma-n}{2} \coth\frac{\rho}{2}+\frac{1-2\gamma}{2}\tanh\frac{\rho}{2}\right]^{2}H_{0,\gamma}\\
  =&\left[\frac{(2\gamma-n)^{2}(\cosh\frac{\rho}{2})^{2}-2\gamma+n}{(2\sinh\frac{\rho}{2})^{2}}+\frac{(2\gamma-n)(1-2\gamma)}{2}
  +\right.\\
  &\left.\frac{(1-2\gamma)^{2}(\sinh\frac{\rho}{2})^{2}+1-2\gamma}{4(\cosh\frac{\rho}{2})^{2}}\right]H_{0,\gamma}.
\end{align*}
Using $(\cosh\frac{\rho}{2})^{2}=1+(\sinh\frac{\rho}{2})^{2}$, we have
\begin{align}\nonumber
  \partial_{\rho\rho}H_{0,\gamma}(\cosh\rho)
  =&\left[\frac{(2\gamma-n)^{2}[1+(\sinh\frac{\rho}{2})^{2}]-2\gamma+n}{(2\sinh\frac{\rho}{2})^{2}}+
  \frac{(2\gamma-n)(1-2\gamma)}{2}
  +\right.\\ \nonumber
  &\left.\frac{(1-2\gamma)^{2}[(\cosh\frac{\rho}{2})^{2}-1]+1-2\gamma}{4(\cosh\frac{\rho}{2})^{2}}\right]H_{0,\gamma}\\
  =&\left[\frac{(2\gamma-n)(2\gamma-n-1)}{(2\sinh\frac{\rho}{2})^{2}}+\frac{(n-1)^{2}}{4}
   +\frac{2\gamma(1-2\gamma)}{4(\cosh\frac{\rho}{2})^{2}}\right]H_{0,\gamma}. \label{6.1}
\end{align}
Similarly,
\begin{align}\nonumber
 (n-1) \frac{\cosh\rho}{\sinh\rho}\partial_{\rho}H_{0,\gamma}=& (n-1)\left[\frac{2\gamma-n}{2} \frac{\cosh\rho}{2(\sinh\frac{\rho}{2})^{2}}+\frac{1-2\gamma}{2}
  \frac{\cosh\rho}{2(\cosh\frac{\rho}{2})^{2}}\right]H_{0,\gamma}\\\nonumber
  =&(n-1)\left[\frac{(2\gamma-n)(1+2(\sinh\frac{\rho}{2})^{2})}{(2\sinh\frac{\rho}{2})^{2}}+
  \frac{(1-2\gamma)[2(\cosh\frac{\rho}{2})^{2}-1]}{4(\cosh\frac{\rho}{2})^{2}}\right]H_{0,\gamma}\\
  =&(n-1)\left[\frac{2\gamma-n}{(2\sinh\frac{\rho}{2})^{2}}+\frac{1-n}{2}+
  \frac{2\gamma-1}{4(\cosh\frac{\rho}{2})^{2}}\right]H_{0,\gamma}. \label{6.2}
\end{align}
Since $H_{0,\gamma}$ is radial, we have,  by using  (\ref{6.1}) and (\ref{6.2}),
\begin{align*}
&\left(-\Delta_{\mathbb{H}}-\frac{(n-1)^{2}}{4}\right)H_{0,\gamma}\\
=&\left(-\partial_{\rho\rho}-(n-1) \frac{\cosh\rho}{\sinh\rho}\partial_{\rho}-\frac{(n-1)^{2}}{4}\right)H_{0,\gamma}\\
=&(n-2\gamma)(2\gamma-2)\frac{1}{(2\sinh\frac{\rho}{2})^{2}}H_{0,\gamma}+\frac{(2\gamma-1)(2\gamma+1-n)}{4(\cosh\frac{\rho}{2})^{2}}H_{0,\gamma}\\
=&\frac{(n-2\gamma)(2\gamma-2)}{4}
\left(\sinh\frac{\rho}{2}\right)^{2\gamma-2-n}\left(\cosh\frac{\rho}{2}\right)^{1-2\gamma}+\frac{(2\gamma-1)(2\gamma+1-n)}{4(\cosh\frac{\rho}{2})^{2}}H_{0,\gamma}.
\end{align*}
This completes the proof of Lemma \ref{lm6.1}.
\end{proof}

\begin{lemma} \label{lm6.2} Let  $n\geq 3$ and $\frac{n-1}{2}\leq \gamma<\frac{n}{2}$. It holds that, for  $u\in C^{\infty}_{0}(\mathbb{H}^{n})$,
\begin{equation}\label{6.3}
\int_{\mathbb{H}^{n}}u\left(-\Delta_{\mathbb{H}}-\frac{(n-1)^{2}}{4}\right)\frac{|\Gamma(\gamma+\frac{1}{2}+i\sqrt{-\Delta_{\mathbb{H}}-\frac{(n-1)^{2}}{4}})|^{2}}
{|\Gamma(\frac{3}{2}+i\sqrt{-\Delta_{\mathbb{H}}-\frac{(n-1)^{2}}{4}})|^{2}}udV\geq S_{n,\gamma}\|u\|^{2}_{\frac{2n}{n-2\gamma}}.
\end{equation}
Furthermore, the inequality is  strict for nonzero $u$'s.
\end{lemma}
\begin{proof} For $\gamma=1$, we have $n=3$ because of $\frac{n-1}{2}\leq \gamma<\frac{n}{2}$. In this case,    inequality (\ref{6.3}) is nothing but the sharp Poincar\'e-Sobolev inequality on $\mathbb{H}^{3}$ (see \cite{bfl}).

Now we
assume $\gamma>1$.
According to the equivalence of Sobolev inequality and Hardy-Littlewoode-Sobolev inequality,
we only need to prove
\begin{equation}\label{6.4}
\left|\int_{\mathbb{H}^{n}}\int_{\mathbb{H}^{n}}f(x)\left(
\left(-\Delta_{\mathbb{H}}-\frac{(n-1)^{2}}{4}\right)^{-1}\ast K_{3/2,\gamma-1}\right)(x,y)g(y)dV_{x}dV_{y}\right|\leq \frac{\|f\|_{p}\|g\|_{p}}{S_{n,\gamma}},
\end{equation}
where $f,g\in C_{0}^{\infty}(\mathbb{B}^{n})$,  $\left(-\Delta_{\mathbb{H}}-\frac{(n-1)^{2}}{4}\right)^{-1}$ is the Green's function of $-\Delta_{\mathbb{H}}-\frac{(n-1)^{2}}{4}$,
$K_{3/2,\gamma-1}$ is given in (\ref{3.8}) and $p=\frac{2n}{n+2\gamma}$.

For simplicity, we set
\begin{align*}
\phi_{n}(\rho)=\left(-\Delta_{\mathbb{H}}-\frac{(n-1)^{2}}{4}\right)^{-1}.
\end{align*}
Without loss of generality, we can assume  $f\geq0$ and $g\geq0$.
By using  (\ref{2.11}) and (\ref{2.10}), we have
\begin{align*}
K_{3/2,\gamma-1}=&\frac{\Gamma(\frac{n}{2}+1)}{2^{n}\pi^{\frac{n}{2}}
\Gamma(\gamma-1)\Gamma(2+\gamma)}\left(\cosh\frac{\rho}{2}\right)^{-n-2} F\big(1+\frac{n}{2},2;2+\gamma;
(\cosh\frac{\rho}{2})^{-2}\big)\\
=&\frac{\Gamma(\frac{n}{2}+1)}{2^{n}\pi^{\frac{n}{2}}
\Gamma(\gamma-1)\Gamma(2+\gamma)}\left(\cosh\frac{\rho}{2}\right)^{-n-2} \left(\tanh\frac{\rho}{2}\right)^{2\gamma-2-n}\times\\
&F\big(\gamma+1-\frac{n}{2},\gamma;2+\gamma;
(\cosh\frac{\rho}{2})^{-2}\big)\\
\leq&\frac{\Gamma(\frac{n}{2}+1)}{2^{n}\pi^{\frac{n}{2}}
\Gamma(\gamma-1)\Gamma(2+\gamma)}\left(\cosh\frac{\rho}{2}\right)^{-n-2} \left(\tanh\frac{\rho}{2}\right)^{2\gamma-2-n}
F\big(\gamma+1-\frac{n}{2},\gamma;2+\gamma;
1\big)\\
=&\frac{\Gamma(\frac{n}{2}+1-\gamma)}{2^{n}\pi^{\frac{n}{2}}
\Gamma(\gamma-1)}\left(\cosh\frac{\rho}{2}\right)^{-2\gamma} \left(\sinh\frac{\rho}{2}\right)^{2\gamma-2-n}\\
\leq& \frac{\Gamma(\frac{n}{2}+1-\gamma)}{2^{n}\pi^{\frac{n}{2}}
\Gamma(\gamma-1)}\left(\cosh\frac{\rho}{2}\right)^{1-2\gamma} \left(\sinh\frac{\rho}{2}\right)^{2\gamma-2-n}.
\end{align*}
Therefore, by Lemma \ref{lm6.1}, we obtain
\begin{align*}
K_{3/2,\gamma-1}\leq& \frac{\Gamma(\frac{n}{2}-\gamma)}{2^{n}\pi^{\frac{n}{2}}
\Gamma(\gamma)}\left(-\Delta_{\mathbb{H}}-\frac{(n-1)^{2}}{4}\right)H_{0,\gamma}
\end{align*}
and thus
\begin{equation}\label{6.5}
  \begin{split}
&\int_{\mathbb{H}^{n}}\int_{\mathbb{H}^{n}}f(x)\left(
\phi_{n}\ast K_{3/2,\gamma-1}\right)(x,y)g(y)dV_{x}dV_{y}\\
\leq&\frac{\Gamma(\frac{n}{2}-\gamma)}{2^{n}\pi^{\frac{n}{2}}
\Gamma(\gamma)}\int_{\mathbb{H}^{n}}\int_{\mathbb{H}^{n}}f(x)\left(
\phi_{n}\ast \left(-\Delta_{\mathbb{H}}-\frac{(n-1)^{2}}{4}\right)H_{0,\gamma}\right)(x,y)g(y)dV_{x}dV_{y}.
  \end{split}
\end{equation}
By using (\ref{3.2}) and Lemma \ref{lm3.5}, we have that the right side of (\ref{6.5}) is equal to
\begin{align*}
&\frac{\Gamma(\frac{n}{2}-\gamma)}{2^{n}\pi^{\frac{n}{2}}
\Gamma(\gamma)}\int^{+\infty}_{-\infty}\int_{\mathbb{S}^{n-1}} \widehat{f}(\lambda,\zeta)
 \frac{1}{\lambda^{2}}\lambda^{2}\widehat{H_{0,\gamma}}(\lambda)\widehat{g}(\lambda,\zeta)|\mathfrak{c}(\lambda)|^{-2}d\lambda d\sigma(\zeta)\\
 =&\frac{\Gamma(\frac{n}{2}-\gamma)}{2^{n}\pi^{\frac{n}{2}}
\Gamma(\gamma)}\int^{+\infty}_{-\infty}\int_{\mathbb{S}^{n-1}} \widehat{f}(\lambda,\zeta)
 \widehat{H_{0,\gamma}}(\lambda)\widehat{g}(\lambda,\zeta)|\mathfrak{c}(\lambda)|^{-2}d\lambda d\sigma(\zeta)\\
 =&\frac{\Gamma(\frac{n}{2}-\gamma)}{2^{n}\pi^{\frac{n}{2}}
\Gamma(\gamma)}\int_{\mathbb{H}^{n}}\int_{\mathbb{H}^{n}}f(x)H_{0,\gamma}(\cosh\rho(x,y))g(y)dV_{x}dV_{y}.
\end{align*}
Therefore, by Theorem \ref{th3.8},
we have
\begin{align*}
&\int_{\mathbb{H}^{n}}\int_{\mathbb{H}^{n}}f(x)\left(
\left(-\Delta_{\mathbb{H}}-\frac{(n-1)^{2}}{4}\right)^{-1}\ast K_{3/2,\gamma-1}\right)(x,y)g(y)dV_{x}dV_{y}\\
\leq&\frac{\Gamma(\frac{n}{2}-\gamma)}{2^{n}\pi^{\frac{n}{2}}
\Gamma(\gamma)}\int_{\mathbb{H}^{n}}\int_{\mathbb{H}^{n}}f(x)H_{0,\gamma}(\cosh\rho(x,y))g(y)dV_{x}dV_{y}\\
\leq&\frac{\Gamma(\frac{n}{2}-\gamma)}{2^{n}\pi^{\frac{n}{2}}
\Gamma(\gamma)}\int_{\mathbb{H}^{n}}\int_{\mathbb{H}^{n}}f(x)\left(\sinh\frac{\rho}{2}\right)^{2\gamma-n}g(y)dV_{x}dV_{y}\\
\leq&\frac{\|f\|_{p}\|g\|_{p}}{S_{n,\gamma}}.
\end{align*}
This completes the proof of Lemma \ref{lm6.2}.
\end{proof}

\textbf{Proof of Theorem \ref{th1.8}}. By Plancherel formula, we get
\begin{align}\nonumber
&\int_{\mathbb{H}^{n}}u\widetilde{P}_{\gamma}udV- \frac{\Gamma(\gamma+\frac{1}{2})^{2}}
{\Gamma(\frac{1}{2})^{2}}\int_{\mathbb{H}^{n}}u^{2}dV\\
\label{6.6}
=&\int^{+\infty}_{-\infty}\int_{\mathbb{S}^{n-1}}
\left(\frac{|\Gamma(\gamma+\frac{1}{2}+i\lambda)|^{2}}{|\Gamma(\frac{1}{2}+i\lambda)|^{2}}-
\frac{|\Gamma(\gamma+\frac{1}{2})|^{2}}{|\Gamma(\frac{1}{2})|^{2}}\right)\widehat{u}(\lambda,\zeta)|^{2}|\mathfrak{c}(\lambda)|^{-2}d\lambda d\sigma(\zeta).
\end{align}
On the other hand, by using  (\ref{2.20}), we obtain
\begin{align}\nonumber
&\frac{|\Gamma(\gamma+\frac{1}{2}+i\lambda)|^{2}}{|\Gamma(\frac{1}{2}+i\lambda)|^{2}}-
\frac{|\Gamma(\gamma+\frac{1}{2})|^{2}}{|\Gamma(\frac{1}{2})|^{2}}\\
\nonumber
=&(\lambda^{2}+1/4)\frac{|\Gamma(\gamma+\frac{1}{2}+i\lambda)|^{2}}{|\Gamma(\frac{3}{2}+i\lambda)|^{2}}-\frac{1}{4}
\frac{|\Gamma(\gamma+\frac{1}{2})|^{2}}{|\Gamma(\frac{1}{2})|^{2}}\\
\nonumber
=&\lambda^{2}\frac{|\Gamma(\gamma+\frac{1}{2}+i\lambda)|^{2}}{|\Gamma(\frac{3}{2}+i\lambda)|^{2}}+\frac{1}{4}\left(
\frac{|\Gamma(\gamma+\frac{1}{2}+i\lambda)|^{2}}{|\Gamma(\frac{3}{2}+i\lambda)|^{2}}-
\frac{|\Gamma(\gamma+\frac{1}{2})|^{2}}{|\Gamma(\frac{3}{2})|^{2}}\right)\\
\label{6.7}
\geq&\lambda^{2}\frac{|\Gamma(\gamma+\frac{1}{2}+i\lambda)|^{2}}{|\Gamma(\frac{3}{2}+i\lambda)|^{2}}.
\end{align}
Substituting (\ref{6.7}) into (\ref{6.6}) and using  Lemma \ref{lm6.2}, we get
\begin{align*}
&\int_{\mathbb{H}^{n}}u\widetilde{P}_{\gamma}udV- \frac{\Gamma(\gamma+\frac{1}{2})^{2}}
{\Gamma(\frac{1}{2})^{2}}\int_{\mathbb{H}^{n}}u^{2}dV\\
\geq&\int^{+\infty}_{-\infty}\int_{\mathbb{S}^{n-1}}
\lambda^{2}\frac{|\Gamma(\gamma+\frac{1}{2}+i\lambda)|^{2}}{|\Gamma(\frac{3}{2}+i\lambda)|^{2}}\widehat{u}(\lambda,\zeta)|^{2}|\mathfrak{c}(\lambda)|^{-2}d\lambda d\sigma(\zeta)\\
=& \int_{\mathbb{H}^{n}}u(-\Delta_{\mathbb{H}}-\frac{(n-1)^{2}}{4})\frac{|\Gamma(\gamma+\frac{1}{2}+i\sqrt{-\Delta_{\mathbb{H}}-\frac{(n-1)^{2}}{4}})|^{2}}
{|\Gamma(\frac{3}{2}+i\sqrt{-\Delta_{\mathbb{H}}-\frac{(n-1)^{2}}{4}})|^{2}}udV\\
\geq& S_{n,\gamma}\left(\int_{\mathbb{H}^{n}}|u|^{\frac{2n}{n-2\gamma}}dV\right)^{\frac{n-2\gamma}{n}}.
\end{align*}
Furthermore, the inequality is  strict for nonzero $u$'s.
This completes the proof of Theorem \ref{th1.8}.

\medskip

Next we shall prove  Theorem \ref{th1.4}. We note that the proof of Theorem \ref{th1.4} is more complicated than that of Theorem \ref{th1.8}.
We need  the following inequality of Gamma function.
\begin{proposition}\label{p6.3}
Let $k\in\mathbb{N} \setminus\{0\}$.
It holds that, for $2k\leq\gamma\leq 2k+1$ and $\lambda\in \mathbb{R}$,
\begin{align}\label{a6.8}
2^{2\gamma}\frac{|\Gamma(\frac{3+2\gamma}{4}+\frac{i}{2}\lambda)|^{2}}
{|\Gamma(\frac{3-2\gamma}{4}+\frac{i}{2}\lambda)|^{2}}-2^{2\gamma}\frac{|\Gamma(\frac{3+2\gamma}{4})|^{2}}
{|\Gamma(\frac{3-2\gamma}{4})|^{2}}
\geq&\frac{|\Gamma(\gamma+i\lambda)|^{2}}
{|\Gamma(i\lambda)|^{2}}.
\end{align}
\end{proposition}

We first show the case $2\leq \gamma\leq 3$. The proof  depends on the following four lemmas.

\begin{lemma}\label{lm6.4}
It holds that, for $2\leq\gamma\leq3$,
  \begin{align*}
2^{2\gamma}\frac{|\Gamma(\frac{3+2\gamma}{4}+\frac{i}{2}\lambda)|^{2}}
{|\Gamma(\frac{3-2\gamma}{4}+\frac{i}{2}\lambda)|^{2}}-2^{2\gamma}\frac{|\Gamma(\frac{3+2\gamma}{4})|^{2}}
{|\Gamma(\frac{3-2\gamma}{4})|^{2}}
\geq\lambda^{2} \frac{\lambda^{2}+(\gamma-\frac{3}{2})^{2}+(\gamma-\frac{7}{2})^{2}}
{[\lambda^{2}+(\gamma-\frac{3}{2})^{2}][\lambda^{2}+(\gamma-\frac{7}{2})^{2}]}
\frac{|\Gamma(\gamma+\frac{1}{2}+i\lambda)|^{2}}
{|\Gamma(\frac{1}{2}+i\lambda)|^{2}}.
\end{align*}
\end{lemma}
\begin{proof}
We compute
\begin{align}\nonumber
2^{2\gamma}\frac{|\Gamma(\frac{3+2\gamma}{4}+\frac{i}{2}\lambda)|^{2}}
{|\Gamma(\frac{3-2\gamma}{4}+\frac{i}{2}\lambda)|^{2}}=&
2^{2\gamma}\frac{[\lambda^{2}+(\gamma-\frac{3}{2})^{2}][\lambda^{2}+(\gamma-\frac{7}{2})^{2}]}
{[\lambda^{2}+(\gamma-\frac{3}{2})^{2}][\lambda^{2}+(\gamma-\frac{7}{2})^{2}]}\frac{|\Gamma(\frac{3+2\gamma}{4}+\frac{i}{2}\lambda)|^{2}}
{|\Gamma(\frac{3-2\gamma}{4}+\frac{i}{2}\lambda)|^{2}}\\
\nonumber
=&2^{2\gamma}\lambda^{2} \frac{\lambda^{2}+(\gamma-\frac{3}{2})^{2}+(\gamma-\frac{7}{2})^{2}}
{[\lambda^{2}+(\gamma-\frac{3}{2})^{2}][\lambda^{2}+(\gamma-\frac{7}{2})^{2}]}\frac{|\Gamma(\frac{3+2\gamma}{4}+\frac{i}{2}\lambda)|^{2}}
{|\Gamma(\frac{3-2\gamma}{4}+\frac{i}{2}\lambda)|^{2}}+\\
\nonumber
& 2^{2\gamma}\frac{(\gamma-\frac{3}{2})^{2}(\gamma-\frac{7}{2})^{2}}
{[\lambda^{2}+(\gamma-\frac{3}{2})^{2}][\lambda^{2}+(\gamma-\frac{7}{2})^{2}]}\frac{|\Gamma(\frac{3+2\gamma}{4}+\frac{i}{2}\lambda)|^{2}}
{|\Gamma(\frac{3-2\gamma}{4}+\frac{i}{2}\lambda)|^{2}}\\
\nonumber
=&2^{2\gamma}\lambda^{2} \frac{\lambda^{2}+(\gamma-\frac{3}{2})^{2}+(\gamma-\frac{7}{2})^{2}}
{[\lambda^{2}+(\gamma-\frac{3}{2})^{2}][\lambda^{2}+(\gamma-\frac{7}{2})^{2}]}\frac{|\Gamma(\frac{3+2\gamma}{4}+\frac{i}{2}\lambda)|^{2}}
{|\Gamma(\frac{3-2\gamma}{4}+\frac{i}{2}\lambda)|^{2}}+\\
\nonumber
& 2^{2\gamma-4}(\gamma-3/2)^{2}(\gamma-7/2)^{2}\frac{|\Gamma(\frac{3+2\gamma}{4}+\frac{i}{2}\lambda)|^{2}}
{|\Gamma(\frac{11-2\gamma}{4}+\frac{i}{2}\lambda)|^{2}}.
\end{align}
Therefore, by using  (\ref{2.20}), we get
\begin{align}\nonumber
&2^{2\gamma}\frac{|\Gamma(\frac{3+2\gamma}{4}+\frac{i}{2}\lambda)|^{2}}
{|\Gamma(\frac{3-2\gamma}{4}+\frac{i}{2}\lambda)|^{2}}-2^{2\gamma}\frac{|\Gamma(\frac{3+2\gamma}{4})|^{2}}
{|\Gamma(\frac{3-2\gamma}{4})|^{2}}\\
\nonumber
=&2^{2\gamma}\lambda^{2} \frac{\lambda^{2}+(\gamma-\frac{3}{2})^{2}+(\gamma-\frac{7}{2})^{2}}
{[\lambda^{2}+(\gamma-\frac{3}{2})^{2}][\lambda^{2}+(\gamma-\frac{7}{2})^{2}]}\frac{|\Gamma(\frac{3+2\gamma}{4}+\frac{i}{2}\lambda)|^{2}}
{|\Gamma(\frac{3-2\gamma}{4}+\frac{i}{2}\lambda)|^{2}}+\\
\nonumber
\nonumber
& 2^{2\gamma-4}(\gamma-3/2)^{2}(\gamma-7/2)^{2}\left(\frac{|\Gamma(\frac{3+2\gamma}{4}+\frac{i}{2}\lambda)|^{2}}
{|\Gamma(\frac{11-2\gamma}{4}+\frac{i}{2}\lambda)|^{2}}-\frac{|\Gamma(\frac{3+2\gamma}{4})|^{2}}
{|\Gamma(\frac{11-2\gamma}{4})|^{2}}\right)\\
\label{6.8}
\geq&2^{2\gamma}\lambda^{2} \frac{\lambda^{2}+(\gamma-\frac{3}{2})^{2}+(\gamma-\frac{7}{2})^{2}}
{[\lambda^{2}+(\gamma-\frac{3}{2})^{2}][\lambda^{2}+(\gamma-\frac{7}{2})^{2}]}\frac{|\Gamma(\frac{3+2\gamma}{4}+\frac{i}{2}\lambda)|^{2}}
{|\Gamma(\frac{3-2\gamma}{4}+\frac{i}{2}\lambda)|^{2}}.
\end{align}

On the other hand, since $2\leq\gamma\leq3$,  we have $\sin\gamma\pi\geq0$. Therefore, by Lemma \ref{lm4.2},
\begin{align}\label{6.9}
\frac{|\Gamma(\frac{3+2\gamma}{4}+\frac{i}{2}\lambda)|^{2}}
{|\Gamma(\frac{3-2\gamma}{4}+\frac{i}{2}\lambda)|^{2}}\geq \frac{|\Gamma(\gamma+\frac{1}{2}+i\lambda)|^{2}}
{|\Gamma(\frac{1}{2}+i\lambda)|^{2}}.
\end{align}
The desired result follows by combining (\ref{6.8}) and (\ref{6.9}).
\end{proof}

\begin{lemma}\label{lm6.5}
It holds that, for $2\leq \gamma\leq3$ and $\lambda\in\mathbb{R}$,
  \begin{align}\label{6.10}
 \frac{\lambda^{2}+(\gamma-\frac{3}{2})^{2}+(\gamma-\frac{7}{2})^{2}}
{[\lambda^{2}+(\gamma-\frac{3}{2})^{2}][\lambda^{2}+(\gamma-\frac{7}{2})^{2}]}
>\frac{\lambda^{2}+1}{(\lambda^{2}+\frac{1}{4})(\lambda^{2}+\frac{9}{4})}.
\end{align}
\end{lemma}
\begin{proof}
For simplicity, we set
\begin{align}\label{a6.11}
A_{1,\gamma}=(\gamma-\frac{3}{2})^{2},\;\;A_{2,\gamma}=(\gamma-\frac{7}{2})^{2}.
\end{align}
Then (\ref{6.10}) is is equivalent to (here we replace $\lambda^{2}$ by $\lambda$)
\begin{align*}
(\lambda+A_{1,\gamma}+A_{2,\gamma})(\lambda+\frac{1}{4})(\lambda+\frac{9}{4})> (\lambda+A_{1,\gamma})(\lambda+A_{2,\gamma}) (\lambda+1),\;\;\lambda\geq0.
\end{align*}
We compute
\begin{align}\nonumber
&(\lambda+A_{1,\gamma}+A_{2,\gamma})(\lambda+\frac{1}{4})(\lambda+\frac{9}{4})- (\lambda+A_{1,\gamma})(\lambda+A_{2,\gamma}) (\lambda+1)\\
\label{6.11}
=&\frac{3}{2}\lambda^{2}+\left[\frac{3}{2}(A_{1,\gamma}+A_{2,\gamma})+\frac{9}{16}-A_{1,\gamma}A_{2,\gamma}\right]\lambda+\frac{9}{16}(A_{1,\gamma}+A_{2,\gamma})
-A_{1,\gamma}A_{2,\gamma}.
\end{align}
We claim
\begin{align}\label{6.12}
A_{1,\gamma}A_{2,\gamma}< \frac{9}{16}(A_{1,\gamma}+A_{2,\gamma}).
\end{align}
In fact,
\begin{align}\nonumber
A_{1,\gamma}A_{2,\gamma}-  \frac{9}{16}(A_{1,\gamma}+A_{2,\gamma})=&\left[(\gamma-\frac{3}{2})(\gamma-\frac{7}{2})\right]^{2}-\frac{9}{16}\left[
(\gamma-\frac{3}{2})^{2}+(\gamma-\frac{7}{2})^{2}\right]\\
\label{6.13}
=&\left[(\gamma-\frac{5}{2})^{2}-1\right]^{2}-\frac{9}{8}\left[
(\gamma-\frac{5}{2})^{2}-1\right]-\frac{9}{4}.
\end{align}
Since $2\leq \gamma\leq3$, we have
\begin{align}\label{6.14}
-1\leq (\gamma-\frac{5}{2})^{2}-1\leq-3/4.
\end{align}
Substituting (\ref{6.14}) into (\ref{6.13}), we get
\begin{align*}
A_{1,\gamma}A_{2,\gamma}-  \frac{9}{16}(A_{1,\gamma}+A_{2,\gamma})
\leq 1+\frac{9}{8}-\frac{9}{4}=-\frac{1}{8}<0.
\end{align*}
This proves the claim.

Substituting (\ref{6.12}) into (\ref{6.11}), we obtain
\begin{align*}
&(\lambda+A_{1,\gamma}+A_{2,\gamma})(\lambda+\frac{1}{4})(\lambda+\frac{9}{4})- (\lambda+A_{1,\gamma})(\lambda+A_{2,\gamma}) (\lambda+1)\\
>&\frac{3}{2}\lambda^{2}+\left[\frac{3}{2}(A_{1,\gamma}+A_{2,\gamma})+\frac{9}{16}-\frac{9}{16}(A_{1,\gamma}+A_{2,\gamma})\right]\lambda\\
=&\frac{3}{2}\lambda^{2}+\left[\frac{15}{16}(A_{1,\gamma}+A_{2,\gamma})+\frac{9}{16}\right]\lambda\\
\geq&0.
\end{align*}
This completes the proof of Lemma \ref{lm6.5}.
\end{proof}

Combining Lemma \ref{lm6.4} and \ref{lm6.5} yields the following corollary:
\begin{corollary}\label{co6.6}
It holds that, for $2\leq\gamma\leq3$,
  \begin{align*}
2^{2\gamma}\frac{|\Gamma(\frac{3+2\gamma}{4}+\frac{i}{2}\lambda)|^{2}}
{|\Gamma(\frac{3-2\gamma}{4}+\frac{i}{2}\lambda)|^{2}}-2^{2\gamma}\frac{|\Gamma(\frac{3+2\gamma}{4})|^{2}}
{|\Gamma(\frac{3-2\gamma}{4})|^{2}}
\geq \lambda^{2}(\lambda^{2}+1)\frac{|\Gamma(\gamma+\frac{1}{2}+i\lambda)|^{2}}
{|\Gamma(\frac{5}{2}+i\lambda)|^{2}}.
\end{align*}
\end{corollary}
\begin{proof}
By Lemma \ref{lm6.4} and \ref{lm6.5}, we have
  \begin{align*}
2^{2\gamma}\frac{|\Gamma(\frac{3+2\gamma}{4}+\frac{i}{2}\lambda)|^{2}}
{|\Gamma(\frac{3-2\gamma}{4}+\frac{i}{2}\lambda)|^{2}}-2^{2\gamma}\frac{|\Gamma(\frac{3+2\gamma}{4})|^{2}}
{|\Gamma(\frac{3-2\gamma}{4})|^{2}}
\geq &\lambda^{2}\frac{\lambda^{2}+1}{(\lambda^{2}+\frac{1}{4})(\lambda^{2}+\frac{9}{4})}
\frac{|\Gamma(\gamma+\frac{1}{2}+i\lambda)|^{2}}
{|\Gamma(\frac{1}{2}+i\lambda)|^{2}}\\
=&\lambda^{2}(\lambda^{2}+1)\frac{|\Gamma(\gamma+\frac{1}{2}+i\lambda)|^{2}}
{|\Gamma(\frac{5}{2}+i\lambda)|^{2}}.
\end{align*}

\end{proof}

\begin{lemma}\label{lm6.7}
It holds that, for $ \gamma\geq2$ and $\lambda^{2}\leq 5$,
\begin{align*}
\frac{|\Gamma(\gamma+\frac{1}{2}+i\lambda)|^{2}}
{|\Gamma(\frac{5}{2}+i\lambda)|^{2}}\geq \frac{|\Gamma(\gamma+i\lambda)|^{2}}
{|\Gamma(2+i\lambda)|^{2}}
\end{align*}
\end{lemma}
\begin{proof}
Set
\begin{align*}
f_{\lambda}(\gamma)=\frac{|\Gamma(\gamma+\frac{1}{2}+i\lambda)|^{2}}
{|\Gamma(\gamma+i\lambda)|^{2}},\;\;\gamma\geq2.
\end{align*}
By (\ref{2.13}), we have
\begin{align*}
\ln f_{\lambda}(\gamma)=2\ln\frac{\Gamma(\gamma+\frac{1}{2})}
{\Gamma(\gamma)}+\sum_{k=0}^{\infty}\ln\left(1+\frac{\lambda^{2}}{(k+\gamma)^{2}}\right)-
\sum_{k=0}^{\infty}\ln\left(1+\frac{\lambda^{2}}{(k+\gamma+\frac{1}{2})^{2}}\right).
\end{align*}
By using (\ref{2.18}), we get
\begin{align*}
\frac{d}{d\gamma}\ln\frac{\Gamma(\gamma+\frac{1}{2})}
{\Gamma(\gamma)}=&
-\sum_{k=0}^{\infty}\left(\frac{1}{k+\gamma+\frac{1}{2}}-\frac{1}{k+\gamma}\right)
=\frac{1}{2}\sum_{k=0}^{\infty}\frac{1}{(k+\gamma+\frac{1}{2})(k+\gamma)}.
\end{align*}
Therefore,
\begin{equation}\label{6.15}
  \begin{split}
\frac{d}{d\gamma}\ln f_{\lambda}(\gamma)=&\sum_{k=0}^{\infty}\frac{1}{(k+\gamma+\frac{1}{2})(k+\gamma)}
-\sum_{k=0}^{\infty}\frac{2\lambda^{2}}{[\lambda^{2}+(k+\gamma)^{2}](k+\gamma)}+\\
&\sum_{k=0}^{\infty}\frac{2\lambda^{2}}{[\lambda^{2}+(k+\gamma+\frac{1}{2})^{2}](k+\gamma+\frac{1}{2})}.
  \end{split}
\end{equation}
We compute
\begin{equation}\label{6.16}
  \begin{split}
&\frac{1}{(k+\gamma+\frac{1}{2})(k+\gamma)}-\frac{2\lambda^{2}}{[\lambda^{2}+(k+\gamma)^{2}](k+\gamma)}+ \frac{2\lambda^{2}}{[\lambda^{2}+(k+\gamma+\frac{1}{2})^{2}](k+\gamma+\frac{1}{2})}
\\
=&\frac{k+\gamma-2\lambda^{2}}{[\lambda^{2}+(k+\gamma)^{2}](k+\gamma+\frac{1}{2})}+\frac{2\lambda^{2}}{[\lambda^{2}+(k+\gamma+\frac{1}{2})^{2}](k+\gamma+\frac{1}{2})}\\
=&\frac{(k+\gamma+\frac{1}{2})(k+\gamma)-\lambda^{2}}
{[\lambda^{2}+(k+\gamma)^{2}][\lambda^{2}+(k+\gamma+\frac{1}{2})^{2}]}.
  \end{split}
\end{equation}
Substituting (\ref{6.16}) into (\ref{6.15}), we obtain,
\begin{align*}
\frac{d}{d\gamma}\ln f_{\lambda}(\gamma)=&\sum_{k=0}^{\infty}\frac{(k+\gamma+\frac{1}{2})(k+\gamma)-\lambda^{2}}
{[\lambda^{2}+(k+\gamma)^{2}][\lambda^{2}+(k+\gamma+\frac{1}{2})^{2}]}\\
\geq&\left[\gamma(\gamma+\frac{1}{2})-\lambda^{2}\right]\sum_{k=0}^{\infty}\frac{1}
{[\lambda^{2}+(k+\gamma)^{2}][\lambda^{2}+(k+\gamma+\frac{1}{2})^{2}]}
\end{align*}
Therefore, $\frac{d}{d\gamma} \ln f_{\lambda}(\gamma)\geq0$ since $\lambda^{2}\leq 5\leq\gamma(\gamma+\frac{1}{2})$.
Thus $f_{\lambda}(\gamma)$ is  increasing and we have
\begin{align*}
f_{\lambda}(\gamma)\geq f_{\lambda}(2)= \frac{|\Gamma(\frac{5}{2}+i\lambda)|^{2}}
{|\Gamma(2+i\lambda)|^{2}},\;\;\gamma\geq2.
\end{align*}
This completes the proof of Lemma \ref{lm6.7}.
\end{proof}

\begin{lemma}\label{lm6.8}
It holds that, for $2\leq \gamma\leq3$ and $\lambda^{2}\geq 5$,
\begin{align}\label{6.17}
\lambda^{2} \frac{\lambda^{2}+(\gamma-\frac{3}{2})^{2}+(\gamma-\frac{7}{2})^{2}}
{[\lambda^{2}+(\gamma-\frac{3}{2})^{2}][\lambda^{2}+(\gamma-\frac{7}{2})^{2}]}
\sqrt{1+\frac{(\gamma-1)^{2}}{\lambda^{2}}}\geq 1.
\end{align}
\end{lemma}
\begin{proof}
For simplicity, we set
\begin{align*}
X_{\gamma}=\lambda^{2}\left (\lambda^{2}+(\gamma-\frac{3}{2})^{2}+(\gamma-\frac{7}{2})^{2}\right).
\end{align*}
Then  (\ref{6.17}) is is equivalent to
\begin{align}\label{6.18}
X_{\gamma}^{2}\left(1+\frac{(\gamma-1)^{2}}{\lambda^{2}}\right)\geq (X_{\gamma}+A_{1,\gamma}A_{2,\gamma})^{2},
\end{align}
where $A_{1,\gamma}$ and $A_{2,\gamma}$ are given in (\ref{a6.11}).
Since $2\leq \gamma\leq 3$, we have
\begin{equation}\label{6.20}
  \begin{split}
 &X_{\gamma}^{2}\left(1+\frac{(\gamma-1)^{2}}{\lambda^{2}}\right)- (X_{\gamma}+A_{1,\gamma}A_{2,\gamma})^{2}\\
 =&X_{\gamma}^{2}\frac{(\gamma-1)^{2}}{\lambda^{2}}-2A_{1,\gamma}A_{2,\gamma}X_{\gamma}-(A_{1,\gamma}A_{2,\gamma})^{2}\\
 \geq&X_{\gamma}^{2}\frac{1}{\lambda^{2}}-2A_{1,\gamma}A_{2,\gamma}X_{\gamma}-(A_{1,\gamma}A_{2,\gamma})^{2}\\
 =&X_{\gamma}\left[\lambda^{2}+A_{1,\gamma}+A_{2,\gamma}-2A_{1,\gamma}A_{2,\gamma}\right]-(A_{1,\gamma}A_{2,\gamma})^{2}.
  \end{split}
\end{equation}
Notice that if $2\leq \gamma\leq3$, then
\begin{align*}
A_{1,\gamma}A_{2,\gamma}=\left[(\gamma-\frac{3}{2})(\gamma-\frac{7}{2})\right]^{2}=\left[(\gamma-\frac{5}{2})^{2}-1\right]^{2}\in[9/16,1].
\end{align*}
So we have, for  $\lambda^{2}\geq 5$,
\begin{equation}\label{6.21}
  \begin{split}
 &X_{\gamma}\left[\lambda^{2}+A_{1,\gamma}+A_{2,\gamma}-2A_{1,\gamma}A_{2,\gamma}\right]-(A_{1,\gamma}A_{2,\gamma})^{2}\\
 \geq&X_{\gamma}\left[\lambda^{2}+A_{1,\gamma}+A_{2,\gamma}-2\right]-1\\
=&\lambda^{2}\left (\lambda^{2}+(\gamma-\frac{3}{2})^{2}+(\gamma-\frac{7}{2})^{2}\right)\left[\lambda^{2}+A_{1,\gamma}+A_{2,\gamma}-2\right]-1\\
 >&\lambda^{2}\cdot\lambda^{2}\cdot(\lambda^{2}-2)-1\\
 >&0.
  \end{split}
\end{equation}
The desired result follows by combining (\ref{6.20}) and (\ref{6.21}).
\end{proof}

Now we can show the case $2\leq \gamma\leq 3$.

\begin{proposition} \label{pro6.9}
Let $2\leq\gamma\leq3$.
It holds that
  \begin{align}\label{6.22}
2^{2\gamma}\frac{|\Gamma(\frac{3+2\gamma}{4}+\frac{i}{2}\lambda)|^{2}}
{|\Gamma(\frac{3-2\gamma}{4}+\frac{i}{2}\lambda)|^{2}}-2^{2\gamma}\frac{|\Gamma(\frac{3+2\gamma}{4})|^{2}}
{|\Gamma(\frac{3-2\gamma}{4})|^{2}}
\geq \frac{|\Gamma(\gamma+i\lambda))|^{2}}
{|\Gamma(i\lambda)|^{2}}.
\end{align}
\end{proposition}
\begin{proof}
The proof is divided into two parts.

\medskip
Case 1: $\lambda^{2}\leq 5$. By Corollary \ref{co6.6} and Lemma \ref{lm6.7}, we have
\begin{align*}
2^{2\gamma}\frac{|\Gamma(\frac{3+2\gamma}{4}+\frac{i}{2}\lambda)|^{2}}
{|\Gamma(\frac{3-2\gamma}{4}+\frac{i}{2}\lambda)|^{2}}-2^{2\gamma}\frac{|\Gamma(\frac{3+2\gamma}{4})|^{2}}
{|\Gamma(\frac{3-2\gamma}{4})|^{2}}
\geq&\lambda^{2}(\lambda^{2}+1)\frac{|\Gamma(\gamma+\frac{1}{2}+i\lambda)|^{2}}
{|\Gamma(\frac{5}{2}+i\lambda)|^{2}}\\
\geq&\lambda^{2}(\lambda^{2}+1)\frac{|\Gamma(\gamma+i\lambda)|^{2}}
{|\Gamma(2+i\lambda)|^{2}}\\
=&\frac{|\Gamma(\gamma+i\lambda)|^{2}}
{|\Gamma(i\lambda)|^{2}}.
\end{align*}

Case 2: $\lambda^{2}\geq 5$.
By using (\ref{2.21}), we have
\begin{align*}
|\Gamma(\gamma+\frac{1}{2}+i\lambda)|^{2}=&|\Gamma(\gamma-1+i\lambda+\frac{3}{2})|^{2}\\
\geq&|\gamma-1+i\lambda|^{3}|\Gamma(\gamma-1+i\lambda)|^{2}\\
=&\sqrt{\lambda^{2}+(\gamma-1)^{2}}|\Gamma(\gamma+i\lambda)|^{2}.
\end{align*}
Therefore, by Lemmas \ref{lm6.4} and \ref{lm6.8}, we obtain
\begin{align}\nonumber
&2^{2\gamma}\frac{|\Gamma(\frac{3+2\gamma}{4}+\frac{i}{2}\lambda)|^{2}}
{|\Gamma(\frac{3-2\gamma}{4}+\frac{i}{2}\lambda)|^{2}}-2^{2\gamma}\frac{|\Gamma(\frac{3+2\gamma}{4})|^{2}}
{|\Gamma(\frac{3-2\gamma}{4})|^{2}}\\
\nonumber
\geq&\lambda^{2} \frac{\lambda^{2}+(\gamma-\frac{3}{2})^{2}+(\gamma-\frac{7}{2})^{2}}
{[\lambda^{2}+(\gamma-\frac{3}{2})^{2}][\lambda^{2}+(\gamma-\frac{7}{2})^{2}]}
\frac{|\Gamma(\gamma+\frac{1}{2}+i\lambda)|^{2}}
{|\Gamma(\frac{1}{2}+i\lambda)|^{2}}\\
\nonumber
\geq& \lambda^{2} \frac{\lambda^{2}+(\gamma-\frac{3}{2})^{2}+(\gamma-\frac{7}{2})^{2}}
{[\lambda^{2}+(\gamma-\frac{3}{2})^{2}][\lambda^{2}+(\gamma-\frac{7}{2})^{2}]}
\sqrt{\lambda^{2}+(\gamma-1)^{2}}\frac{|\Gamma(\gamma+i\lambda)|^{2}}
{|\Gamma(\frac{1}{2}+i\lambda)|^{2}}\\
\label{6.23}
\geq&|\lambda|\frac{|\Gamma(\gamma+i\lambda)|^{2}}
{|\Gamma(\frac{1}{2}+i\lambda)|^{2}}.
\end{align}
Notice that
\begin{align}\label{6.24}
\frac{|\lambda|}
{|\Gamma(\frac{1}{2}+i\lambda)|^{2}}=\frac{|\lambda|\cosh\pi\lambda}{\pi}\geq\frac{\lambda\sinh\pi\lambda}{\pi}=\frac{1}{|\Gamma(i\lambda)|^{2}}.
\end{align}
We obtain, by substituting (\ref{6.24}) into (\ref{6.23}),
\begin{align*}
2^{2\gamma}\frac{|\Gamma(\frac{3+2\gamma}{4}+\frac{i}{2}\lambda)|^{2}}
{|\Gamma(\frac{3-2\gamma}{4}+\frac{i}{2}\lambda)|^{2}}-2^{2\gamma}\frac{|\Gamma(\frac{3+2\gamma}{4})|^{2}}
{|\Gamma(\frac{3-2\gamma}{4})|^{2}}\geq \frac{|\Gamma(\gamma+i\lambda)|^{2}}
{|\Gamma(i\lambda)|^{2}}.
\end{align*}
This completes the proof of Proposition \ref{pro6.9}.
\end{proof}

Now we can give the  proof of  Proposition \ref{p6.3}.

\medskip

\textbf{Proof of Proposition \ref{p6.3}}. We shall prove it by induction.

By Proposition \ref{pro6.9}, inequality (\ref{a6.8}) is valid for $2\leq \gamma\leq 3$. Now assume it is valid for $2k\leq \gamma\leq 2k+1$. Then for $2k+2\leq \gamma\leq 2k+3$,
 \begin{align*}
&2^{2\gamma}\frac{|\Gamma(\frac{3+2\gamma}{4}+\frac{i}{2}\lambda)|^{2}}
{|\Gamma(\frac{3-2\gamma}{4}+\frac{i}{2}\lambda)|^{2}}-2^{2\gamma}\frac{|\Gamma(\frac{3+2\gamma}{4})|^{2}}
{|\Gamma(\frac{3-2\gamma}{4})|^{2}}\\
=&\left(\lambda^{2}+\frac{(2\gamma-1)^{2}}{4}\right)\left(\lambda^{2}+\frac{(2\gamma+1)^{2}}{4}\right)
2^{2(\gamma-2)}\frac{|\Gamma(\frac{3+2(\gamma-2)}{4}+\frac{i}{2}\lambda)|^{2}}
{|\Gamma(\frac{3-2(\gamma-2)}{4}+\frac{i}{2}\lambda)|^{2}}-2^{2\gamma}\frac{|\Gamma(\frac{3+2\gamma}{4})|^{2}}
{|\Gamma(\frac{3-2\gamma}{4})|^{2}}\\
\geq&\left(\lambda^{2}+\frac{(2\gamma-1)^{2}}{4}\right)\left(\lambda^{2}+\frac{(2\gamma+1)^{2}}{4}\right)
\left(2^{2(\gamma-2)}\frac{|\Gamma(\frac{3+2(\gamma-2)}{4}+\frac{i}{2}\lambda)|^{2}}
{|\Gamma(\frac{3-2(\gamma-2)}{4}+\frac{i}{2}\lambda)|^{2}}
-
2^{2(\gamma-2)}\frac{|\Gamma(\frac{3+2(\gamma-2)}{4})|^{2}}
{|\Gamma(\frac{3-2(\gamma-2)}{4})|^{2}}\right)\\
\geq&\left(\lambda^{2}+\frac{(2\gamma-1)^{2}}{4}\right)\left(\lambda^{2}+\frac{(2\gamma+1)^{2}}{4}\right)\frac{|\Gamma(\gamma-2+i\lambda))|^{2}}
{|\Gamma(i\lambda)|^{2}}\\
\geq&\left(\lambda^{2}+(\gamma-1)^{2}\right)\left(\lambda^{2}+(\gamma-2)^{2}\right)\frac{|\Gamma(\gamma-2+i\lambda))|^{2}}
{|\Gamma(i\lambda)|^{2}}\\
=&\frac{|\Gamma(\gamma+i\lambda))|^{2}}
{|\Gamma(i\lambda)|^{2}}.
 \end{align*}
This completes the proof of Proposition \ref{p6.3}.

\medskip

Finally,  we  give the proof of Theorem \ref{th1.4}.

\medskip

\textbf{Proof of Theorem \ref{th1.4}}. The proof  is divided into two parts.

\medskip
Case 1: $n=4k+1, 4k+2$. In this case, we have $2k\leq \gamma\leq 2k+1$. By Proposition \ref{p6.3} and Theorem \ref{th3.10} , we have
\begin{align*}
&\int_{\mathbb{H}^{n}}uP_{\gamma}udV- 2^{2\gamma}\frac{\Gamma(\frac{3+2\gamma}{4})^{2}}
{\Gamma(\frac{3-2\gamma}{4})^{2}}\int_{\mathbb{H}^{n}}u^{2}dV\\
=&\int^{+\infty}_{-\infty}\int_{\mathbb{S}^{n-1}}
\left(2^{2\gamma}\frac{|\Gamma(\frac{3+2\gamma}{4}+\frac{i}{2}\lambda)|^{2}}
{|\Gamma(\frac{3-2\gamma}{4}+\frac{i}{2}\lambda)|^{2}}-2^{2\gamma}\frac{|\Gamma(\frac{3+2\gamma}{4})|^{2}}
{|\Gamma(\frac{3-2\gamma}{4})|^{2}}\right)|\widehat{u}(\lambda,\zeta)|^{2}|\mathfrak{c}(\lambda)|^{-2}d\lambda d\sigma(\zeta)\\
\geq&\int^{+\infty}_{-\infty}\int_{\mathbb{S}^{n-1}}
\frac{|\Gamma(\gamma+i\lambda))|^{2}}
{|\Gamma(i\lambda)|^{2}}|\widehat{u}(\lambda,\zeta)|^{2}|\mathfrak{c}(\lambda)|^{-2}d\lambda d\sigma(\zeta)\\
=&\int_{\mathbb{B}^{n}}u\frac{|\Gamma(\gamma+i\sqrt{-\Delta_{\mathbb{H}}-\frac{(n-1)^{2}}{4}})|^{2}}
{|\Gamma(i\sqrt{-\Delta_{\mathbb{H}}-\frac{(n-1)^{2}}{4}})|^{2}}u dV\\
\geq& S_{n,\gamma}\left(\int_{\mathbb{B}^{n}}|u|^{\frac{2n}{n-2\gamma}}dV\right)^{\frac{n-2\gamma}{n}}.
\end{align*}

 Case 2: $n=4k+3, 4k+4$. In this case, we have $2k+1\leq \gamma\leq 2k+2$ and hence $\sin\gamma\pi\leq0$.
Therefore, by Theorem  \ref{th1.4} and (\ref{2.19}), we have
\begin{align*}
&\int_{\mathbb{H}^{n}}uP_{\gamma}udV- 2^{2\gamma}\frac{\Gamma(\frac{3+2\gamma}{4})^{2}}
{\Gamma(\frac{3-2\gamma}{4})^{2}}\int_{\mathbb{H}^{n}}u^{2}dV\\
=&\int_{\mathbb{H}^{n}}u\widetilde{P}_{\gamma}udV- \frac{\Gamma(\gamma+\frac{1}{2})^{2}}
{\Gamma(\frac{1}{2})^{2}}\int_{\mathbb{H}^{n}}u^{2}dV+\\
&\frac{\sin\gamma\pi}{\pi}
\int_{\mathbb{H}^{n}}u\left(\left|\Gamma\left(\gamma+\frac{1}{2}+i\sqrt{-\Delta_{\mathbb{H}}-\frac{(n-1)^{2}}{4}}\right)\right|^{2}-
\left|\Gamma\left(\gamma+\frac{1}{2}\right)\right|^{2}\right)
udV\\
=&\int_{\mathbb{H}^{n}}u\widetilde{P}_{\gamma}udV- \frac{\Gamma(\gamma+\frac{1}{2})^{2}}
{\Gamma(\frac{1}{2})^{2}}\int_{\mathbb{H}^{n}}u^{2}dV+\\
&\frac{\sin\gamma\pi}{\pi}
\int^{+\infty}_{-\infty}\int_{\mathbb{S}^{n-1}}\left(\left|\Gamma\left(\gamma+\frac{1}{2}+i\lambda\right)\right|^{2}-
\left|\Gamma\left(\gamma+\frac{1}{2}\right)\right|^{2}\right)
|\widehat{u}(\lambda,\zeta)|^{2}|\mathfrak{c}(\lambda)|^{-2}d\lambda d\sigma(\zeta))\\
\geq&\int_{\mathbb{H}^{n}}u\widetilde{P}_{\gamma}udV- \frac{\Gamma(\gamma+\frac{1}{2})^{2}}
{\Gamma(\frac{1}{2})^{2}}\int_{\mathbb{H}^{n}}u^{2}dV\;\;\;\;\left(\because \left|\Gamma\left(\gamma+\frac{1}{2}+i\lambda\right)\right|\leq\left|\Gamma\left(\gamma+\frac{1}{2}\right)\right|\right)\\
\geq& S_{n,\gamma}\left(\int_{\mathbb{B}^{n}}|u|^{\frac{2n}{n-2\gamma}}dV\right)^{\frac{n-2\gamma}{n}}.
\end{align*}
The proof of Theorem \ref{th1.4} is thereby completed.

\end{document}